\documentclass[12pt]{amsart}
\theoremstyle{plain}
\usepackage[mathscr]{eucal}
\usepackage{setspace}
\usepackage{amssymb}
\usepackage{mathrsfs}
\usepackage{bbding}
\usepackage{bbold}
\usepackage{graphicx}
\usepackage{pifont}
\usepackage[all,cmtip]{xy}
\usepackage{verbatim}
\usepackage{eucal}
\usepackage{fancyhdr}
\usepackage{bm}
\usepackage{upgreek}

\newtheorem{theorem}{Theorem}[section]
\newtheorem{lemma}[theorem]{Lemma}

\newtheorem*{claim}{Claim}
\newtheorem{corollary}[theorem]{Corollary}
\newtheorem{proposition}[theorem]{Proposition}

\newtheorem{problem}[theorem]{Problem}

\theoremstyle{definition}
\newtheorem{definition}[theorem]{Definition}
\newtheorem{example}[theorem]{Example}
\newtheorem{note}[theorem]{Note}

\newtheorem*{notation}{Notation}
\newtheorem{question}[theorem]{Question}

\theoremstyle{remark}
\newtheorem{remark}[theorem]{Remark}

\advance \topmargin by -\headheight
\advance \topmargin by -\headsep
\evensidemargin \oddsidemargin
\newcommand{\Aut}{\mbox{Aut}}

\newcommand{\csuchthat}{\, :\,}

\newcommand{\dom}{\mbox{dom}}

\newcommand{\La}{\Leftarrow}

\newcommand{\ra}{\rightarrow}
\newcommand{\Ra}{\Rightarrow}

\newcommand{\mc}[1]{\mathcal{#1}}

\newcommand{\ol}[1]{\overline{#1}}
\newcommand{\cc}{\curvearrowright}

\newcommand{\IFF}{\Leftrightarrow}

\newcommand{\U}{\EuScript{U}}

\renewcommand{\H}{\EuScript{H}}
\newcommand{\K}{\EuScript{K}}

\numberwithin{equation}{section}
\DeclareSymbolFont{AMSb}{U}{msb}{m}{n}
\DeclareMathSymbol{\N}{\mathbin}{AMSb}{"4E}
\DeclareMathSymbol{\Z}{\mathbin}{AMSb}{"5A}
\DeclareMathSymbol{\R}{\mathbin}{AMSb}{"52}
\DeclareMathSymbol{\Q}{\mathbin}{AMSb}{"51}
\DeclareMathSymbol{\I}{\mathbin}{AMSb}{"49}
\DeclareMathSymbol{\C}{\mathbin}{AMSb}{"43}
\DeclareMathSymbol{\T}{\mathbin}{AMSb}{"54}
\DeclareMathSymbol{\E}{\mathbin}{AMSb}{"45}
\DeclareMathSymbol{\F}{\mathbin}{AMSb}{"46}
\DeclareMathSymbol{\PP}{\mathbin}{AMSb}{"50}
\renewcommand{\P}{\PP}
\renewcommand{\emptyset}{\varnothing}

\begin{document}
\title{Weak equivalence and non-classifiability of measure preserving actions}
\author{Robin D. Tucker-Drob}
\date{April 6, 2012}
\begin{abstract}
Ab\'{e}rt-Weiss have shown that the Bernoulli shift $s_\Gamma$ of a countably infinite group $\Gamma$ is weakly contained in any free measure preserving action $\bm{a}$ of $\Gamma$.  Proving a conjecture of Ioana we establish a strong version of this result by showing that $\bm{s}_\Gamma\times \bm{a}$ is weakly equivalent to $\bm{a}$. Using random Bernoulli shifts introduced by Ab\'{e}rt-Glasner-Virag we generalized this to non-free actions, replacing $\bm{s}_\Gamma$ with a random Bernoulli shift associated to an invariant random subgroup, and replacing the product action with a relatively independent joining. The result for free actions is used along with the theory of Borel reducibility and Hjorth's theory of turbulence to show that the equivalence relations of isomorphism, weak isomorphism, and unitary equivalence on the weak equivalence class of a free measure preserving action do not admit classification by countable structures.  This in particular shows that there are no free \emph{weakly rigid} actions, i.e., actions whose weak equivalence class and isomorphism class coincide, answering negatively a question of Ab\'{e}rt and Elek.

We also answer a question of Kechris regarding two ergodic theoretic properties of residually finite groups. A countably infinite residually finite group $\Gamma$ is said to have property $\mbox{EMD}^*$ if the action $\bm{p}_\Gamma$ of $\Gamma$ on its profinite completion weakly contains all ergodic measure preserving actions of $\Gamma$, and $\Gamma$ is said to have property $\mbox{MD}$ if $\bm{\iota}\times \bm{p}_\Gamma$ weakly contains all measure preserving actions of $\Gamma$, where $\bm{\iota}$ denotes the identity action on a standard non-atomic probability space. Kechris shows that $\mbox{EMD}^*$ implies $\mbox{MD}$ and asks if the two properties are actually equivalent. We provide a positive answer to this question by studying the relationship between convexity and weak containment in the space of measure preserving actions.
\end{abstract}
\maketitle

\setcounter{tocdepth}{4}
\tableofcontents

\section{Introduction}

By a \emph{measure preserving action} of a countable group $\Gamma$ we mean a triple $(\Gamma , a , (X,\mu ))$, which we write as $\Gamma \cc ^ a (X,\mu )$, where $(X,\mu )$ is a standard probability space (i.e., a standard Borel space equipped with a Borel probability measure) and $a:\Gamma\times X\ra X$ is a Borel action of $\Gamma$ on $X$ that preserves the Borel probability measure $\mu$. In what follows all measures are probability measures unless explicitly stated otherwise and we will write $\bm{a}$ and $\bm{b}$ to denote the measure preserving actions $\Gamma \cc ^a (X,\mu )$ and $\Gamma \cc ^b (Y,\nu )$ respectively when the group $\Gamma$ and the underlying probability spaces $(X,\mu )$ and $(Y,\nu )$ are understood. 
Given measure preserving actions $\bm{a} = \Gamma \cc ^a (X,\mu )$ and $\bm{b} = \Gamma \cc ^b (Y,\nu )$, we say that $\bm{a}$ is {\emph{weakly contained}} in $\bm{b}$, and write $\bm{a}\prec \bm{b}$, if for every finite partition $A_0,\dots ,A_{k-1}$ of $X$ into Borel sets, every finite subset $F\subseteq \Gamma$, and every $\epsilon >0$, there exists a Borel partition $B_0,\dots ,B_{k-1}$ of $Y$ such that
\[
|\mu (\gamma ^a A_i\cap A_j ) - \nu (\gamma ^b B_i \cap B_j ) | < \epsilon
\]
for all $\gamma \in F$ and $0\leq i, j <k$.  
We write $\bm{a}\sim \bm{b}$ if both $\bm{a}\prec \bm{b}$ and $\bm{b}\prec \bm{a}$, in which case $\bm{a}$ and $\bm{b}$ are said to be \emph{weakly equivalent}. The notion of weak containment of measure preserving actions was introduced by Kechris \cite{Ke10} as an ergodic theoretic analogue of weak containment for unitary representations.

Weak containment of unitary representations may be defined as follows (see \cite[Appendix F]{BdlHV08}). Let $\pi$ and $\rho$ be unitary representations of $\Gamma$ on the Hilbert spaces $\H _\pi$ and $\H _\rho$ respectively. Then $\pi$ is \emph{weakly contained} in $\rho$, written $\pi\prec \rho$, if for every unit vector $\xi$ in $\H _\pi$, every finite subset $F\subseteq \Gamma$, and every $\epsilon >0$, there exists a finite collection $\eta _0,\dots ,\eta _{k-1}$ of unit vectors in $\H _\rho$ and nonnegative real numbers $\alpha _0,\dots ,\alpha _{k-1}$ with $\sum _{i=0}^{k-1}\alpha _i =1$ such that
\[
|\langle \pi (\gamma )\xi , \xi \rangle - \sum _{i=0}^{k-1} \alpha _i \langle \rho (\gamma )\eta _i ,\eta _i \rangle |  < \epsilon
\]
for all $\gamma \in F$. Each unit vector $\xi \in H_\pi$ gives rise to a normalized positive definite function on $\Gamma$ defined by $\gamma \mapsto \langle \pi (\gamma )\xi ,\xi \rangle$. We call such a function a \emph{normalized positive definite function realized in $\pi$} and we may rephrase the definition of $\pi\prec \rho$ accordingly as: every normalized positive definite function realized in $\pi$ is a pointwise limit of convex sums of normalized positive definite functions realized in $\rho$.

A similar rephrasing also applies to weak containment of measure preserving actions, as pointed out by Ab\'{e}rt-Weiss \cite{AW11}. If we view a finite Borel partition $A_0,\dots ,A_{k-1}$ of $X$ as a Borel function $\phi : X \ra k = \{ 0,1,\dots ,k-1 \}$ (where we view $k$ as a discrete space) then, given a measure preserving action $\bm{a} = \Gamma \cc ^a (X,\mu )$, each partition $\phi : X\ra k$ gives rise to a shift-invariant Borel probability measure $(\Phi ^{\phi ,a} )_*\mu$ on $k^\Gamma$, where
\[
\Phi ^{\phi ,a}(x)(\gamma ) = \phi ((\gamma ^{-1})^a\cdot x) .
\]
The map $\Phi ^{\phi ,a}$ is equivariant between the action $a$ and the shift action $s$ on $k^\Gamma$ given by $(\gamma ^s \cdot f)(\delta ) = f(\gamma ^{-1}\delta )$, and one may show that the measures $(\Phi ^{\phi ,a})_*\mu$, as $\phi$ ranges over all Borel partitions of $X$ into $k$-pieces, are precisely those shift-invariant Borel measures $\lambda$ such that $\Gamma \cc ^s (k^\Gamma ,\lambda )$ is a factor of $\bm{a}$. In this language $\bm{a}$ being weakly contained in $\bm{b}$ means that for every natural number $k$, each shift-invariant measure on $k^\Gamma$ that is a factor of $\bm{a}$ is a weak${}^*$-limit of shift-invariant measures that are factors of $\bm{b}$.

More precisely, given a compact Polish space $K$ we equip $K^\Gamma$ with the product topology, and we let $M_s(K^\Gamma )$ denote the convex set of shift-invariant Borel probability measures on $K^\Gamma$ equipped with the weak${}^*$-topology so that it is also a compact Polish space. 
We define
\[
E(\bm{a},K) = \{ (\Phi ^{\phi ,a})_*\mu \csuchthat \phi : X\ra K \mbox{ is Borel} \} \subseteq M_s(K^\Gamma ) .
\]
Then Ab\'{e}rt-Weiss characterize weak containment of measure preserving actions as follows: $\bm{a}\prec \bm{b}$ if and only if $E(\bm{a},K)\subseteq \ol{E(\bm{b},K)}$ for every finite $K$ if and only if $E(\bm{a},K)\subseteq \ol{E(\bm{b},K)}$ for every compact Polish space $K$.

From this point of view one difference between the two notions of weak containment is apparent. While weak containment of representations allows for normalized positive definite functions realized in $\pi$ to be approximated by \emph{convex sums} of normalized positive definite functions realized in $\rho$, weak containment of measure preserving actions asks that shift invariant factors of $\bm{a}$ be approximated by a \emph{single} shift invariant factor of $\bm{b}$ at a time.  It is natural to ask for a characterization of the situation in which shift invariant factors of $\bm{a}$ are approximated by convex sums of shift invariant factors of $\bm{b}$. When this is the case we say that $\bm{a}$ is \emph{stably weakly contained in} $\bm{b}$ and we write $\bm{a}\prec _s \bm{b}$. The relationship between weak containment and stable weak containment of measure preserving actions is analogous to the relationship between weak containment in the sense of Zimmer (see \cite[F.1.2.(ix)]{BdlHV08} and \cite[Appendix H.{\bf (B)}]{Ke10}) and weak containment of unitary representations. Our first theorem is a characterization of this stable version of weak containment of measure preserving actions.

In what follows $(X,\mu )$ and $(Y,\nu )$ and $(Z,\eta )$ always denote standard probability spaces. We let $\iota _\eta :\Gamma \times Z\ra Z$ denote the trivial (identity) action of $\Gamma$ on $(Z,\eta )$, writing $\bm{\iota} _\eta$ for the corresponding triple $\Gamma \cc ^{\iota _\eta}(Z,\eta )$, and we write $\iota$ and $\bm{\iota}$ for $\iota _{\eta}$ and $\bm{\iota}_{\eta}$ respectively when $\eta$ is non-atomic.  We show the following in \S\ref{section2}.

\begin{theorem}\label{thm:conv}
Let $\bm{b}=\Gamma \cc ^b (Y,\nu )$ be a measure preserving action of $\Gamma$. Then $\ol{E(\bm{\iota} \times \bm{b} , K )} = \ol{\mbox{co}}E(\bm{b},K)$ for every compact Polish $K$. In particular, for any $\bm{a}=\Gamma \cc ^a (X,\mu )$ we have that $\bm{a}\prec \bm{\iota}\times \bm{b}$ if and only if $E(\bm{a},K) \subseteq \ol{\mbox{co}}E(\bm{b},K)$ for every compact Polish space $K$.
\end{theorem}

When $\bm{a}$ is ergodic, so that $E(\bm{a},K )$ is contained in the extreme points of $M_s (K^\Gamma )$, we show that Theorem \ref{thm:conv} implies the following direct analogue of the fact (see \cite[F.1.4]{BdlHV08}) that if $\pi$ and $\rho$ are representations of $\Gamma$, $\pi$ is irreducible, and $\pi$ is weakly contained in $\rho$, then every normalized positive definite function realized in $\pi$ is actually a pointwise limit of normalized positive definite functions realized in $\rho$.

\begin{theorem}\label{thm:main}
Let $\bm{a}=\Gamma \cc ^a (X,\mu )$ and $\bm{b}=\Gamma \cc ^b(X,\mu )$ be measure preserving actions of $\Gamma$ and suppose that $\bm{a}$ is ergodic. If $\bm{a}\prec \bm{\iota} \times \bm{b}$ then $\bm{a}\prec \bm{b}$.
\end{theorem}

In Theorem \ref{thm:aeweakcont} we show more generally that if $\bm{a}$ is an ergodic measure preserving action that is weakly contained in $\bm{d}$, then $\bm{a}$ is weakly contained in almost every ergodic component of $\bm{d}$.  This may be seen as a weak containment analogue of the fact that if $\bm{a}$ is a factor of $\bm{d}$, then $\bm{a}$ is a factor of almost every ergodic component of $\bm{d}$ (see Proposition \ref{prop:factor} below).

One consequence of Theorem \ref{thm:main} is that every non-amenable group has a free, non-ergodic weak equivalence class, and this in fact characterizes non-amenability (Corollary \ref{cor:newchar} below).

\begin{theorem}\label{thm:nonerg}
If $\bm{b}$ a measure preserving action of $\Gamma$ that is strongly ergodic, then $\bm{\iota}\times \bm{b}$ is not weakly equivalent to any ergodic action. In particular, if $\Gamma$ is a non-amenable group and $\bm{s}_\Gamma =\Gamma \cc ^ {s_\Gamma} ([0,1]^\Gamma , \lambda ^\Gamma )$ is the Bernoulli shift action of $\Gamma$, then $\bm{\iota}\times \bm{s}_\Gamma$ is a free action of $\Gamma$ that is not weakly equivalent to any ergodic action.
\end{theorem}
%

If $\mc{B}$ is a class of measure preserving actions of a countable group $\Gamma$ and $\bm{a}\in \mc{B}$, then $\bm{a}$ is called \emph{universal for $\mc{B}$} if $\bm{b}\prec \bm{a}$ for every $\bm{b}\in \mc{B}$. When $\bm{a}$ is universal for the class of \emph{all} measure preserving actions of $\Gamma$ then $\bm{a}$ is simply called \emph{universal}. In \S\ref{section3} we study the universality properties $\mbox{EMD}$, $\mbox{EMD}^*$, and $\mbox{MD}$ of residually finite groups introduced by Kechris \cite{Ke11} (MD was also independently studied by Bowen \cite{Bo03}, but with different terminology), and defined 
as follows. Let $\Gamma$ be a countably infinite group. $\Gamma$ is said to have property $\mbox{EMD}$ if the measure preserving action $\bm{p}_\Gamma$ of $\Gamma$ on its profinite completion is universal. $\Gamma$ is said to have property $\mbox{EMD}^*$ if $\bm{p}_\Gamma$ is universal for the class of all ergodic measure preserving actions of $\Gamma$. $\Gamma$ is said to have property $\mbox{MD}$ if $\bm{\iota}\times \bm{p}_\Gamma$ is universal.

Each of these properties imply that $\Gamma$ is residually finite and it is clear that $\mbox{EMD}$ implies both $\mbox{EMD}^*$ and $\mbox{MD}$. Kechris shows that $\mbox{EMD}^*$ implies $\mbox{MD}$ and asks (Question 4.11 of \cite{Ke11}) whether the converse is true. We provide a positive answer to this question.

\begin{theorem}\label{thm:MD}
The properties $\mbox{MD}$ and $\mbox{EMD}^*$ are equivalent.
\end{theorem}

This implies (Corollary \ref{cor:EMD(T)} below) that the properties $\mbox{EMD}$ and $\mbox{MD}$ are equivalent for all groups without property (T). We also show in Theorem \ref{thm:subgrpEMD} that the free product of groups with property MD has EMD and we give two reformulations of the problem of whether $\mbox{EMD}$ and $\mbox{MD}$ are equivalent in general (Theorem \ref{thm:MDEMD} below).

In \S\ref{section4} we discuss the structure of weak equivalence with respect to invariant random subgroups.  A countable group $\Gamma$ acts on the compact space $\mbox{Sub}(\Gamma )\subseteq 2^\Gamma$ of all of its subgroups by conjugation. Following \cite{AGV11}, a conjugation-invariant Borel probability measure on $\mbox{Sub}(\Gamma )$ will be called an \emph{invariant random subgroup} (IRS) of $\Gamma$. We let $\mbox{IRS}(\Gamma )$ denote the set of all invariant random subgroups of $\Gamma$. If $\bm{a}=\Gamma \cc ^a (Y, \nu )$ is a measure preserving action of $\Gamma$ then the stabilizer map $y\mapsto \mbox{stab}_a(y) \in \mbox{Sub}(\Gamma )$ 
is equivariant so that the measure $(\mbox{stab}_a)_*\nu$ is an IRS of $\Gamma$ which we call the \emph{type} of $\bm{a}$, and denote $\mbox{type}(\bm{a})$.  It is shown in \cite{AE11a} that the type of a measure preserving action is an invariant of weak equivalence (we give a proof of this in \ref{thm:contin} below).

In \S\ref{subsection:compspace} we use the framework laid out in \S\ref{section2} to study the compact metric topology introduced by Ab\'{e}rt-Elek \cite{AE11a} on the set $A_\sim (\Gamma ,X,\mu )$ of weak equivalence classes of measure preserving actions of $\Gamma$. We show that the map $A_\sim (\Gamma ,X,\mu )\ra \mbox{IRS}(\Gamma )$ sending each weak equivalence class to its type in $\mbox{IRS}(\Gamma )$ is continuous when $\mbox{IRS}(\Gamma )$ is equipped with the weak${}^*$ topology.

In \S\ref{subsection:randbern} we detail a construction, described in \cite{AGV11}, whereby, given a probability space $(Z,\eta )$, one canonically associates to each $\theta \in \mbox{IRS}(\Gamma )$ a measure preserving action $\bm{s}_{\theta ,\eta}$ of $\Gamma$ such that $\mbox{type}(\bm{s}_{\theta ,\eta}) =\theta$ when $\eta$ is non-atomic. We call $\bm{s}_{\theta ,\eta}$ the \emph{$\theta$-random Bernoulli shift of $\Gamma$ over $(Z,\eta )$}. When $\bm{a}$ is free then $\mbox{type}(\bm{a})$ is the point mass $\updelta _{\langle e\rangle}$ on the trivial subgroup $\langle e \rangle$ of $\Gamma$ and $\bm{s}_{\updelta _{\langle e\rangle} ,\eta}$ is the usual Bernoulli shift action of $\Gamma$ on $(Z^\Gamma ,\eta ^\Gamma )$.  
After establishing some properties of random Bernoulli shifts we show the following in \S\ref{subsection:randshift}.

\begin{theorem}\label{thm:randshift}
Let $\bm{a} = \Gamma \cc ^a (Y,\nu )$ be a non-atomic measure preserving action of type $\theta$, and let $\bm{s}_{\theta ,\eta}$ be the $\theta$-random Bernoulli shift over  $(Z,\eta )$. Then the relatively independent joining of $\bm{s}_{\theta ,\eta}$ and $\bm{a}$ over their common factor $\Gamma \cc (\mbox{\emph{Sub}}(\Gamma ),\theta )$ is weakly equivalent to $\bm{a}$. In particular, $\bm{s}_{\theta ,\eta}$ is weakly contained in every non-atomic action of type $\theta$.
\end{theorem}

When $\bm{a}$ is free 
then the relatively independent joining of $\bm{s}_{\updelta _{\langle e\rangle},\eta}$ and $\bm{a}$ is simply the product of the Bernoulli shift with $\bm{a}$ and Theorem \ref{thm:randshift} proves a conjecture of Ioana, becoming the following strengthening of Ab\'{e}rt-Weiss \cite[Theorem 1]{AW11}:

\begin{corollary}\label{thm:shift}
Let $\bm{s}_\Gamma = \Gamma \cc ^ {s_\Gamma} ( [0,1]^\Gamma ,\lambda ^\Gamma )$ be the Bernoulli shift action of $\Gamma$, where $\lambda$ denotes Lebesgue measure on $[0,1]$. Let $\bm{a}=\Gamma \cc ^a (X,\mu )$ be a free measure preserving action of $\Gamma$ on a non-atomic standard probability space $(X,\mu )$. Then $\bm{s}_\Gamma \times \bm{a}$ is weakly equivalent to $\bm{a}$.
\end{corollary}

Several invariants of measure preserving actions such as groupoid cost \cite{AW11} (\cite{Ke10} for the case of free actions) and independence number \cite{CK10} are known to increase or decrease with weak containment (see also \cite{AE11a} and \cite{CKT-D11} for other examples). A consequence of Theorem \ref{thm:randshift} is that, for a finitely generated group $\Gamma$, among all non-atomic measure preserving actions of type $\theta$, the groupoid cost attains its maximum and the independence number attains its minimum on $\bm{s}_{\theta ,\lambda}$. Likewise, Corollary \ref{thm:shift} implies that for any free measure preserving action $\bm{a}$ of $\Gamma$, both $\bm{a}$ and $\bm{s}_\Gamma\times\bm{a}$ have the same independence number, and the orbit equivalence relation associated to $\bm{a}$ and $\bm{s}_\Gamma \times \bm{a}$ have the same cost.

In \S\ref{section5} we address the question of how many isomorphism classes of actions are contained in a given weak equivalence class. We answer a question of Ab\'{e}rt-Elek \cite[Question 6.1]{AE11a}, showing that the weak equivalence class of any free action always contains non-isomorphic actions. Our arguments show that there are in fact continuum many isomorphism classes of actions in any free weak equivalence class, and from the perspective of Borel reducibility we can strengthen this even further. Let $A(\Gamma ,X,\mu )$ denote the Polish space of measure preserving actions of $\Gamma$ on $(X,\mu )$ and let $\bm{a}, \bm{b} \in A(\Gamma ,X,\mu )$. Then $\bm{a}$ and $\bm{b}$ are called \emph{weakly isomorphic}, written $\bm{a}\cong ^w \bm{b}$, if both $\bm{a}\sqsubseteq \bm{b}$ and $\bm{b}\sqsubseteq \bm{a}$. We call $\bm{a}$ and $\bm{b}$ \emph{unitarily equivalent}, written $\bm{a}\cong ^\U \bm{b}$, if the corresponding Koopman representations $\kappa ^{\bm{a}}_0$ and $\kappa ^{\bm{b}}_0$ are unitarily equivalent. We let $\cong$ denote isomorphism of actions. Then $\bm{a}\cong \bm{b} \, \Ra \, \bm{a}\cong ^w \bm{b} \, \Ra \, \bm{a}\cong ^\U \bm{b}$. We now have the following.

\begin{theorem}\label{thm:classif}
Let $\bm{a}=\Gamma \cc ^a (X,\mu )$ be a free action of a countably infinite group $\Gamma$ and let $[\bm{a}] = \{ \bm{b}\in A(\Gamma ,X,\mu ) \csuchthat \bm{b}\sim \bm{a} \}$ be the weak equivalence class of $\bm{a}$.  Then isomorphism on $[\bm{a}]$ does not admit classification by countable structures. The same holds for both weak isomorphism and unitary equivalence on $[\bm{a}]$.
\end{theorem}

Any two free actions of an infinite amenable group are weakly equivalent (\cite{FW04}, see also Remark \ref{rem:amen1} and Theorem \ref{thm:typeamen} below), so for amenable $\Gamma$ Theorem \ref{thm:classif} follows from \cite{FW04}, \cite{Hj97} and \cite[13.7, 13.8,  13.9]{Ke10} (see also \cite[4.4]{KLP10}), while for non-amenable $\Gamma$ there are continuum-many weak equivalence classes of free actions (see Remark \ref{rem:continuum} below), and Theorem \ref{thm:classif} is therefore a refinement of the existing results. The proof of \ref{thm:classif} uses the methods of \cite[13.7]{Ke10} and \cite{KLP10}. We fix an infinite dimensional separable Hilbert space $\H$, and denote by $\mbox{Rep}_\lambda (\Gamma ,\H )$ the Polish space of unitary representations of $\Gamma$ on $\H$ that are weakly contained in the left regular representation $\lambda _\Gamma$ of $\Gamma$.  The conjugacy action of the unitary group $\U (\H )$ on $\mbox{Rep}_\lambda (\Gamma ,\H )$ is generically turbulent by \cite[3.3]{KLP10}, so Theorem \ref{thm:classif} will follow by showing that unitary conjugacy on $\mbox{Rep}_\lambda (\Gamma ,\H )$ is not generically $\cong \! |[\bm{a}]$-ergodic (and that the same holds for $\cong ^w$ and $\cong ^\U$ in place of $\cong$). For this we find a continuous homomorphism $\psi$ from unitary conjugacy on $\mbox{Rep}_\lambda (\Gamma ,\H )$ to isomorphism on $[\bm{a}]$ 
 with the property that the inverse image of each $\cong ^\U$-class is meager. 
The main new ingredient that is needed in the proof of Theorem \ref{thm:classif} is Corollary \ref{thm:shift}, which shows that the homomorphism $\psi$ we define takes values in $[\bm{a}]$.

In \S\ref{sec:amenable} we show that when $\Gamma$ is amenable, $\mbox{type}(\bm{a})$ completely determines the stable weak equivalence class (definition \ref{def:swc}) of a measure preserving action $\bm{a}$ of $\Gamma$.

\begin{theorem}\label{thm:typeamen} Let $\bm{a}$ and $\bm{b}$ be two measure preserving actions of an amenable group $\Gamma$. Then
\begin{enumerate}
\item $\mbox{type}(\bm{a})=\mbox{type}(\bm{b})$ if and only if $\bm{a}\sim _s \bm{b}$.
\item Suppose that $\mbox{type}(\bm{a}) = \mbox{type}(\bm{b})$ concentrates on the infinite index subgroups of $\Gamma$. Then $\bm{a}\sim \bm{b}$.
\end{enumerate}
\end{theorem}

Combining this with the results of \S\ref{subsection:compspace} (in particular, Remark \ref{rem:stableconv}) shows that when $\Gamma$ is amenable, the type map $[\bm{a}]_s\mapsto \mbox{type}(\bm{a})$, from the compact space $A_{\sim _s}(\Gamma ,Y,\nu )$ of all stable weak equivalence classes of measure preserving actions of $\Gamma$, to the space $\mbox{IRS}(\Gamma )$, is a homeomorphism.

We end with two appendices, one on ultraproducts of measure preserving actions, and one on stable weak containment.

\begin{remark}
After sending G\'{a}bor Elek a preliminary version of this paper, I was informed by him that he has independently obtained a version of Theorem \ref{thm:typeamen}. See \cite{El12}. 
\end{remark}

\medskip

{\bf Acknowledgments}. I would like to thank my advisor Alexander Kechris for his encouragement and support and for many valuable discussions related to this paper.  I would also like to thank Miklos Ab\'{e}rt for very useful feedback and discussions. The research of the author was partially supported by NSF Grant DMS-0968710.

\section{Preliminaries and notation}\label{sectionP}

$\Gamma$ will always denote a countable group, and $e$ will always denote the identity element of $\Gamma$.

\subsection{Measure algebras and standard probability spaces} All measures will be probability measures unless explicitly stated otherwise. A \emph{standard probability space} is a probability measure space $(X, \mu ) = (X, \bm{B}(X), \mu )$ where $X$ is a standard Borel space and $\mu$ is a probability measure on the $\sigma$-algebra $\bm{B}(X)$ of Borel subsets of $X$. In what follows, $(X,\mu )$, $(Y,\nu )$, and $(Z, \eta )$ will always denote standard probability spaces. Though we mainly focus on standard probability spaces we will make use of nonstandard probability spaces arising as ultraproducts of standard probability spaces. We will write $(W,\rho )$ for a probability space that may or may not be standard.

The \emph{measure algebra} $\mbox{MALG}_\rho$ of a probability space $(W, \rho )$ is the $\sigma$-algebra of $\rho$-measurable sets modulo the $\sigma$-ideal of null sets, equipped with the measure $\rho$. We also equip $\mbox{MALG}_\rho$ with the metric $d_\rho (A, B) = \rho (A\Delta B)$. We will sometimes abuse notation and identify a measurable set $A\subseteq W$ with its equivalence class in $\mbox{MALG}_\rho$ when there is no danger of confusion.

\subsection{Measure preserving actions} Let $\Gamma$ be a countable group. A \emph{measure preserving action} of $\Gamma$ 
is a triple $(\Gamma , a , (X,\mu ))$, which we write as $\Gamma \cc ^ a (X,\mu )$, where $(X,\mu )$ is a standard probability space and $a:\Gamma\times X\ra X$ is a Borel action of $\Gamma$ on $X$ that preserves the probability measure $\mu$. A measure preserving action $\Gamma \cc ^a (X ,\mu )$ will often also be denoted by a boldface letter such as $\bm{a}$ or $\bm{\mu}$ depending on whether we want to emphasize the underlying action or the underlying probability measure. 
When $\gamma \in \Gamma$ and $x\in X$ we write $\gamma ^a \cdot x$ or $\gamma ^ax$ for $a(\gamma ,x)$. In what follows, $\bm{a}$, $\bm{b}$, and $\bm{c}$ and $\bm{d}$ will always denote measure preserving actions of $\Gamma$.

We will also make use of actions of $\Gamma$ on nonstandard probability spaces. When $(W,\rho )$ is a probability space and $o:\Gamma \times W\ra W$ is a measurable action of $\Gamma$ on $W$ that preserves $\rho$ then we will still use the notations $\bm{o} = \Gamma \cc ^o (W,\rho )$, $\gamma ^o$, etc., from above, \emph{though we reserve the phrase ``{measure preserving action}" for the case when the underlying probability space is standard}.
%
%
%
%

\subsection{The space of measure preserving actions} We let $A(\Gamma ,X,\mu )$ denote the set of all measure preserving actions of $\Gamma$ on $(X,\mu )$ modulo almost everywhere equality. That is, two measure preserving actions $\bm{a}$ and $\bm{b}$ of $\Gamma$ on $(X,\mu )$ are equivalent if $\mu (\{ x\in X\csuchthat \gamma ^ax \neq \gamma ^b x \} ) =0$ for all $\gamma \in \Gamma$. Though elements of $A(\Gamma ,X,\mu )$ are equivalence classes of measure preserving actions we will abuse notation and confuse elements of $A(\Gamma ,X,\mu )$ with their Borel representatives, making sure our statements and definitions are independent of the choice of representative when it is not obvious. We equip $A(\Gamma ,X,\mu )$ with the \emph{weak topology}, which is a Polish topology generated by the maps $\bm{a} \mapsto \gamma ^a A  \in \mbox{MALG}_\mu$, with $A$ ranging over $\mbox{MALG}_\mu$ and $\gamma$ ranging over elements of $\Gamma$.

\medskip

{\bf Notation.} For $\bm{a} \in A(\Gamma ,X,\mu )$ and $\bm{b}\in A(\Gamma ,Y,\nu )$ we let $\bm{a}\sqsubseteq \bm{b}$ denote that $\bm{a}$ is a factor of $\bm{b}$ and we let $\bm{a}\cong \bm{b}$ denote that $\bm{a}$ and $\bm{b}$ are isomorphic. We let $\bm{\iota}_\eta \in A(\Gamma ,Z,\eta )$ denote the trivial (identity) system $\Gamma \cc ^{\iota _\eta} (Z,\eta )$, and we write $\bm{\iota}$ for $\bm{\iota}_\eta$ when $\eta$ is non-atomic. We call $\Gamma \cc ^a (X,\mu )$ \emph{non-atomic} if the probability space $(X,\mu )$ is non-atomic. If $T:X\ra X$ then we let $\mbox{supp}(T) = \{ x\in X\csuchthat T(x)\neq x \}$. For a $A\subseteq X$ we denote by $\mu |A$ the restriction of $\mu$ to $A$ given by $(\mu |A )(B)= \mu (B\cap A)$ and we denote by $\mu _A$ the conditional probability measure $\mu _A (B) =\frac{\mu (B\cap A)}{\mu (A)}$ where we use the convention that $\mu _A \equiv 0$ when $A\subseteq X$ is null.

\medskip

{\bf Convention.} We will regularly neglect null sets when there is no danger of confusion.

\section{Proofs of Theorems \ref{thm:conv} and \ref{thm:main}}\label{section2}

\subsection{Weak containment and shift-invariant factors} Let $K$ be a compact Polish space and equip $K^\Gamma$ with the product topology so that it is also a compact Polish space. Then $\Gamma$ acts continuously on $K^\Gamma$ by the shift action $s$, given by $(\delta ^s f)(\gamma ) = f(\delta ^{-1}\gamma )$ for $\delta ,\gamma\in \Gamma$, $f\in K^\Gamma$. Let $(W,\rho )$ be a probability space and let $\bm{o} =\Gamma \cc ^o (W,\rho )$ be a measurable action of $\Gamma$ on $W$ that preserves $\rho$. For each measurable function $\phi :W \ra K$ we define $\Phi ^{\phi ,o}: W\ra K^\Gamma$ by $\Phi ^{\phi ,o} (w)(\gamma ) = \phi ((\gamma ^{-1})^o\cdot w)$, and we let
\[
E(\bm{o},K) = \{ (\Phi ^{\phi ,o})  _*\rho \csuchthat \phi :W\ra K\mbox{ is } \rho\mbox{{}-measurable}\} .
\]
Each map $\Phi ^{\phi ,o}$ is a factor map from $\bm{o}$ to $\Gamma \cc ^s (K^\Gamma ,(\Phi ^{a,\phi })_*\mu )$ since
\begin{align*}
\Phi ^{\phi ,o} (\delta ^o\cdot w)(\gamma ) &= \phi ((\gamma ^{-1}\delta )^o\cdot w) = \phi (((\delta ^{-1}\gamma )^{-1})^o\cdot w ) = \Phi ^{\phi ,o} (w)(\delta ^{-1}\gamma ) = (\delta ^s \cdot \Phi ^{\phi ,o}(w))(\gamma ) .
\end{align*}
Conversely, given any measurable factor map $\psi : \Gamma \cc ^o (W,\rho ) \ra \Gamma \cc ^s (K^\Gamma , \pi _*\mu )$ the map $\phi (w) = \psi (w)(e)$ is also measurable, and for almost all $w\in W$ and all $\gamma \in \Gamma$ we have $\Phi ^{\phi ,o} (w)(\gamma ^{-1} ) = \phi (\gamma ^a\cdot w) = \psi (\gamma ^o \cdot w)(e) = (\gamma ^s \cdot \psi (w))(e) = \psi (w) (\gamma ^{-1} )$ so that $\psi _*\rho = (\Phi ^{\phi ,o })_*\rho$.  It follows that $E(\bm{o},K)$ is the set of all shift-invariant Borel probability measures on $K^\Gamma$ that are factors of $\bm{o}$. We let $M_s(K^\Gamma )$ denote the convex set of all shift-invariant Borel probability measures on $K^\Gamma$. Equipped with the weak${}^*$ topology this is a compact metrizable subset of $C(K^\Gamma )^*$. If $E\subseteq M_s(K ^\Gamma )$ we let $\mbox{co}E$ denote the convex hull of $E$ and we let $\ol{\mbox{co}}E$ denote the closed convex hull of $E$. For $\gamma \in \Gamma$ we let $\pi _\gamma : K^\Gamma \ra K$ denote the projection map $\pi _\gamma (f)=f(\gamma )$. 

\begin{lemma}\label{lem:convermeas}
Suppose that $\phi _n : W\ra K$, $n\in \N$, is a sequence of measurable functions that converge in measure to the measurable function $\phi : W\ra K$ . Then $(\Phi ^{\phi _n ,o})_*\rho \ra (\Phi ^{\phi ,o})_*\rho$ in $M_s(K^\Gamma )$.
\end{lemma}

\begin{proof}
$\phi _n$ converges to $\phi$ in measure if and only if for every subsequence $\{ n_i \}$ there is a further subsequence $\{ m_{i}\}$ such that $\phi _{m_i} \ra \phi$ almost surely. If $\phi _{m_i} \ra \phi$ almost surely then for all $\gamma \in \Gamma$, $\Phi ^{\phi _{m_i},o}(w)(\gamma )\ra \Phi ^{\phi ,o}(w)(\gamma )$ almost surely, and so $\Phi ^{\phi _{m_i},o}(w)\ra \Phi ^{\phi ,o}(w)$ almost surely. It follows that $\Phi ^{\phi _n,o}\ra \Phi ^{\phi ,o}$ in measure. Since convergence in measure implies convergence in distribution it follows that $(\Phi ^{\phi _n,o})_*\rho \ra (\Phi ^{\phi ,o})_*\rho$ in $M_s(K^\Gamma )$.
\end{proof}

\begin{remark}\label{rem:RVspace}
We may form the space $L(W,\rho , K)$ of all measurable maps $\phi : W\ra K$, where we identify two such maps if they agree $\rho$-almost everywhere. If $d\leq 1$ is a compatible metric for $K$ then we equip $L(W,\rho ,K)$ with the metric $\tilde{d} (\phi ,\psi ) = \int _W d(\phi (w),\psi (w)) \, d\rho (w)$, and then $\phi _n\ra \phi$ in this topology if and only if $\phi _n$ converges to $\phi$ in measure. Then Lemma \ref{lem:convermeas} says that for each measure preserving action $\Gamma \cc ^o (W,\rho )$, the map $\phi \mapsto (\Phi ^{\phi ,o})_*\rho$ from $L(W,\rho ,K)$ to $M_s(K^\Gamma )$ is continuous. The metric $\tilde{d}$ is complete, and $\tilde{d}$ is separable when $(W,\rho )$ is standard. We note for later use that the set of all $\phi \in L(W,\rho ,K)$ with finite range is dense in $L(W,\rho ,K)$ (this follows from $d$ being separable). Proofs of these facts may be found in \cite[Section 19]{Ke10} and \cite{Mo76} (these references assume that the space $(W,\rho )$ is standard, but this assumption is not used to prove the facts mentioned here).
\end{remark}

We will find the following generalization of weak containment useful.

\begin{definition}
Let $\mc{A}$ and $\mc{B}$ be two sets of measure preserving actions of $\Gamma$. We say that $\mc{A}$ is \emph{weakly contained in} $\mc{B}$, written $\mc{A}\prec \mc{B}$, if for every $\Gamma \cc ^a (X,\mu ) = \bm{a} \in \mc{A}$, for any Borel partition $A_0,\dots ,A_{k-1}$ of $X$, $F\subseteq \Gamma$ finite, and $\epsilon >0$, there exists $\Gamma \cc ^b (Y,\nu ) = \bm{b}\in \mc{B}$ and a Borel partition $B_0,\dots ,B_{k-1}$ of $Y$ such that
\[
|\mu (\gamma ^aA_i\cap A_j) - \nu (\gamma ^bB_i\cap B_j) |<\epsilon
\]
for all $i,j< k$ and $\gamma \in F$.
\end{definition}

This is a generalization of weak containment in the sense that when $\mc{A} = \{ \bm{a}\} $ and $\mc{B} = \{ \bm{b} \}$ are both singletons then $\mc{A}\prec \mc{B}$ if and only if $\bm{a}\prec \bm{b}$ in the original sense defined in the introduction. We write $\bm{a} \prec \mc{B}$ for $\{ \bm{a} \} \prec \mc{B}$, and $\mc{A}\prec \bm{b}$ for $\mc{A}\prec \{ \bm{b} \}$. If both $\mc{A}\prec \mc{B}$ and $\mc{B}\prec \mc{A}$ then we put $\mc{A}\sim \mc{B}$. It is clear that $\prec$ is a reflexive and transitive relation on sets of actions. The arguments in 10.1 of \cite{Ke10} show the following.

\begin{proposition}\label{lem:10.1}$\ $ Let $\mc{A}$ and $\mc{B}$ be sets of non-atomic measure preserving actions of $\Gamma$.
Then $\mc{A}\prec \mc{B}$ if and only if for every $\Gamma \cc ^a (X,\mu ) = \bm{a}\in \mc{A}$, there exists a sequence $\bm{a}_n \in A(\Gamma ,X,\mu )$, $n\in \N$, converging to $\bm{a}$ such that each $\bm{a}_n$ is isomorphic to some $\bm{b}_n \in \mc{B}$. In particular, $\bm{a}\prec \mc{B}$ if and only if $\bm{a}\in \ol{\{ \bm{d}\in A(\Gamma ,X,\mu ) \csuchthat \exists\bm{b}\in \mc{B} \ \bm{d}\cong \bm{b} \}}$.
\end{proposition}

We also have the corresponding generalization of \cite[Lemma 8]{AW11}.

\begin{proposition}\label{prop:AWgen}
Let $\mc{A}$ and $\mc{B}$ be sets of measure preserving actions of $\Gamma$. Then the following are equivalent
\begin{enumerate}
\item $\mc{A}$ is weakly contained in $\mc{B}$
\item $\bigcup _{\bm{d}\in \mc{A}} E(\bm{d} , K) \subseteq \ol{\bigcup _{\bm{b}\in \mc{B}}E(\bm{b},K)}$ for every finite $K$.
\item $\bigcup _{\bm{d}\in \mc{A}} E(\bm{d} , K) \subseteq \ol{\bigcup _{\bm{b}\in \mc{B}}E(\bm{b},K)}$ for every compact Polish $K$.
\item $\bigcup _{\bm{d}\in \mc{A}} E(\bm{d} , 2^\N ) \subseteq \ol{\bigcup _{\bm{b}\in \mc{B}}E(\bm{b},2^\N )}$.
\end{enumerate}
\end{proposition}

\begin{proof}
It suffices to show this for the case $\mc{A}= \{ \bm{d} \}$ is a singleton. We let $(X,\mu )$ be the space of $\bm{d}$.

We begin with the implication (1)$\Ra$(2). It suffices to show (2) for the case $K=k=\{ 0,1,\dots ,k-1\}$ for some $k\in \N$. Fix a Borel function $\phi : X\ra k$, let $\lambda = (\Phi ^{\phi , d})_*\mu$, and let $A_i = \phi ^{-1}(\{ i \} )$ for $i<k$. 
Fix an exhaustive sequence $e\in F_0\subseteq F_1\subseteq \cdots$ of finite subsets of $\Gamma$.
For each finite $F\subseteq \Gamma$ and function $\tau :F \ra k$ let $A_\tau = \bigcap _{\gamma \in F} \gamma ^d A_{\tau (\gamma )}$. As $\bm{d}\prec \mc{B}$ we may find for each $n\in \N$ a measure preserving action $\bm{b}_n=\Gamma \cc ^{b_n}(Y_n,\nu _n)$ in $\mc{B}$ along with Borel partitions $\{ B_\tau ^n \} _{\tau \in k^{F_n}}$ of $Y_n$ such that
\begin{equation}\label{eqn:wc}
|\mu (\gamma ^a A_{\tau _1}\cap A_{\tau _2})- \nu _n (\gamma ^{b_n} B^n_{\tau _1}\cap B^n_{\tau _2}) |<\epsilon _n
\end{equation}
for all $\tau _1, \tau _2\in k^{F_n}$, and where $\epsilon _n$ is small depending on $n$, $k$, and $|F_n|$. 
Define $\psi _n :Y_n \ra k$ by $\psi _n (y) = i$ if $y\in B^n_{\tau}$ for some $\tau\in k^{F_n}$ with $\tau (e) = i$, and let $\lambda _n = (\Phi ^{\psi _n,b_n})_*\nu _n$. To show that $\lambda _n \ra \lambda$ it suffices to show that $\lambda _n (A) \ra \lambda (A)$ for every basic clopen set $A\subseteq k^\Gamma$ of the form $A= \bigcap _{\gamma \in F}\pi _\gamma ^{-1}(\{ i_\gamma \} )$, where $e\in F\subseteq \Gamma$ is finite 
and $i_\gamma < k$ for each $\gamma \in F$.  We let $\upupsilon\in k^F$ be the function $\upupsilon (\gamma ) = i_\gamma$.

For $i<k$ let $B^n_i = \bigsqcup \{ B_{\tau}\csuchthat \tau \in k^{F_n} \mbox{ and } \tau (e) = i \}$. Let $n_0$ be so large that $F^2\subseteq F_{n_0}$ and for all $n>n_0$ and each $\sigma \in k^J$, $J\subseteq F_n$, let $B^n_\sigma = \bigsqcup \{ B_\tau \csuchthat \tau \in k^{F_n} \mbox{ and } \sigma \sqsubseteq \tau \}$ and let $\tilde{B}^n_{\sigma} = \bigcap _{\gamma \in J}\gamma ^d B^n_{\sigma (\gamma )}$. Then $B^n_i = \bigsqcup \{ B^n_\sigma \csuchthat \sigma \in k^F\mbox{ and }\sigma (e) =i \}$. For $\gamma \in \Gamma$, $J\subseteq \Gamma$ and $\sigma \in k^J$ let $\gamma \cdot \sigma \in k^{\gamma J}$ be given by  $(\gamma\cdot \sigma )(\delta ) = \sigma (\gamma ^{-1}\delta )$ for all $\delta \in \gamma J$. For $\sigma \in k^F$ and $\gamma \in F$ we have  $|\nu _n (\gamma ^{b_n}B^n_\sigma \cap B^n_{\gamma \cdot\sigma} ) - \mu (\gamma ^dA_{\sigma} \cap A_{\gamma\cdot \sigma} )| \leq \sum _{\{ \tau \in k^{F_n}\csuchthat \sigma\sqsubseteq \tau \}}\sum _{\{ \tau '\in k^{F_n}\csuchthat \gamma\cdot \sigma\sqsubseteq \tau' \} } |\nu _n(\gamma ^{b_n}B^n_\tau\cap B^n_{\tau '}) - \mu (\gamma ^dA_\tau \cap A_{\tau '})| \leq \epsilon _nk^{2|F_n|}$.
Similarly, $|\nu _n (B^n_\sigma ) - \mu (A_\sigma )|<\epsilon _n k^{2|F_n|}$ and $|\nu _n (B^n_{\gamma \cdot\sigma}) - \mu (A_{\gamma \cdot\sigma})|<\epsilon _nk^{2|F_n|}$. Since $\gamma ^d A_\sigma = A_{\gamma\cdot\sigma}$ we obtain from this the estimate
\begin{equation}\label{eqn:est}
d_{\nu _n}(\gamma ^{d_n}(B^n_\sigma ) , \, B^n_{\gamma\cdot \sigma})= \nu _n (B^n_\sigma ) + \nu _n(B^n_{\gamma\cdot \sigma}) - 2\nu _n(\gamma ^{d_n}(B^n_\sigma )\cap B^n_{\gamma\cdot \sigma}) < 3\epsilon _nk^{2|F_n|}.
\end{equation}
Since $\{ B^n _\tau \} _{\tau \in k^{F_n}}$ is a partition of $Y_n$ and $F^2\subseteq F_n$ we have the set identities
\begin{align*}
B^n_{\upupsilon} = \bigsqcup _{\substack{\tau \in k^{F_n} \\ \upupsilon \sqsubseteq \tau}} B^n_\tau
&= \bigcap _{\gamma \in F}\bigsqcup _{\substack{\sigma \in k^{\gamma F} \\ \sigma (\gamma ) =\upupsilon (\gamma )}} B^n _\sigma
= \bigcap _{\gamma \in F} \bigsqcup _{\substack{\sigma \in k^{F} \\ \sigma (e) =\upupsilon (\gamma )}} B^n _{\gamma \cdot \sigma }.
\end{align*}
By (\ref{eqn:est}) the $d_{\nu _n}$-distance of this is no more than $3|F|\epsilon _nk^{3|F_n|}$ from the set
\begin{align*}
\bigcap _{\gamma \in F} \bigsqcup _{\substack{\sigma \in k^{F} \\ \sigma (e) =\upupsilon (\gamma )}} \gamma ^{d_n}B^n _{\sigma } &= \bigcap _{\gamma \in F} \gamma ^{d_n} B^n_{\upupsilon (\gamma )} = \tilde{B}^n_{\upupsilon} .
\end{align*}
Thus $|\lambda _n(A) - \lambda (A)| = |\nu _n ( \tilde{B}^n_{\upupsilon} ) -\mu ( A_{\upupsilon} ) |
\leq 3|F|\epsilon _nk^{3|F_n|} + |\nu _n (B^n_{\upupsilon}) - \mu (A_{\upupsilon})| < 3|F|\epsilon _nk^{3|F_n|} + \epsilon _n k^{2|F_n|} \ra 0$ by our choice of $\epsilon _n$.

(2)$\Ra$(3): Let $K$ be a compact Polish space. It follows from Lemma \ref{lem:convermeas} and Remark \ref{rem:RVspace} that the set $E_f(\bm{d},K)$ of all measures $\lambda \in E(\bm{d},K)$ coming from Borel $\phi : X\ra K$ with finite range is dense in $E(\bm{d},K)$. By (2) we then have $E_f(\bm{d},K)\subseteq \ol{\bigcup _{\bm{b}\in\mc{B}}E_f(\bm{b},K)} \subseteq \ol{\bigcup _{\bm{b}\in\mc{B}}E(\bm{b},K)}$, and (3) now follows.

The implication (3)$\Ra$(4) is trivial. (4)$\Ra$(1): Given a Borel partition $A_0,\dots ,A_{m-1}$ of $X$, $F\subseteq \Gamma$ finite, and $\epsilon >0$, let $k_0,\dots ,k_{m-1} \in 2^\N$ be distinct and define the function $\phi : X\ra 2^\N$ by $\phi (x) = i$ if $x\in A_i$. Then $\lambda = (\Phi ^{\phi ,d})_*\mu \in E(\bm{d},2^\N )$ so by (4) there exists a sequence $\Gamma \cc ^{b_n}(Y_n ,\nu _n) = \bm{b}_n \in \mc{B}$, along with $\phi _n :Y_n\ra 2^\N$ such that $\lambda _n \ra \lambda$, where $\lambda _n =(\Phi ^{\phi _n , b_n}) _* \nu _n$. Let $C_0,\dots ,C_{m-1}$ disjoint clopen subsets of $2^\N$ with $k_i \in C_i$ and for each $n\in \N$ let $B^n_i = \phi _n ^{-1}(C_i)$. Then for all $\gamma \in F$ we have
\begin{align*}
|\mu (\gamma ^dA_i\cap A_j ) - &\nu _n (\gamma ^{b_n}B^n_i \cap B^n_j )| = |\lambda (\pi _\gamma ^{-1}(C_i)\cap \pi _e^{-1}(C_j)) - \lambda _n (\pi _\gamma ^{-1}(C_i)\cap \pi _e^{-1}(C_j))| \ra 0 ,
\end{align*}
so for large enough $n$ this quantity is smaller than $\epsilon$.
%
\end{proof}

\subsection{Convexity in the space of actions}\label{subsection:convexity}

The convex sum of measure preserving actions is defined as follows (see also \cite[10.{\bf (F)}]{Ke10}). Let $N\in \{ 1,2,\dots , \infty = \N \}$ and let $\bm{\alpha} = (\alpha _0, \alpha _1\dots )\in [0,1]^N$ be a finite or countably infinite sequence of non-negative real numbers with $\sum _{i< N}\alpha _i = 1$. Given actions $\bm{b}_i = \Gamma \cc ^{b_i}(X_i ,\mu _i)$, $i<N$, we let $\sum _{i<N}X_i = \{ (i,x)\csuchthat i<N\mbox{ and }x\in X_i \}$ and we let $\tilde{\mu} _i$ be the image measure of $\mu _i$ under the inclusion map $X_i\hookrightarrow \sum _{i<N}X_i$, $\ x\mapsto (i,x)$.  We obtain a measure preserving action $\sum _{i<N}\alpha _i \bm{b}_i = \Gamma \cc ^{\sum _{i<N}b_i}(\sum _{i<N} X_i , \sum _{i<N}\alpha _i \tilde{\mu} _i )$ defined by $\gamma ^{\sum _{i<N}b_i} \cdot (i,x) = (i, \gamma ^{b_i}\cdot x)$.  If $(X_i,\mu _i)= (X,\mu )$ for each $i<N$ then $(\sum _{i<N}X_i ,\sum _{i<N}\alpha _i\tilde{\mu} _i )= (N\times X , \eta _{\bm{\alpha}} \times \mu )$ where $\eta _{\bm{\alpha}}$ is the discrete probability measure on $N$ given by $\eta _{\bm{\alpha} } (\{ i \} ) = \alpha _i$. If furthermore $\bm{b}_i=\bm{b}$ for each $i<N$ then $\sum _{i<N}\alpha _i \bm{b}_i =\bm{\iota} _{\eta _{\bm{\alpha}}} \times \bm{b}$ is simply the product action.

\begin{lemma}\label{lem:subsets}
Let $\bm{b}\in A(\Gamma ,X,\mu )$ and let $\bm{d} = \bm{\iota} _{\eta _{\bm{\alpha} }} \times \bm{b} = \sum _{i=0}^{n-1} \alpha _i\bm{b}$. Then $E(\bm{d},K)\subseteq \mbox{co}E(\bm{b},K) \subseteq E(\bm{\iota }\times \bm{b},K)$ for every compact Polish $K$.
\end{lemma}

\begin{proof}
Given $\phi : n\times X \ra K$, we want to show that $(\Phi ^{\phi ,d})_*(\eta _{\bm{\alpha} } \times \mu ) \in \mbox{co}E(\bm{b},K)$.  Let $\phi _i : X\ra K$ be given by $\phi _i (x) = \phi (i,x)$.  Then $(\Phi ^{\phi ,d})^{-1} (A) = \bigsqcup _{i=0}^{n-1} \{ i\} \times (\Phi ^{\phi _i , b})^{-1}(A)$ for $A\subseteq K^\Gamma$ and it follows that $(\Phi ^{\phi ,d})_*(\eta _{\bm{\alpha} } \times \mu ) = \sum _{i=0}^{n-1} \alpha _i (\Phi ^{\phi _i, b})_*\mu$, which shows the first inclusion.

Let the underlying space of $\bm{\iota}$ be $(Z,\eta )$. Given Borel functions $\phi _0,\dots ,\phi _{n-1} :X\ra K$ and $\alpha _0,\dots ,\alpha _{n-1}\geq 0$ with $\sum _{i=0}^{n-1}\alpha _i =1$, we want to show that $\sum _{i=0}^{n-1} \alpha _i (\Phi ^{\phi _i,b})_*\mu \in E(\bm{\iota} \times \bm{b},K)$. Let $C_0,\dots ,C_{n-1}$ be a Borel partition of $Z$ with $\eta (C_i) = \alpha _i$ for $i=0,\dots ,n-1$. Define $i:Z\ra n$ by $i(z) = i$ if $z\in C_i$ and let $\phi : Z\times X \ra K$ be the map $\phi (z,x) = \phi _{i(z)} (x)$. Then
\[
\Phi ^{\phi , \iota\times b} (z,x)(\gamma ) = \phi (\gamma ^{\iota \times b}\cdot (z,x)) = \phi (z,\gamma ^b\cdot x) = \phi _{i(z)}(\gamma ^b\cdot x) = \Phi ^{\phi _{i(z)},b}(x)(\gamma ),
\]
and so $(\Phi ^{\phi ,\iota\times b})^{-1}(A) = \bigsqcup _{i=0}^{n-1} C_i \times (\Phi ^{\phi _i,b})^{-1}(A)$ for all $A\subseteq K^\Gamma$. It now follows that $\sum _{i=0}^{n-1} \alpha _i (\Phi ^{\phi _i,b})_*\mu = (\Phi ^{\phi , \iota\times b}) _*(\eta \times \mu )$.
\end{proof}

\begin{lemma}\label{lem:setcon}
Let $\bm{b}\in A(\Gamma ,X,\mu )$, let $\bm{\alpha}(n) = (\tfrac{1}{n},\dots ,\tfrac{1}{n}) \in [0,1]^n$, and let
\begin{align*}
& \mc{B}_1 = \{ \bm{\iota} _{\eta _{\bm{\alpha}(n)}}\times \bm{b} \csuchthat n\geq 1 \} , \ \ \ \mc{B}_2 = \big\{ \bm{\iota} _{\eta _{\bm{\alpha}}}\times \bm{b} \csuchthat n\geq 1, \ \bm{\alpha} \in [0,1]^n, \ \textstyle{\sum }_{i=0}^{n-1}\alpha _i =1 \big\} . \nonumber
\end{align*}
Then $\bm{\iota }\times \bm{b}\sim \mc{B}_1\sim \mc{B}_2$.
\end{lemma}

\begin{proof}
$\mc{B}_1\prec \mc{B}_2$ is trivial. $\mc{B}_2 \prec \bm{\iota}\times \bm{b}$ is clear (in fact, $\bm{d} \sqsubseteq \bm{\iota}\times \bm{b}$ for every $\bm{d}\in \mc{B}_2$). It remains to show that $\bm{\iota }\times \bm{b}\prec \mc{B}_1$. Let $(Z ,\eta )$ be the underlying non-atomic probability space of $\bm{\iota}$ and let $\lambda = \eta \times \mu$. Fix a partition $\mc{P}$ of $Z\times X$, $F\subseteq \Gamma$ finite and $\epsilon >0$. We may assume without loss of generality that $\mc{P}$ is of the form $\mc{P} = \{ A_i\times B_j \csuchthat 0\leq i<n ,\  0\leq j< m \}$ where $\{ A_i \} _{i =0}^{n-1}$ is a partition of $Z$, $\{ B_j \} _{j=0}^{m-1}$ is a partition of $X$, and all the sets $A_0,\dots ,A_{n-1}$ have equal measure. Let $C_{i,j} = \{ (i,x)\in n\times X\csuchthat x\in B_j \}$. Then, letting $\bm{d}=\bm{\iota} _{\eta _{\bm{\alpha}(n) }}\times \bm{b}$, for all $\gamma \in F$ and $i,i'\leq n$, $j,j'\leq m$, if $i\neq i'$ we have $\gamma ^dC_{i,j}\cap C_{i',j'} = \emptyset = \gamma ^{\iota\times b}(A_i\times B_j) \cap (A_{i'}\cap B_{j'})$, while if $i=i'$ we have
\begin{align*}
(\eta _{\bm{\alpha}(n)} \times \mu )(\gamma ^{d}C_{i,j}\cap C_{i,j'})&= \tfrac{1}{n}\mu (\gamma ^bB_j\cap B_{j'}) \\
&= \eta (A_i)\mu (\gamma ^b B_j\cap B_{j'}) =\lambda (\gamma ^{\iota \times b} (A_i\times B_j)\cap (A_{i}\times B_{j'})),
\end{align*}
showing that $\bm{\iota }\times \bm{b} \prec \mc{B}_1$.
\end{proof}

\begin{proof}[Proof of Theorem \ref{thm:conv}]
We apply \ref{prop:AWgen} and \ref{lem:setcon}, then \ref{lem:subsets} to obtain
\[
E(\bm{\iota }\times \bm{b}, K) \subseteq \ol{{\textstyle{\bigcup _{n\geq 1} }}E(\bm{\iota }_{\eta _{\bm{\alpha}(n) }}\times \bm{b},K)} \subseteq \ol{\mbox{co}}E(\bm{b},K) \subseteq \ol{E(\bm{\iota }\times \bm{b} , K)}
\]
and so $\ol{E(\bm{\iota}\times \bm{b} ,K)} = \ol{\mbox{co}}E(\bm{b},K)$.
\end{proof}

The proof of Theorem \ref{thm:main} now proceeds in analogy with the proof of the corresponding fact for unitary representations (see \cite[F.1.4]{BdlHV08}).

\begin{proof}[Proof of Theorem \ref{thm:main}]
Suppose that $\bm{a}$ is ergodic and $\bm{a}\prec \bm{\iota }\times \bm{b}$. We want to show that $\bm{a}\prec \bm{b}$, or equivalently $E(\bm{a},K) \subseteq \ol{E,(\bm{b},K)}$ for every compact Polish $K$. By hypothesis we have that $E(\bm{a},K) \subseteq \ol{E(\bm{\iota}\times \bm{b} ,K )}$, so by Theorem \ref{thm:conv}, $E(\bm{a},K) \subseteq \ol{\mbox{co}}E(\bm{b},K)$. Since every element of $E(\bm{a},K)$ is ergodic, $E(\bm{a},K)$ is contained in the extreme points of $M_s(K^\Gamma )$, and so \emph{a fortiori} $E(\bm{a},K)$ is contained in the extreme points of $\ol{\mbox{co}}E(\bm{b},K)$. Since in a locally convex space the extreme points of a given compact convex set are contained in every closed set generating that convex set (see, e.g., \cite[Proposition 1.5]{Ph02}), it follows that $E(\bm{a},K) \subseteq \ol{E(\bm{b},K)}$ as was to be shown.
\end{proof}

\subsection{Ergodic decomposition and weak containment}\label{subsection:ergdec} We begin with the following observation about factors.

\begin{proposition}\label{prop:factor}
Let ${\bm d}$ be a measure preserving action of $\Gamma$ on $(Y,\nu )$ and suppose $\pi :(Y,\nu ) \ra (Z,\eta )$ is a factor map from $\bm{d}$ onto an identity action $\Gamma \cc ^{\iota _\eta } (Z,\eta )$. Let $\nu = \int _z \nu _z \, d\eta$ be the disintegration of $\nu$ with respect to $\pi$ and let $\bm{d}_z = \Gamma \cc ^d (Y,\nu _z)$. Suppose that $\bm{a}= \Gamma \cc ^a (X,\mu )$ is an ergodic factor of ${\bm d}$ via the map $\varphi :(Y,\nu )\ra (X,\mu )$. Then for $\eta$-almost every $z \in Z$, $\bm{a}$ is a factor of $\bm{d}_z$ via the map $\varphi$.
\end{proposition}


\begin{proof}
The map $\pi\times \varphi : (Y,\nu ) \ra (Z\times X, (\pi\times \varphi )_*\nu )$, $y\mapsto (\pi (y),\varphi (y))$, factors ${\bm d}$ onto a joining $\bm{b}$ of the identity action $\bm{\iota}_{\eta}$ and the ergodic action $\bm{a}$. Since ergodic and identity actions are disjoint (\cite[6.24]{Gl03}) we have that $(\pi\times \varphi )_*\nu = \eta \times \mu$ and $\bm{b}=\bm{\iota}_\eta \times \bm{a}$.  The measure $(\pi\times \varphi )_*\nu _z$ lives on $\{ z \} \times X$ almost surely, and $\eta \times \mu = (\pi \times \varphi )_*\nu = \int _Z (\pi \times \varphi )_*\nu _z \, d\eta$,
so by uniqueness of disintegration $(\pi \times \varphi )_* \nu _z = \updelta _z \times \mu$ almost surely. Since $\mbox{proj}_X\circ (\pi\times \varphi ) = \varphi$ we have that $\varphi _*\nu _z = (\mbox{proj}_X)_*(\updelta _z \times \mu )= \mu $ almost surely.
\end{proof}

\begin{corollary}
If $\bm{a}$ is ergodic and $\varphi$ factors $\bm{d}$ onto $\bm{a}$ then $\varphi$ factors almost every ergodic component of $\bm{d}$ onto $\bm{a}$.
\end{corollary}

Using ultraproducts of measure preserving actions (see Appendix \ref{app:ultra}) we can prove an analogous result for weak containment which generalizes Theorem \ref{thm:main}. For the remainder of this section we fix a nonprincipal ultrafilter $\mc{U}$ on $\N$ and we also fix a compact Polish space $K$ homeomorphic to $2^\N$. Let $\bm{a}_n = \Gamma \cc ^{a_n} (Y_n, \nu)$, $n\in \N$, be a sequence of measure preserving actions of $\Gamma$ and let $\bm{a}_{\mc{U}}= \Gamma \cc ^{a_\mc{U}}(Y_{\mc{U}},\mu _{\mc{U}})$ be the ultraproduct of the sequence $\bm{a}_n$ with respect to the nonprincipal ultrafilter $\mc{U}$ on $\N$. Let $\phi _n :Y_n \ra K$ be a sequence of Borel functions and let $\Phi _n= \Phi ^{\phi _n,a_n} : Y_n \ra K^\Gamma$. We let $\phi$ denote the ultralimit function determined by the sequence $\phi _n$, i.e., $\phi : Y_{\mc{U}}\ra K$ is the function given by
\[
\phi ([y_n]) = \lim _{n\ra \mc{U}} \phi _n (y_n)
\]
for $[y_n]\in Y_{\mc{U}}$. The function $\phi$ is $\bm{B}_{\mc{U}}$-measurable since $\phi ^{-1}(V) = [\phi _n ^{-1}(V)]$ whenever $V\subseteq K$ is open.

\begin{proposition}\label{prop:ultrafunct} Let $\Phi = \Phi ^{\phi , a_{\mc{U}}}$. Then
\begin{enumerate}
\item $\Phi ([y_n]) = \lim _{n\ra \mc{U}} \Phi _n(y_n)$ for all $[y_n]\in Y_{\mc{U}}$;
\item $\Phi _* \nu _{\mc{U}} = \lim _{n\ra \mc{U}}(\Phi _n)_*\nu _n$;
\item For every $\bm{B}_{\mc{U}}$-measurable function $\psi : Y_{\mc{U}}\ra K$ there exists a sequence $\varphi _n : Y_n \ra K$ of Borel functions such that $\psi ([y_n]) = \lim _{n\ra \mc{U}} \varphi _n (y_n)$ for $\nu _{\mc{U}}$-almost every $[y_n]\in Y_{\mc{U}}$.
\end{enumerate}
\end{proposition}

\begin{proof}
(1): For each $[y_n]\in Y_{\mc{U}}$ and $\gamma \in \Gamma$ we have $\Phi ([y_n])(\gamma ^{-1}) = \phi (\gamma ^{a_\mc{U}}[y_n]) = \phi ([\gamma ^{a_n}y_n]) = \lim _{n\ra \mc{U}} \phi (\gamma ^{a_n}y_n) = \lim _{n\ra \mc{U}} \Phi _n(y_n)(\gamma ) = (\lim _{n\ra \mc{U}} \Phi _n (y_n) )(\gamma )$, the last equality following by continuity of the evaluation map $f\mapsto f(\gamma )$ on $K^\Gamma$.

(2): Let $\lambda = \lim _{n\ra \mc{U}}(\Phi _n)_*\nu _n$. Then $\lambda$ is the unique element of $M_s(K^\Gamma )$ such that $\lambda (C) = \lim _{n\ra\mc{U}} ((\Phi _n)_*\nu _n (C))$ for all clopen $C\subseteq K^\Gamma$. Part (1) implies that $\Phi ^{-1}(C) = [\Phi _n ^{-1}(C)]$ whenever $C\subseteq K^\Gamma$ is clopen, and so $\Phi _*\nu _{\mc{U}} (C) = \lim _{n\ra \mc{U}} \nu _n (\Phi _n^{-1}(C)) = \lim _{n\ra \mc{U}} ((\Phi _n)_*\nu _n(C))$.

(3): We may assume $K= 2^\N$. For $m\in \N$ define $\psi _m : Y_{\mc{U}}\ra K$ by $\psi _m ([y_n])=\psi ([y_n])(m)$. For $i\in \{ 0,1 \}$ let $A^{m,i} = \psi _m^{-1} (\{ i \} )\in \bm{B}_{\mc{U}}$ and fix $[A^{m,i}_n] \in \bm{A}_\mc{U}$ such that $\nu _{\mc{U}}(A^{m,i} \Delta [A^{m,i}_n]) =0$. For each $m,n\in \N$ let $B^{m,0}_n= A^{m,0}_n\setminus A^{m,1}_n$ and let $B^{m,1}_n = Y_n\setminus B^{m,0}_n$ so that $\{ B^{m,0}_n , B^{m,1}_n \}$ is a Borel partition of $Y_n$. Then for each $m\in \N$ we have $\nu _{\mc{U}}(A^{m,0}\Delta [B^{m,0}_n]) = 0 = \nu _{\mc{U}}(A^{m,1} \Delta [B^{m,1}_n])$. Define $\varphi _n :Y_n\ra K$ by taking $\varphi _n (y)(m) = i$ if and only if $y\in B^{m,i}_n$. Let $\varphi :Y_{\mc{U}}\ra K$ be the ultralimit function $\varphi ([y_n]) = \lim _{n\ra \mc{U}} \varphi _n (y_n)$. Then for $i\in \{ 0,1 \}$ we have
\begin{align*}
\varphi ([y_n])(m) = i \ \Leftrightarrow \ \lim _{n\ra \mc{U}} (\varphi _n (y_n)(m)) = i \ \Leftrightarrow \ \{ n \csuchthat y_n\in B^{m,i}_n \} \in \mc{U} \ \Leftrightarrow \ [y_n] \in [B^{m,i}_n] ,
\end{align*}
and so $\varphi$ is equal to $\psi$ off the null set $\bigcup _{m\in \N , i\in \{ 0,1\} } A^{m,i}\Delta [B^{m,i}_n]$.
\end{proof}

\begin{theorem}\label{thm:aeweakcont}
Let ${\bm d}$ be a measure preserving action of $\Gamma$ on $(Y,\nu )$ and suppose $\pi :(Y,\nu ) \ra (Z,\eta )$ is a factor map from $\bm{d}$ onto an identity action $\Gamma \cc ^{\iota _\eta } (Z,\eta )$. Let $\nu = \int _z \nu _z \, d\eta$ be the disintegration of $\nu$ with respect to $\pi$ and let $\bm{d}_z = \Gamma \cc ^d (Y,\nu _z)$. Suppose that $\bm{a}= \Gamma \cc ^a (X,\mu )$ is ergodic and is weakly contained in ${\bm d}$. Then $\bm{a}$ is weakly contained in $\bm{d}_z$ for almost all $z\in Z$.
\end{theorem}

\begin{proof}
Taking $K= 2^\N$ it suffices to show for each $\lambda \in E(\bm{a},K)$ that $\eta (\{ z \csuchthat \lambda \in \ol{E(\bm{d}_z,K)} \} ) =1$. Let $\mc{U}$ be a nonprincipal ultrafilter on $\N$ and let $\bm{d}_\mc{U} = \Gamma \cc ^{d_\mc{U}} (Y_\mc{U} ,\nu _\mc{U})$ and $\bm{\iota}_{\mc{U}} = \Gamma \cc ^{\iota _\mc{U}}(Z_\mc{U} , \eta _{\mc{U}})$ be the ultrapowers of $\bm{d}$ and $\bm{\iota}_\eta$ respectively. The map $\pi _{\mc{U}}: Y_\mc{U}\ra Z_{\mc{U}}$ defined by $\pi _{\mc{U}}([y_n]) = [\pi(y_n)]$ factors $\bm{d}_\mc{U}$ onto $\bm{\iota}_{\mc{U}}$.

Given any $\lambda \in E(\bm{a},K)$, since $\bm{a}\prec \bm{d}$ there exists $\phi _n : Y\ra K$ such that $(\Phi ^{\phi _n, d})_*\nu \ra \lambda$. Let $\phi : Y_\mc{U}\ra K$ be the ultralimit of the functions $\phi _n$, let $\Phi _n = \Phi ^{\phi _n ,d}$ and let $\Phi = \Phi ^{\phi , d_{\mc{U}}} : Y_\mc{U}\ra K^\Gamma$. By Proposition \ref{prop:ultrafunct}.(2), $\Phi$ factors $\bm{d}_\mc{U}$ onto $\Gamma \cc ^s (K^\Gamma , \lambda )$.

Let $\rho = \sigma _* \nu _{\mc{U}}$, 
where $\sigma =\pi _\mc{U} \times \Phi : Y_\mc{U} \ra Z_{\mc{U}}\times K^\Gamma$ is the map $\sigma ([y_n]) = ( \pi _\mc{U} ([y_n]) , \Phi ([y_n]) )$. Then $\rho = \eta _{\mc{U}}\times \lambda$ since each standard factor of $\bm{\iota}_{\mc{U}}$ is an identity action so is disjoint from $\bm{a}$.
 Let $\nu _{[z_n]} = \prod _n \nu _{z_n} / \mc{U}$, so that $\nu _{[z_n]}$ is a probability measure on $\bm{B}_{\mc{U}}(Y_{\mc{U}})$ for all $[z_n]\in Z_{\mc{U}}$.

\begin{claim}
$\lim _{n\ra \mc{U}} (\Phi _n)_*\nu _{z_n} = \lambda$ for $\eta _{\mc{U}}$-almost every $[z_n] \in Z_{\mc{U}}$.
\end{claim}

\begin{proof}[Proof of Claim]
By Proposition \ref{prop:ultrafactor}, $\nu _{\mc{U}}(A) = \int _{[z_n]} \nu _{[z_n]} (A) \, d\eta_{\mc{U}}$ for all $A\in \bm{B}_{\mc{U}} (Y_{\mc{U}})$.
As $\sigma _* \nu _{[z_n]}$ lives on $\{ [z_n] \} \times K^\Gamma$ it follows for $D\in \bm{B}_\mc{U} (Z_{\mc{U}})$ and $C\subseteq K^\Gamma$ clopen that
\begin{align*}
\int _{[z_n]\in D} \lambda (C) \, d\eta _{\mc{U}} 
&= \eta _{\mc{U}}(D)\lambda (C) = \rho (D\times C) = \int _{[z_n]} \sigma _*\nu _{[z_n]} (D\times C) \, d\eta _{\mc{U}} \\
&= \int _{[z_n]\in D} \sigma _* \nu _{[z_n]} (Z_{\mc{U}}\times C)\, d\eta _{\mc{U}} = \int _{[z_n]\in D}\Phi _* \nu _{[z_n]} (C) \, d\eta _{\mc{U}} .
\end{align*}
Thus for each clopen $C\subseteq K^\Gamma$, $\Phi _*\nu _{[z_n]}(C) = \lambda (C)$ for $\eta _{\mc{U}}$ almost every $[z_n] \in Z_{\mc{U}}$. It follows that $\Phi _*\nu _{[z_n]} = \lambda$ for $\eta _{\mc{U}}$ almost every $[z_n]\in Z_{\mc{U}}$. By Proposition \ref{prop:ultrafunct}.(2), $\lim _{n\ra \mc{U}} (\Phi _n)_*\nu _{z_n} = \lambda$ for $\eta _{\mc{U}}$ almost every $[z_n] \in Z_{\mc{U}}$. \qedhere[Claim]
\end{proof}

If now $V$ is any open neighborhood of $\lambda$ in $M_s(K^\Gamma )$ then let $B= \{ z \in Z \csuchthat E(\bm{d}_z , K)\cap V =\emptyset \}$. If $\eta (B) >0$ then let $B_n = B$ for all $n$ so that $[B_n]\in \bm{A}_{\mc{U}}(Z_{\mc{U}})$ and $\eta _{\mc{U}} ([B_n])>0$. Thus, for some $[z_n]\in [B_n]$ we have $\lim _{n\ra \mc{U}} (\Phi _n)_*\nu _{z_n} = \lambda$ and so $(\Phi _n )_*\nu _{z_n} \in E(\bm{d}_{z_n},K)\cap V$ for some $n\in \N$. Since $z_n \in B_n=B$ this is a contradiction. Thus, $\eta (B)=0$. It follows that $\lambda \in \ol{E(\bm{d}_z, K)}$ almost surely.
\end{proof}

\begin{theorem}\label{thm:more}
Let $\varphi : \Gamma \cc ^b (X,\mu ) \ra \Gamma \cc ^{\iota _\eta} (Z,\eta )$ and $\psi : \Gamma \cc ^d (Y,\nu )\ra \Gamma \cc ^{\iota _\eta } (Z,\eta )$ be factor maps from $\bm{b}$ and $\bm{d}$ respectively onto $\bm{\iota}_\eta$. Let $\mu = \int _z \mu _z \, d\eta$ and $\nu = \int _z \nu _z \, d\eta$ be the disintegrations of $\mu$ and $\nu$ via $\varphi$ and $\psi$ respectively, and for each $z\in Z$ let $\bm{b}_z = \Gamma \cc ^b (X, \mu _z)$ and let $\bm{d}_z = \Gamma \cc ^d (Y, \nu _z)$. Then
\begin{enumerate}
\item If $\bm{b}_z\prec \bm{d}_z$ for all $z\in Z$ then $\bm{b} \prec \bm{d}$.
\item If $\bm{b}\prec \bm{d}_z$ for all $z\in Z$ the $\bm{\iota}_\eta \times \bm{b}\prec \bm{d}$ and if $\bm{b}_z \prec \bm{d}$ for all $z\in Z$ then $\bm{b}\prec \bm{\iota}_\eta \times \bm{d}$.
\item If $\bm{b}_z\sim \bm{d}_z$ for all $z\in Z$ then $\bm{b}\sim \bm{d}$ and if $\bm{b}\sim \bm{d}_z$ for all $z\in Z$ then $\bm{\iota}_\eta \times \bm{b}\sim \bm{d}$.
\end{enumerate}
We also have the following version for stable weak containment (see Appendix \ref{app:stable}):
\begin{enumerate}
\item[(4)] If $\bm{b}_z\prec _s \bm{d}_z$ for all $z\in Z$ then $\bm{b}\prec _s \bm{d}$.
\item[(5)] If $\bm{b}_z\prec _s\bm{d}$ for all $z\in Z$ then $\bm{b}\prec _s \bm{d}$ and if $\bm{b}_z\prec _s \bm{d}$ for all $z\in Z$ then $\bm{b}\prec _s \bm{d}$.
\item[(6)] If $\bm{b}_z\sim _s\bm{d}_z$ for all $z\in Z$ then $\bm{b}\sim _s \bm{d}$ and if $\bm{b}\sim _s\bm{d}_z$ for all $z\in Z$ then $\bm{b}\sim _s\bm{d}$.
\end{enumerate}
\end{theorem}

\begin{proof}

(1): Let $\{ B_n \} _{n\in \N}$ be a countable algebra of subsets of $Y$ generating the Borel $\sigma$-algebra of $Y$. 
Fix a partition $A_0,\dots ,A_{k-1}$ of Borel subsets of $X$ along with $F\subseteq \Gamma$ finite and $\epsilon >0$. For each $z$ there exists a $k$-tuple $(n_0,\dots ,n_{k-1}) \in \N ^k$ such that the sets $B_{n_0},\dots ,B_{n_{k-1}}\subseteq Y$ witness that $\bm{b}_z\prec \bm{d}_z$ with respect to the parameters $A_0,\dots ,A_{k-1}$, $F$, and $\epsilon$. We let $n(z) = (n_0(z),\dots ,n_{k-1}(z))$ be the lexicographically least $k$-tuple that satisfies this for $z$. For each $j<k$ the set
\[
D_j = \{ y \in Y\csuchthat \exists z\in Z (\psi (y) =z \mbox{ and } y\in B_{n_j(z)}) \} =\bigsqcup _z (B_{n_j(z)}\cap \psi ^{-1}(z))
\]
is analytic and so is measurable. For all $z\in Z$, $\gamma \in \Gamma$, and $j<k$ we then have that $\gamma ^d D_j \cap \psi ^{-1}(z) = \gamma ^{d_z} B_{n_j(z)}\cap \psi ^{-1}(z)$ and it follows that $\nu _z (\gamma ^d D_j \cap D_{j'}) = \nu _z (\gamma ^{d_z}B_{n_j(z)} \cap B_{n_{j'}(z)})$, since $\nu _z$ concentrates on $\psi ^{-1}(z)$.  For $\gamma \in F$ and $i,j < k$ we then have
\begin{align*}
|\nu (\gamma ^d D_i\cap D_j)& - \mu (\gamma ^{b}A_i \cap A_j )| = \big| \int _{z\in Z} \nu _z (\gamma ^d D_{i}\cap D_{j}) \, d\eta (z) - \int _{z\in Z} \mu _z (\gamma ^{b} A_i \cap A_{j}) \, d\eta (z) \big| \\
& \leq  \int _{z\in Z} |\nu _z(\gamma ^{d_z} B_{n_i(z)} \cap B_{n_{j}(z)}) - \mu _z (\gamma ^{b_z} A_i \cap A_{j})| \, d\eta (z) \leq \eta (Z)\epsilon <\epsilon
\end{align*}
which finishes the proof of (1).

Statements (2) through (6) now follow from (1).
\end{proof}

%
%
%

\begin{question}\label{question2}
Is every measure preserving action $\bm{d}$ of $\Gamma$ stably weakly equivalent to an action with countable ergodic decomposition?
\end{question}

A positive answer to Question \ref{question2} would be an ergodic theoretic analogue of the fact that every unitary representation of $\Gamma$ on a separable Hilbert space is weakly equivalent to a countable direct sum of irreducible representations (\cite{Di77}, this also follows from \cite[F.2.7]{BdlHV08}). 
We also mention the following related problem.

\begin{problem}\label{problempts}
Describe the set $\mbox{ex}(\ol{co}E(\bm{a}, 2^\N ))$ of extreme points of $\ol{co}E(\bm{a}, 2^\N )$ for $\bm{a} \in A(\Gamma ,X,\mu )$.
\end{problem}
%
%

\section{Consequences of Theorem \ref{thm:main} and applications to MD and EMD}\label{section3}

\subsection{Free, non-ergodic weak equivalence classes} We can now prove Theorem \ref{thm:nonerg}.
%

\begin{proof}[Proof of Theorem \ref{thm:nonerg}]
If $\bm{a}$ is any ergodic action of $\Gamma$ and $\bm{a}\prec {\bm \iota}\times \bm{b}$ then by Theorem \ref{thm:main} $\bm{a}\prec \bm{b}$, and so $\bm{a}$ is strongly ergodic. It follows that we cannot also have ${\bm \iota}\times \bm{b}\prec \bm{a}$, otherwise $\bm{a}$ would not be strongly ergodic.
\end{proof}

\begin{remark}\label{rem:amen1} Foreman and Weiss \cite[Claim 18]{FW04} show that for any free measure preserving action $\bm{a} = \Gamma \cc ^a (X,\mu )$ of an infinite amenable group $\bm{b}\prec \bm{a}$ for every $\bm{b}\in A(\Gamma ,X,\mu )$. We note that a quick alternative proof of this follows from \cite[Theorem 1.2]{BT-D11}, which says that if $\Delta$ is a normal subgroup of a countably infinite group $\Gamma$ and $\Gamma /\Delta$ is amenable, then $\bm{b}\prec \mbox{CInd}_\Delta ^\Gamma ((\bm{\iota }\times \bm{b})|\Delta )$ for every $\bm{b}\in A(\Gamma ,X,\mu )$. Taking $\Gamma$ to be an infinite amenable group and $\Delta = \langle e \rangle $ the trivial group, the restriction $({\bm \iota} \times \bm{b} )|\langle e \rangle$ is trivial, so $\mbox{CInd}_{\langle e \rangle }^\Gamma ( ({\bm \iota} \times \bm{b}) |\langle e \rangle )$ is the Bernoulli shift action $\bm{s}_\Gamma$ of $\Gamma$. Thus, $\bm{b}\prec \bm{s}_\Gamma$. By \cite[Theorem 1]{AW11} (or alternatively, Corollary \ref{thm:shift} below), since $\bm{a}$ is free, we have $\bm{s}_\Gamma \prec \bm{a}$ and so $\bm{b}\prec \bm{a}$.

Combining this with Theorem \ref{thm:nonerg} gives a new characterization of (non-)amenability for a countable group $\Gamma$.

\begin{corollary}\label{cor:newchar}
A countably infinite group $\Gamma$ is non-amenable if and only if there exists a free measure preserving action of $\Gamma$ that is not weakly equivalent to any ergodic action.
\end{corollary}
\end{remark}

\begin{remark}\label{rem:continuum}
It is noted in \cite[4.({\bf C})]{CK10} that if $\Gamma$ is a non-amenable group, and if $S\subseteq \Gamma$ is a set of generators for $\Gamma$ such that the Cayley graph $\mbox{Cay}(\Gamma ,S)$ is bipartite, then there are continuum-many weak equivalence classes of free measure preserving actions of $\Gamma$. Their method of using convex combinations of actions can be used to show that this holds for \emph{all} non-amenable $\Gamma$, and in fact the proof shows that there exists a collection $\{ \bm{a}_\alpha \csuchthat 0<\alpha \leq \tfrac{1}{2}\}$ with $\bm{a}_\alpha$ and $\bm{a}_\beta$ weakly incomparable when $\alpha\neq \beta$. Indeed, if $\bm{a} =\Gamma \cc ^a (X,\mu )$ is any free strongly ergodic action of $\Gamma$ (which exists when $\Gamma$ is non-amenable), then for any $0<\alpha <\beta \leq \tfrac{1}{2}$ the actions $\bm{a}_\alpha = \alpha \bm{a}+ (1-\alpha )\bm{a}$ and $\bm{a}_\beta = \beta \bm{a}+(1-\beta )\bm{a}$ are weakly incomparable. To see this note that any action weakly containing $\bm{a}_\alpha$ has a sequence of asymptotically invariant sets with measures converging to $\alpha$. Since $\bm{a}$ is strongly ergodic it is clear that no such sequence exists for $\bm{a}_\beta$, and so $\bm{a}_\alpha \not\prec \bm{a}_\beta$. Similarly, $\bm{a}_\beta\not\prec \bm{a}_\alpha$.

It is open whether every non-amenable group has continuum-many weak equivalence classes of free \emph{ergodic} measure preserving actions. It is in fact unknown whether there exists a non-amenable group with just one free ergodic action up to weak equivalence (though it is shown in the fourth remark after 13.2 of \cite{Ke10} that any such group must, among other things, have property (T) and cannot contain a non-abelian free group).  Ab\'{e}rt-Elek \cite{AE11b} show that $\Gamma$ has continuum-many pairwise weakly incomparable (hence inequivalent) free ergodic actions when $\Gamma$ is a finitely generated free group or a linear group with property (T). Their result also holds for stable weak equivalence in view of  the following consequence of Theorem \ref{thm:main}.

\begin{corollary}\label{cor:sweakerg}
Let $\bm{a}$ and $\bm{b}$ be ergodic measure preserving actions of $\Gamma$ and let $(Z,\eta )$ be a standard probability space. Then $\bm{a}\sim \bm{b}$ if and only if ${\bm \iota} _\eta \times \bm{a} \sim {\bm \iota} _\eta\times \bm{b}$. In particular $\bm{a}\sim \bm{b}$ if and only if $\bm{a}\sim _s \bm{b}$.
\end{corollary}

\begin{proof}
If $\bm{a}\sim \bm{b}$ then $\bm{\iota}_\eta \times \bm{a} \sim \bm{\iota} _\eta \times \bm{b}$ by continuity of the product operation. Conversely, if $\bm{\iota}_\eta \times \bm{a} \sim \bm{\iota} _\eta \times \bm{b}$ then $\bm{a}\prec \bm{\iota}_\eta \times \bm{a} \prec \bm{\iota}_\eta \times \bm{b}$ so that $\bm{a}\prec \bm{b}$ by Theorem \ref{thm:main}. Likewise, $\bm{b}\prec \bm{a}$, so $\bm{a}\sim \bm{b}$.
\end{proof}

I also do not know whether every non-amenable group has continuum-many stable weak equivalence classes of free measure preserving actions, or whether there exists a non-amenable group all of whose free measure preserving actions are stably weakly equivalent.
\end{remark}

%
%

\subsection{The properties MD and EMD}

\begin{definition}
Let $\mc{B}$ be a class of measure preserving actions of a countable group $\Gamma$. If $\bm{a}\in \mc{B}$ then $\bm{a}$ is called \emph{universal for $\mc{B}$} if $\bm{b}\prec \bm{a}$ for every $\bm{b}\in \mc{B}$. When $\bm{a}$ is universal for the class of \emph{all} measure preserving actions of $\Gamma$ then $\bm{a}$ is simply called \emph{universal}.
\end{definition}

\begin{definition}[\cite{Ke11}]
Let $\Gamma$ be a countably infinite group. Then $\Gamma$ is said to have property $\mbox{EMD}$ if the measure preserving action $\bm{p}_\Gamma$ of $\Gamma$ on its profinite completion is universal. $\Gamma$ is said to have property $\mbox{EMD}^*$ if $\bm{p}_\Gamma$ is universal for the class of all ergodic measure preserving actions of $\Gamma$. $\Gamma$ is said to have property $\mbox{MD}$ if $\bm{\iota}\times \bm{p}_\Gamma$ is universal.
\end{definition}

If $\Gamma$ has property $\mbox{EMD}$, $\mbox{EMD}^*$, or $\mbox{MD}$, then $\bm{p}_\Gamma$ must be free (this follows, e.g., from the \ref{lem:monot} below) and so $\Gamma$ must be residually finite. It is also clear that $\mbox{EMD}$ implies both $\mbox{EMD}^*$ and $\mbox{MD}$. We now show that $\mbox{EMD}^*$ and $\mbox{MD}$ are equivalent.
%
%

\begin{proof}[Proof of Theorem \ref{thm:MD}]
The implication $\mbox{EMD}^*\Ra \mbox{MD}$ is shown in \cite{Ke11}, but also follows from Theorem \ref{thm:more} above. For the converse, suppose $\Gamma$ has $\mbox{MD}$ so that ${\bm \iota}\times \bm{p}_\Gamma$ is universal and let $\bm{a}$ be an ergodic action of $\Gamma$. Then $\bm{a}\prec {\bm \iota}\times \bm{p}_\Gamma$, so since $\bm{a}$ is ergodic, Theorem \ref{thm:main} implies $\bm{a}\prec \bm{p}_\Gamma$. Thus $\bm{p}_\Gamma$ is universal for ergodic actions of $\Gamma$, and so $\Gamma$ has $\mbox{EMD}^*$.
\end{proof}

\begin{corollary}\label{cor:EMD(T)}
EMD and MD are equivalent for groups without property \emph{(T)}.
\end{corollary}

\begin{proof}
Suppose $\Gamma$ has MD and does not have (T). Then $\bm{\iota}\times \bm{p}_\Gamma$ is universal and by Theorem \ref{thm:MD}, $\bm{p}_\Gamma$ is universal for ergodic measure preserving actions. Since $\Gamma$ does not have property (T) there exists an ergodic $\bm{a} = \Gamma \cc ^a(X,\mu )$ with ${\bm \iota} \prec \bm{a}$, and so ${\bm \iota}\prec \bm{a} \prec \bm{p}_\Gamma$. Since $\bm{p}_\Gamma$ is ergodic with ${\bm \iota}\prec \bm{p}_\Gamma$ it follows that ${\bm \iota}\times \bm{p}_\Gamma \prec \bm{p}_\Gamma$ (see \cite[Theorem 3]{AW11}) and so $\bm{p}_\Gamma$ is universal.
\end{proof}

In what follows, if $\varphi : \Gamma \ra \Delta$ is group homomorphism then for each $\bm{a}\in A(\Delta ,X,\mu )$ we let $\bm{a}^\varphi \in A(\Gamma ,X,\mu )$ denote the action that is the composition of $\bm{a}$ with $\varphi$, i.e., $\gamma ^{a^\varphi} = \varphi (\gamma )^a$. Also, we note that for any two countable groups $\Gamma _1, \Gamma _2$, there is a natural equivariant homeomorphism from the diagonal action $\Aut (X,\mu )\cc A(\Gamma _1  ,X,\mu )\times A(\Gamma _2 ,X,\mu )$ to $\Aut (X,\mu )\cc A(\Gamma _1 \ast \Gamma _2  , X,\mu )$. We denote this map by $(\bm{a}_1,\bm{a}_2)\mapsto \bm{a}_1\ast \bm{a}_2$. See \cite[10.{\bf (G)}]{Ke10}. We also refer to \cite[Appendix G]{Ke10} and \cite{Zi84} for information about induced actions.

\begin{theorem}\label{thm:subgrpEMD}
Suppose $\Gamma _1$ and $\Gamma _2$ are nontrivial countable groups and that for each $i\in \{ 1,2 \}$, $\Gamma _i$ is either finite or has property MD. Then $\Gamma _1\ast \Gamma _2$ has property EMD.
\end{theorem}

\begin{proof}
Let $(\bm{a}_1,\bm{a}_2)\in A(\Gamma _1 ,X,\mu )\times A(\Gamma _2  ,X,\mu )$ be given and let $U= U_1\times U_2$ be an open neighborhood of $(\bm{a}_1,\bm{a}_2)$ where $U_i$ is a open neighborhood of $\bm{a}_i$ for $i=1,2$.  By hypothesis, for each $i=1,2$ there exists a finite group $F_i\neq \{ e \}$ along with a homomorphism $\varphi _i : \Gamma _i  \ra F_i$ and $\bm{b}_i\in A(F_i , X,\mu )$ such that the corresponding measure preserving action $\bm{b}_i ^{\varphi _i}$ of $\Gamma _i$ is in $U_i$.  Let $\varphi = \varphi _1 \ast \varphi _2 : \Gamma _1\ast \Gamma _2 \ra F_1\ast F_2$ and let $\bm{b}= \bm{b}_1\ast \bm{b}_2$. Then $\bm{b}^\varphi = \bm{b}_1^{\varphi _1}\ast \bm{b}_2^{\varphi _2}\in U_1\times U_2$. Let $V_1,V_2$ be open subsets about $\bm{b}_1\in A(F_1 ,X,\mu )$ and $\bm{b}_2\in A(F_2,X,\mu )$ respectively such that $\{ \bm{a}^{\varphi _i}\csuchthat \bm{a}\in V_i \} \subseteq U_i$ for $i=1,2$ (this is possible since the map $\bm{a}\mapsto \bm{a}^{\varphi _i}$ is continuous). Then $\bm{b}\in V_1\times V_2$ and for all $\bm{d}\in V_1\times V_2$ we have $\bm{d}^\varphi \in U_1\times U_2$.

There is a (possibly abelian) free subgroup $\F \leq F = F_1\ast F_2$ of finite index (explicitly: $\F =\mbox{ker}(\psi ) = [F,F]$ where $\psi : F_1\ast F_2 \ra F_1\times F_2$ is the natural projection map), and since $\F$ has EMD \cite[Theorem 1]{Ke11} we have $\bm{b}|\F \prec \bm{p}_\F$. Letting $\bm{a}_{F/ \F}$ denote the action of $F$ on $F/ \F$ with normalized counting measure we now have
\[
\bm{b}\sqsubseteq \bm{b}\times \bm{a}_{F/ \F } \cong \mbox{Ind}_\F ^F (\bm{b}|\Gamma )\prec \mbox{Ind}_\F ^F (\bm{p}_\F ).
\]
The action $\bm{d}=\mbox{Ind}_{\F} ^F (\bm{p}_\F )$ is a profinite action, and $\bm{d}$ is ergodic since $\bm{p}_\F$ is ergodic.  As $\bm{b}\prec \bm{d}$ there exists an isomorphic copy $\bm{d}_0$ of ${\bm d}$ in $V_1\times V_2$. Then $\bm{d}_0^{\varphi} \in U_1\times U_2$ and $\bm{d}_0^{\varphi}$ is ergodic since $\bm{d}_0$ is ergodic. Thus $U_1\times U_2$ contains an ergodic profinite action.
\end{proof}

\begin{note}
The group $\Gamma _1 \ast \Gamma _2 $ never has property (T) when $\Gamma _1$ and $\Gamma _2$ are nontrivial, so by Corollary \ref{cor:EMD(T)} it would have been enough to show in the above proof that $\Gamma _1 \ast \Gamma _2$ has MD, and then EMD would follow.
\end{note}

\begin{theorem}\label{thm:MDEMD}
The following are equivalent
\begin{enumerate}
\item MD and EMD are equivalent properties for any countably infinite group $\Gamma$.
\item EMD passes to subgroups.
\item MD is incompatible with property (T).
\end{enumerate}
\end{theorem}

\begin{proof}
(1)$\Ra$(2): Property MD passes to subgroups, so if MD and EMD are equivalent, then EMD passes subgroups. (2)$\Ra$(1): If $\Gamma$ is a countable group with MD then $\Gamma\ast \Gamma$ has EMD, so if EMD passes to subgroups then $\Gamma$ actually has EMD. (1)$\Ra$(3): EMD is incompatible with property (T) since if $\Gamma$ is an infinite residually finite group with property (T) then $\bm{p}_\Gamma$ is strongly ergodic so that $\bm{\iota}\not\prec \bm{p}_\Gamma$. Thus, if MD and EMD are equivalent then MD is also incompatible with property (T). (3)$\Ra$(1): This follows immediately from Corollary \ref{cor:EMD(T)}.
\end{proof}

Note also that Theorem \ref{thm:main} gives the following
\begin{proposition}
MD is incompatible with $((\tau )$ and $\neg (\mbox{\emph{T}}))$. That is, if a group $\Gamma$ has both MD and property $(\tau )$, then $\Gamma$ actually has property $(\mbox{\emph{T}})$.
\end{proposition}

\begin{proof}
If $\Gamma$ has MD then by \ref{cor:EMD(T)}, $\bm{p}_\Gamma$ is universal for ergodic actions, so if $\Gamma$ does not have (T) then there exists an ergodic $\bm{a}$ with ${\bm \iota}\prec \bm{a}$. This implies ${\bm \iota}\prec \bm{p}_\Gamma$ so that $\Gamma$ does not have property $(\tau )$.
\end{proof}

\section{Weak equivalence and invariant random subgroups}\label{section4}

\subsection{Invariant random subgroups}\label{subsection:IRS} We let $\mbox{Sub}(\Gamma )$ denote the set of all subgroups of $\Gamma$. This is a compact subset of $2^\Gamma$ with the product topology, and is invariant under the left conjugation action of $\Gamma$, which is continuous, and which we denote by $c$, i.e., $\gamma ^c\cdot H = \gamma H \gamma ^{-1}$. \emph{We will always view $\Gamma$ as acting on $\mbox{\emph{Sub}}(\Gamma )$ by conjugation, though the underlying measure on $\mbox{\emph{Sub}}(\Gamma )$ will vary}.  By an \emph{invariant random subgroup} (IRS) of $\Gamma$ we mean a conjugation-invariant Borel probability measure $\theta$ on $\mbox{Sub}(\Gamma )$. Invariant random subgroups are studied in \cite{AGV11} as a stochastic generalization of normal subgroups. See also \cite{AE11a}, \cite{Bo10} and \cite{Ve11}. We let $\mbox{IRS}(\Gamma )$ denote the space of all invariant random subgroups of $\Gamma$. When $\theta\in \mbox{IRS}(\Gamma )$ we will let $\bm{\theta}$ denote the measure preserving action $\Gamma \cc ^c (\mbox{Sub}(\Gamma ),\theta )$. For a measure preserving action $\bm{a}=\Gamma \cc ^a (X,\mu )$ we let $\mbox{type}(\bm{a})$ denote the \emph{type} of $\bm{a}$, which is defined to be the measure $(\mbox{stab}_a)_*\mu$ on $\mbox{Sub}(\Gamma )$, where $\mbox{stab}_a : X\ra \mbox{Sub}(\Gamma )$ is the stabilizer map $x\mapsto \mbox{stab}_a(x) = \Gamma _x= \{ \gamma \in \Gamma \csuchthat \gamma ^ax =x \} \in \mbox{Sub}(\Gamma )$. It is clear that $\mbox{type}(\bm{a})$ is always an IRS of $\Gamma$. Types are studied in \cite{AE11a} in order to examine freeness properties of measure preserving actions.

\subsection{The compact space of weak equivalence classes}\label{subsection:compspace} Ab\'{e}rt and Elek (\cite{AE11a}) define a compact Polish topology on the set of weak equivalence classes of measure preserving actions of $\Gamma$. We define this topology below and provide a variation of their proof showing that it is a compact Polish topology.

For this subsection we fix a standard probability space $(X,\mu )$ and a compact zero-dimensional Polish space $K$ homeomorphic to Cantor space $2^\N$. We let $\K = \K (M_s(K^\Gamma ))$ denote the space of all nonempty compact subsets of $M_s(K^\Gamma )$, equipped with the Vietoris topology $\tau _V$ which makes $\K$ into a compact Polish space. Since $M_s(K^\Gamma )$ is a compact metric space, convergence in this topology may be described as follows. A sequence $L_n \in \K$, $n\in \N$ converges if and only if the sets
\begin{align*}
\ol{\mbox{Tlim}}_nL_n &= \{ \lambda \in M_s(K^\Gamma )\csuchthat \exists (\lambda _n)\, [\forall n \, \lambda _n \in L_n, \mbox{ and } \lambda _n \ra \lambda ] \} \\
\underline{\mbox{Tlim}}_n L_n
&= \{ \lambda \in M_s(K^\Gamma )\csuchthat \exists (\lambda _n)\, [\forall n \, \lambda _n \in L_n,\mbox{ and for some subsequence }(\lambda _{n_k}), \ \lambda _{n_k}\ra \lambda ] \}
\end{align*}
are equal, in which case their common value is the limit of the sequence $L_n$ (see e.g., \cite[4.F]{Ke95}). 

Let $\Phi: A(\Gamma ,X,\mu )\ra \K$ be the map
\[
\Phi (\bm{a}) = \ol{E(\bm{a},K)} .
\]
By Proposition \ref{prop:AWgen}, $\Phi (\bm{a}) = \Phi (\bm{b})$ if and only if $\bm{a}\sim \bm{b}$. We now have

\begin{theorem}
The image of $\Phi$ in $\K$ is a closed, hence compact subset of $(\K ,\tau _V)$.
\end{theorem}

\begin{proof}
Let $\bm{a}_0,\bm{a}_1,\bm{a}_2,\dots$ be a sequence in $A(\Gamma ,X,\mu )$ and suppose that $\Phi (\bm{a}_n)$ converges in $(\K ,\tau _V)$ to the compact set $L\in \K$. We will show that there exists $\bm{a}_\infty \in A(\Gamma ,X,\mu )$ such that $\Phi (\bm{a}_\infty ) = L$.  Since $E(\bm{a}_n,K)$ is dense in $\Phi (\bm{a}_n)$ we may write $L$ as
\begin{align*}
L &= \{ \lambda \in M_s(K^\Gamma )\csuchthat \exists (\lambda _n)\, [\forall n \, \lambda _n \in E(\bm{a}_n,K),\mbox{ and }\lambda _n \ra \lambda ] \} \\
  &= \{ \lambda \in M_s(K^\Gamma )\csuchthat \exists (\lambda _n)\, [\forall n \, \lambda _n \in E(\bm{a}_n,K),\mbox{ and for some subsequence }(\lambda _{n_k}), \ \lambda _{n_k}\ra \lambda ] \} .
\end{align*}
Fix a nonprincipal ultrafilter $\mc{U}$ on $\N$, let $(X_\mc{U} , \mu _{\mc{U}})$ be the ultrapower of the measure space $(X,\mu )$, and let $\bm{a}_\mc{U} =\Gamma \cc ^{a_{\mc{U}}}(X_\mc{U}, \mu _{\mc{U}})$ denote the ultraproduct $\prod _n \bm{a}_n/\mc{U}$ of the sequence $\{ \bm{a}_n\} _{n\in \N}$. 
\begin{claim}
$L = E(\bm{a}_{\mc{U}} ,K)$.
\end{claim}

\begin{proof}[Proof of Claim]
Let $\lambda \in L$ and let $\lambda _n \in E(\bm{a}_n, K)$, $n\in \N$, with $\lambda _n \ra \lambda$. For each $n$ there exists $\phi _n :X\ra K$ such that $\lambda _n = (\Phi ^{\phi _n ,a_n})_*\mu$. Let $\phi : X_{\mc{U}}\ra K$ be the ultralimit of the functions $\phi _n$. By Proposition \ref{prop:ultrafunct}.(2) $(\Phi ^{\phi ,a_{\mc{U}}})_*\mu _{\mc{U}} = \lim _{n\ra \mc{U}} (\Phi ^{\phi _n, a_n})_*\mu = \lim _{n\ra \mc{U}}\lambda _n = \lambda$. This shows $\lambda \in E(\bm{a}_{\mc{U}},K)$, and thus $L\subseteq E(\bm{a}_{\mc{U}}, K)$.
%

Conversely, let $\lambda \in E(\bm{a}_\mc{U},K)$, say $\lambda = (\Phi ^{\psi , a_\mc{U}})_*\mu _{\mc{U}}$ for some $\bm{B}_{\mc{U}}$-measurable $\psi : X_{\mc{U}}\ra K$. By Proposition \ref{prop:ultrafunct}.(3)
we may find a sequence $\phi _n:X\ra K$, $n\in \N$, of Borel functions such that, letting $\phi$ denote the ultralimit of the $\phi _n$, $\mu _\mc{U}$-almost everywhere $\psi ([x_n]) = \phi ([x_n])$.  Let $\Phi _n = \Phi ^{\phi _n ,a_n}$, let $\Phi = \Phi ^{\phi ,a_{\mc{U}}}$, and let $\lambda _n = (\Phi _n)_*\mu \in E(\bm{a}_n ,K)$. Then $\Phi ^{\psi , a_{\mc{U}}} ([x_n]) = \Phi ([x_n])$ almost everywhere, so by Proposition \ref{prop:ultrafunct}.(2) we have $\lambda = (\Phi ^{\psi , a_{\mc{U}}})_* \mu _{\mc{U}} = \Phi _* \mu _{\mc{U}} = \lim _{n\ra \mc{U}} \lambda _n$ so there exists a subsequence $n_0<n_1<\cdots$ such that $\lambda _{n_k}\ra \lambda$. Hence $\lambda \in L$ and so $E(\bm{a}_{\mc{U}},K)\subseteq L$.
\qedhere[Claim]
\end{proof}

Let $D \subseteq L$ be a countable dense subset of $L = E(\bm{a}_\mc{U},K)$. For each $\lambda \in D$ we choose some $\bm{B} _\mc{U}$-measurable $\phi _\lambda : X_\mc{U} \ra K$ with $(\Phi ^{\phi _\lambda ,a_{\mc{U}}}) _* \mu _{\mc{U}} = \lambda$, and we also choose a sequence $\phi _{\lambda ,m}: X_\mc{U} \ra K$, $m\in \N$, of functions converging in measure to $\phi _\lambda$, such that each $\phi _{\lambda ,m}$ is constant on some $\bm{B} _\mc{U}$-measurable finite partition $\mc{P}^{(\lambda ,m)}$ of $X_\mc{U}$. By Theorem \ref{thm:genmeasalg} there exists a countably generated standard factor $\bm{M}$ of $\mbox{MALG}_{\mu _\mc{U}}$ containing $\bigcup _{\lambda \in D }\bigcup _{m\in \N} \mc{P}^{(\lambda ,m)}$ that is isomorphic to $\mbox{MALG}_{\mu}$. Let $\bm{a}_\infty$ be an action on $(X,\mu )$ corresponding to a point realization of the action of $\Gamma$ on $\bm{M}$ by measure algebra automorphisms. It is clear that $E(\bm{a}_\infty ,K ) \subseteq E(\bm{a}_\mc{U} ,K) = L$. We show that $D\subseteq \ol{E(\bm{a}_\infty ,K)}$. Given $\lambda \in D$, each of the functions $\phi _{\lambda ,m}$ is $\bm{M}$-measurable, so $(\Phi ^{\phi _{\lambda ,m},a_\mc{U}})_*\mu _{\mc{U}} \in E(\bm{a}_\infty ,K)$ for all $m$. Since $\phi _{\lambda ,m}\ra \phi _\lambda$ in measure it follows that $(\Phi ^{\phi _{\lambda ,m},a_\mc{U}})_*\mu_{\mc{U}} \ra \lambda$, and thus $\lambda \in \ol{E(\bm{a}_\infty ,K)}$. Thus $L = \ol{E(\bm{a} _\infty ,K)}$.
\end{proof}

For $\bm{a} \in A(\Gamma ,X,\mu )$ let $[\bm{a}]\subseteq A(\Gamma ,X,\mu )$ denote the weak equivalence class of $\bm{a}$ in $A(\Gamma ,X,\mu )$. Let $A_\sim (\Gamma ,X,\mu ) = \{ [\bm{a}]\csuchthat \bm{a}\in A(\Gamma ,X,\mu ) \}$ be the set of all weak equivalence classes of elements of $A(\Gamma ,X,\mu )$, and let $\tau$ denote the topology on $A_\sim (\Gamma ,X,\mu )$ obtained by identifying $A_\sim (\Gamma ,X,\mu )$ with a closed subset of $(\K ,\tau _V)$ via $\Phi$. This makes $A_\sim (\Gamma ,X,\mu )$ into a compact metrizable space.

\begin{theorem}\label{thm:contin}$\ $
\begin{enumerate}
\item[(1)] \cite{AE11a} The type, $\mbox{\emph{type}}(\bm{a})$, of a measure preserving action is an invariant of weak equivalence.
\item[(2)] The map $[\bm{a}]\mapsto \mbox{\emph{type}}(\bm{a})$ is a continuous map from the space $(A_\sim (\Gamma ,X,\mu ), \tau )$ of weak equivalence classes of measure preserving actions of $\Gamma$ to the space $\mbox{\emph{IRS}}(\Gamma )$ of invariant random subgroups of $\Gamma$ equipped with the weak${}^*$-topology.
\end{enumerate}
\end{theorem}

\begin{proof}
Let $\bm{b}_n \in A(\Gamma ,X,\mu )$, $n\in \N$, and suppose that $[\bm{b}_n]\ra [\bm{b}]$ in $\tau$, i.e., $\ol{E(\bm{b}_n,K)}\ra \ol{E(\bm{b},K)}$ in $\tau _V$. In light of Proposition \ref{prop:weakcontspace}, both (1) and (2) will follow once we show that $\mbox{type}(\bm{a}_n)\ra \mbox{type}(\bm{a})$ for all $\bm{a}_n \in [\bm{b}_n]$ and $\bm{a}\in [\bm{b}]$.  Let $\theta _n = \mbox{type}(\bm{a}_n )$ and let $\theta =\mbox{type}(\bm{a} )$. Let $F,G\subseteq \Gamma$ be finite. We define $N_F = \{ H\in \mbox{Sub}(\Gamma )\csuchthat F\cap H =\emptyset \}$, $N_{F,G} = \{ H \in\mbox{Sub}(\Gamma )\csuchthat F\cap H=\emptyset \mbox{ and } G\subseteq H \}$ and
\begin{align*}
A^n_F = \bigcap _{\gamma \in F} \mbox{supp}(\gamma ^{a_n})  \ \ \ &A^n_{F,G} =\bigcap _{\gamma \in F} \mbox{supp}(\gamma ^{a_n}) \cap \bigcap _{\gamma \in G} X\setminus \mbox{supp}(\gamma ^{a_n}) \\
A_F = \bigcap _{\gamma \in F} \mbox{supp}(\gamma ^{a})  \ \ \ &A_{F,G} =\bigcap _{\gamma \in F} \mbox{supp}(\gamma ^{a}) \cap \bigcap _{\gamma \in G} X\setminus \mbox{supp}(\gamma ^{a}) .
\end{align*}
Then $\theta _n (N_F)=\mu (A^n_F)$, $\theta _n(N_{F,G}) = \mu (A^n_{F,G})$, $\theta (N_F)= \mu (A_F)$, and $\theta (N_{F,G}) = \mu (A_{F,G})$. We will be done once we show that $\mu (A^n_{F,G}) \ra \mu (A_{F,G})$ for all finite $F,G\subseteq \Gamma$.

We first show that $\mu (A^n_{F}) \ra \mu (A_{F})$ for all finite $F\subseteq \Gamma$. 

\begin{lemma}\label{lem:monot}
$\mu (A_F)\leq \liminf _n \mu (A^n_F)$ for all finite $F\subseteq \Gamma$. 
\end{lemma}

\begin{proof}
We may write $A_F$ as a countable disjoint union $A_F=\bigsqcup _{m\geq 0} A_m$ where $\mu (\gamma ^aA_m \cap A_m )= 0$ for all $\gamma \in F$ and $m\in \N$. Then for any $\epsilon >0$ we can find $M$ so large that $\sum _{m\geq M}\mu (A_m) <\tfrac{\epsilon }{2|F|}$. Since $[\bm{a_n}]\ra [\bm{a}]$ in $\tau$ we have that $E(\bm{a} ,K) \subseteq \mbox{TLim}_n \ol{E(\bm{a}, K)}$ so by Proposition \ref{prop:AWgen} $\bm{a}\prec \{ \bm{a}_n\csuchthat n\in I\}$ for any infinite $I\subseteq \N$. Thus there exists $N$ such that for each $n>N$ we can find $A_0^n,\dots , A^n_{M-1}$ such that for all $\gamma \in F\cup \{ e \}$ and $i,j<M$ we have
\[
|\mu (\gamma ^a A_i \cap A_j) - \mu (\gamma ^{a_n} A_i^n\cap A_j^n ) |< \frac{\epsilon}{2M^2|F|} .
\]
Then, fixing $n$ with $n>N$, in particular we have $\mu (\gamma ^{a_n} A^n_i \cap A^n_i)<\tfrac{\epsilon}{2M ^2|F|}$ and $|\mu (A_i )-\mu (A^n_i) |<\frac{\epsilon}{2M^2|F|}$ for all $\gamma \in F$ and $i<M$, and $\mu (A^n_i\cap A^n_j) < \tfrac{\epsilon}{2M^2|F|}$ for all $i,j<M$, $i\neq j$. Define for $i<M$ the sets
\[
{B^n_i} = A^n_i \setminus \big( \bigcup _{\gamma \in F} \gamma ^{a_n} A^n_i \cup \bigcup _{j\neq i}A^n_j \big) .
\]
Then for $\gamma \in F$, $\gamma ^{a_n} B^n_i{} \cap {B^n_i} =\emptyset$ and for $i\neq j$, $B^n_i{}\cap B^n_j{} = \emptyset$. Thus $\bigsqcup B^n_i{} \subseteq A^n_F$. Since $\mu (B^n_i{})\geq \mu (A^n_i) - ((M - 1)+|F|)\tfrac{\epsilon}{2M^2|F|} > \mu (A_i) -\tfrac{\epsilon}{2M}$ it follows that $\mu (A^n_F) \geq \sum _{i<M} \mu (B^n_i{}) > (\sum _{i<M}\mu (A_i)) - \tfrac{\epsilon}{2} > \mu (A_F) - \epsilon$. Since this holds for all $n>N$ and since $\epsilon >0$ was arbitrary we are done. \qedhere[Lemma]
\end{proof}

\begin{lemma}\label{lem:sup}
$\limsup _n \mu (A^n_F) \leq \mu (A_F)$ for all finite $F\subseteq \Gamma$.
\end{lemma}

\begin{proof}
We may write each $A^{n}_{F}$ as a countable disjoint union $A^{n}_{F} = \bigsqcup _{m=0}^\infty A^{n}_m$ where for all $n,m\in \N$, $\gamma ^{a_n}\cdot A^{n}_m \cap A^{n}_m  = \emptyset$.  We also define $A^n_{-1} = X\setminus A^{n}_F$. Let $B_{-1},B_0,B_1,B_2,\dots$ be a sequence of disjoint nonempty clopen subsets of $K$, let $k_m \in B_m$, and define $\phi _n :X\ra K$ by $\phi _n (x) = k_m$ for $x\in A^{n}_m$. The set
\begin{align*}
B_{F} &= \{ f\in K^\Gamma \csuchthat (\forall m\geq -1) \, [f(e) \in B_m \Ra (\forall \gamma \in F ) (f(\gamma )\not\in B_m )] \} \\
			&= K^\Gamma\setminus \bigcup _{m\geq -1 } (\pi _e ^{-1}(B_m) \cap \bigcup _{\gamma \in F}\pi _\gamma ^{-1}(B_m))
\end{align*}
is closed and contained in the open set $U_F = \{ f\csuchthat \forall \gamma \in F \ f(\gamma ) \neq f(e) \}$. Fixing $n$, for each $m\geq 0$ we have that
\[
(\Phi ^{\phi _n,a_n})^{-1}(\pi _e ^{-1}(B_m) \cap \bigcup _{\gamma \in F}\pi _\gamma ^{-1}(B_m))=A^{n}_m \cap \bigcup _{\gamma \in F} \gamma ^{a_n}A^{n}_F =\emptyset ,
\]
while for $m=-1$ we have that $(\Phi ^{\phi _n , a_n})^{-1}(\pi _e ^{-1}(B_{-1}) \cap \bigcup _{\gamma \in F}\pi _\gamma ^{-1}(B_{-1}))= A^{n}_{-1}$ since $A^{n}_{-1}\subseteq \bigcup _{\gamma \in F} \gamma ^{a_n}A^{n}_{-1}$. It follows that $(\Phi ^{\phi _n, a_n})^{-1}(B_F) = A_F^{n}$. Let $\lambda _n = (\Phi ^{\phi _n , a_n})_*\mu \in E(\bm{a}_n,K)$. Take any convergent subsequence $\{\lambda _{n_k} \}$, and let $\lambda = \lim _k \lambda _{n_k}$. Since $\ol{E(\bm{a}_n,K)}\ra \ol{E(\bm{a},K)}$ we have that $\lambda \in \ol{E(\bm{a},K)}$, so let $\rho _n =(\Phi ^{\psi _n, a})_*\mu \in E(\bm{a},K)$ be such that $\rho _n \ra \lambda$. We now have
\begin{align*}
\textstyle{\limsup _k}\mu (A_F^{n_k})&=\textstyle{\limsup _k} \lambda _{n_k}(B_F) \leq \lambda (B_F) \leq \lambda ( U_F ) \\
&\leq \textstyle{\liminf _n}\rho _n(U_F)= \textstyle{\liminf _n}\mu ( \{ x\csuchthat \forall \gamma \in F \ \psi _n((\gamma ^{-1})^a x)\neq \psi _n(x) \} ) \leq \mu (A_F) .
\end{align*}
Since the convergent subsequence $ (\lambda _{n_k} )$ was arbitrary we conclude that $\limsup_n  \mu (A_F^{n})\leq \mu (A_F)$.
\end{proof}

It follows from the above two lemmas that $\mu (A_F)=\lim_n\mu (A_F^{n})$ for all finite $F\subseteq \Gamma$. Now let $F,G\subseteq \Gamma$ be finite and note that $A_{F}^{n} = A_{F , G}^{n} \sqcup \bigcup _{\gamma \in G}A_{F \cup \{ \gamma \} } ^{n}$ and $A_{F} = A_{F, G} \sqcup \bigcup _{\gamma \in G}A_{F \cup \{ \gamma \} }$. We have just shown that $\mu (A_{F})=\lim _n \mu (A_{F}^n)$. By the inclusion-exclusion principle we have $\mu (\bigcup _{\gamma \in G}A_{F \cup \{ \gamma \} }^{n}) = \sum_{k=1}^{|G|}(-1)^{k-1}\sum_{\{ J\subseteq G :|J|=k \}} \mu (A_{F\cup J}^{n})$, and since $\mu (A_{F\cup J}^{n})\ra \mu (A_{F\cup J})$ for each $J\subseteq G$ it follows after another application of inclusion-exclusion that $\mu (\bigcup _{\gamma \in G}A_{F \cup \{ \gamma \} }^{n} ) \ra \mu (\bigcup _{\gamma \in G}A_{F \cup \{ \gamma \} } )$. Thus $\mu (A_{F,G}^{n}) \ra \mu (A_{F,G})$.
\end{proof}

\begin{corollary}[\cite{AE11a}] For each $\theta \in \mbox{\emph{IRS}}(\Gamma )$, $\{ [\bm{a}] \csuchthat \mbox{\emph{type}}(\bm{a}) = \theta \} \subseteq A_\sim (\Gamma ,X, \mu )$ is compact in $\tau$. In particular $\{ [\bm{a}] \csuchthat [\bm{a}] \mbox{ is free} \}$ is compact in $\tau$.
\end{corollary}

%
%
%
\begin{remark}
The technique used in the proof of Theorem \ref{thm:contin} can be used to show that combinatorial invariants of measure preserving actions such as independence number (see \cite{CK10} and \cite{CKT-D11}) are continuous functions on $(A_\sim (\Gamma ,X,\mu ),\tau )$.
\end{remark}

\begin{theorem}Let $\Gamma$ be a countable group.
\begin{enumerate}
\item The map $( A(\Gamma ,X,\mu ) ,w ) \ra (A_{\sim}(\Gamma ,X,\mu ) , \tau )$, $\bm{a}\mapsto [\bm{a}]$, is Baire class 1. In particular, for each $\theta \in \mbox{\emph{IRS}}(\Gamma )$ the space $\{ \bm{a} \in A(\Gamma ,X,\mu )\csuchthat \mbox{\emph{type}}(\bm{a} ) =\theta \}$ is a $G_\delta$ hence Polish subspace of $(A(\Gamma ,X,\mu ),w)$.
\item The topology $\tau$ is a refinement of the quotient topology on $A_\sim (\Gamma ,X,\mu )$ induced by $w$. If $(X,\mu )$ is not a discrete space and $\Gamma \neq \{ e \}$ then the $\tau$ topology is strictly finer than the quotient topology.
\end{enumerate}
\end{theorem}

\begin{proof}
We begin with (1). For this we show that $\bm{a}\mapsto \ol{E(\bm{a},K)} \in \K$ is Baire class 1. We observe that $\{ \bm{a} \csuchthat \ol{E(\bm{a},K)} \subseteq C \}$ is closed in $( A(\Gamma ,X,\mu ), w)$ whenever $C\subseteq M_s(K^\Gamma )$ is closed.  This is because if $\bm{a}_n \in A(\Gamma ,X,\mu )$, $n\in \N$, is such that $\ol{E(\bm{a},K)} \subseteq C$ and $\bm{a}_n\ra \bm{a} \in A(\Gamma ,X,\mu )$ in the weak topology then $E(\bm{a},K) \subseteq \ol{\bigcup _n E(\bm{a}_n,K)} \subseteq C$. 

The topology $\tau _V$ on $\K$ is generated by the sets $\{ L\csuchthat L\subseteq U \}$ and $\{ L \csuchthat L\cap U\neq \emptyset \}$, where $U$ ranges over all open subsets of $M_s(K^\Gamma )$. For any open $U\subseteq M_s(K^\Gamma )$ the above observation shows that $\{ \bm{a} \csuchthat \ol{E(\bm{a},K)}\cap U \neq \emptyset \}$ is open, and if we write $U=\bigcup _n C_n$ where each $C_n$ is closed and $C_n\subseteq\mbox{int}(C_{n+1})$ 
then $\{ \bm{a}\csuchthat \ol{E(\bm{a},K)}\subseteq U \} = \bigcup _n \{ \bm{a}\csuchthat \ol{E(\bm{a},K)} \subseteq C_n \}$, which is $F_\sigma$.

For the first part of (2) we note that the following are equivalent for a subset $\mc{B}$ of $A(\Gamma ,X,\mu )$:
\begin{enumerate}
\item[(i)] $\mc{B}$ is weakly closed and for all $\bm{a},\bm{b}\in A(\Gamma ,X,\mu )$, $\bm{a}\in \mc{B}$ and $\bm{b}\sim \bm{a}$ implies $\bm{b}\in \mc{B}$.
\item[(ii)] $\mc{B}$ is weakly closed and for all $\bm{a},\bm{b}\in A(\Gamma ,X,\mu )$, $\bm{a}\in \mc{B}$ and $\bm{b}\cong \bm{a}$ implies $\bm{b}\in \mc{B}$.
\item[(iii)] For all $\bm{a}\in A(\Gamma ,X,\mu )$, $\bm{a}\prec \mc{B}$ implies $\bm{a}\in \mc{B}$.
\end{enumerate}
The implication (i)$\Ra$(ii) is trivial, (ii)$\Ra$(iii) follows from Proposition \ref{prop:AWgen}, and (iii)$\Ra$(i) follows from the fact that if $\bm{a}_n\ra \bm{a}$ in $A(\Gamma ,X,\mu )$ then $\bm{a}\prec \{ \bm{a}_n \} _{n\in \N}$. To show the first part of (2) it suffices to show that if $\mc{B}$ satisfies the above equivalent properties, then $\mc{B}_\sim = \{ [\bm{a}]\csuchthat \bm{a}\in \mc{B} \}$ is closed in $\tau$. Let $L = \ol{\bigcup _{\bm{a}\in \mc{B}}E(\bm{a},K)}$. Then $L\subseteq M_s(K^\Gamma )$ is closed and property (iii) tells us that $\mc{B}_\sim  = \{ [\bm{a}] \in A_\sim (\Gamma ,X,\mu ) \csuchthat \ol{E(\bm{a},K)}\subseteq L \}$, which is exactly the definition of a basic closed set in $\tau _V$.

Suppose that $(X,\mu )$ is not discrete and let $C\subseteq X$ be the continuous part of $X$ so that $\mu (C)>0$. Then $(C, \mu _C)$ is a standard non-atomic probability space so there exists a universal measure preserving action $\bm{a} = \Gamma \cc ^a (C,\mu _C)$ weakly containing all other measure preserving actions of $\Gamma$.  Let $b$ be the action of $\Gamma$ on $(X,\mu )$ whose restriction to $C$ is equal to $a$ and whose restriction to $X\setminus C$ is identity and let $\bm{b}= \Gamma \cc ^b (X,\mu )$. As $\bm{\iota}_{\mu _C}\prec \bm{a}$ by Lemma \ref{lem:10.1} there exist isomorphic copies of $\bm{a}$ converging to $\bm{\iota}_{\mu _C}$ in $A(\Gamma ,C, \mu _C)$. This yields isomorphic copies of $\bm{b}$ converging to $\bm{\iota}_\mu$ in $A(\Gamma ,X,\mu )$. Thus $[\bm{\iota}_\mu ]$ is in the closure of $\{ [\bm{b}] \}$ in the quotient topology, but $[\bm{\iota} _\mu ]$ is not in the $\tau$ topology closure of $\{ [\bm{b}]\}$ since $\Gamma \neq \{ e \}$ so that $[\bm{\iota}_\mu ]\neq [\bm{b}]$.
\end{proof}

\begin{remark}\label{rem:stableconv}
The map $\K \ra \K$ sending $L\mapsto \ol{\mbox{co}}L$ is continuous in the Vietoris topology $\tau _V$. Indeed, if $L_n\ra L_\infty$ we show that $\underline{\mbox{Tlim}}_n\ol{\mbox{co}}L_n \subseteq \ol{\mbox{co}}L_\infty\subseteq \ol{\mbox{Tlim}}_n \ol{\mbox{co}}L_n$. Let $\lambda \in \underline{\mbox{TLim}}_n \ol{\mbox{co}}L_n$ so that there exists $\lambda _{n_k}\in \ol{\mbox{co}}L_{n_k}$ with $\lambda _{n_k}\ra \lambda$. Then there exist probability measures $\mu_{n_k}$ on $M_s(K^\Gamma )$ supported on $L_{n_k}$ with $\lambda _{n_k}= \int _{\rho \in M_s(K^\Gamma )} \rho \, d\mu _{n_k}$ and (after moving to a subsequence if necessary) we may assume that $\mu _{n_k}$ converges to some measure $\mu$ on $M_s(K^\Gamma )$. Then $\lambda = \int _{\rho \in M_s(K^\Gamma )}\rho \, d\mu$. Let $C_0 \supseteq C_1\supseteq \cdots$ be a sequence of closed subsets of $M_s(K^\Gamma )$ with $L_\infty\subseteq \mbox{int}(C_m)$ for all $m$ and $L_\infty = \bigcap _m C_m$. For each $m$ the set $\{ L \in \K \csuchthat L\subseteq C_m \}$ is a neighborhood of $L_\infty$ in $\K$ and so contains $L_{n_k}$ for all large enough $k$. It follows that $\mu (C_m)\geq \liminf _k \mu _{n_k}(C_m) = 1$, and so $\mu (L_\infty ) = \lim _m \mu (C_m) = 1$. Since $\mu$ is supported on $L_\infty$ and has barycenter $\lambda$, it follows that $\lambda \in \ol{\mbox{co}}L_\infty$. For the second inclusion it is easy to see that $\mbox{co}L_\infty \subseteq \ol{\mbox{Tlim}}_n \ol{\mbox{co}}L_n$ and since the latter set is closed it follows that $\ol{\mbox{co}}L_\infty\subseteq \ol{\mbox{Tlim}}_n\ol{\mbox{co}}L_n$.

If now $\bm{a}$ is a measure preserving action of $\Gamma$ and $(Y,\nu )$ is non-atomic then $\bm{a}$ is stably weakly equivalent to an action on $(Y,\nu )$ and we let $[\bm{a}]_s= \{ \bm{b}\in A(\Gamma ,Y,\nu )\csuchthat \bm{b}\sim _s \bm{a} \}$ denote the stable weak equivalence class of $\bm{a}$ in $(Y,\nu )$ (see Definition \ref{def:swc}). It follows that the space $A_{\sim _s}(\Gamma ,Y,\nu ) = \{ [\bm{a}]_s \csuchthat \bm{a} \mbox{ is a measure preserving action of }\Gamma \}$ of all stable weak equivalence classes of measure preserving actions of $\Gamma$ may be viewed as a compact subset of $\K$ via the map $[\bm{a}]_s \mapsto \ol{\mbox{co}}E(\bm{a} ,K)$. Since $\mbox{type}(\bm{a} ) = \mbox{type}(\bm{\iota}\times \bm{a})$ it follows that $\mbox{type}(\bm{a} )$ is an invariant of stable weak equivalence.  The map $[\bm{a}]\mapsto \mbox{type}(\bm{a})$ then factors through $[\bm{a}]\mapsto [\bm{a}]_s$, and so Theorem \ref{thm:contin} also holds for stable weak equivalence.
\end{remark}

\subsection{Random Bernoulli shifts}\label{subsection:randbern} Given $\theta \in \mbox{IRS}(\Gamma )$, one constructs a measure preserving action of $\Gamma$ of type $\theta$ as follows (see {\cite[Proposition 45]{AGV11}}).

Fix a standard probability space $(Z,\eta )$ and let $Z^{\leq\backslash\Gamma} =\bigsqcup _{H\in\mbox{\tiny{Sub}}(\Gamma )} Z^{H\backslash \Gamma}$. Here, $H\backslash\Gamma$ denotes the collection of right cosets of $H$ in $\Gamma$. We define the projection map $Z^{\leq \backslash \Gamma}\ra \mbox{Sub}(\Gamma )$, $f\mapsto H_f\in \mbox{Sub}(\Gamma )$, where $H_f = H$ when $f\in Z^{H\backslash \Gamma }$. We endow $Z^{\leq\backslash \Gamma}$ with the standard Borel structure it inherits as a Borel subset of $Z^\Gamma \times \mbox{Sub}(\Gamma )$ via the injection $f\mapsto ((\gamma \mapsto f(H_f\gamma ) ), H_f )$. The image of $Z^{\leq\backslash\Gamma}$ under this map is invariant under the product action $\tilde{s}\times c$ of $\Gamma$ on $Z^\Gamma \times \mbox{Sub}(\Gamma )$ (where $\tilde{s}$ denotes the shift action of $\Gamma$ on $Z^\Gamma$), and we let $s$ denote the corresponding action of $\Gamma$ on $Z^{\leq\backslash\Gamma}$. We have that $H_{\gamma ^sf} = \gamma H_f \gamma ^{-1}$ for each $\gamma \in \Gamma$ and $f\in Z^{\leq\backslash\Gamma}$ and $(\gamma ^sf)(\gamma H_f\gamma ^{-1}\delta ) = f(H_f\gamma ^{-1}\delta )$. 
Let $\eta ^{H\backslash \Gamma}$ denote the product measure on $Z^{H\backslash\Gamma}\subseteq Z^{\leq\backslash\Gamma}$, and observe that under this action we have $(\gamma ^s)_*\eta ^{H\backslash \Gamma} = \eta ^{(\gamma H \gamma ^{-1})\backslash \Gamma }$. It follows that the measure $\eta ^{\theta\backslash \Gamma}$ on $Z^{\leq\backslash\Gamma}$ defined by
\[
\eta ^{\theta\backslash \Gamma} = \int _H\eta ^{H\backslash\Gamma} \, d\theta (H)
\]
is invariant under the action of $\Gamma$. We let $\bm{s}_{\theta ,\eta}$ denote the measure preserving action $\Gamma \cc ^s (Z^{\leq\backslash\Gamma}, \eta ^{\theta\backslash \Gamma} )$, and we call $\bm{s}_{\theta ,\eta }$ the \emph{$\theta$-random Bernoulli shift of $\Gamma$ over $(Z,\eta )$}. This action always contains $\bm{\theta}$ as a factor via the ``projection" map $f\mapsto H_f$. When $\eta$ is non-atomic then the stabilizer map $f\mapsto \Gamma _f$ of $\bm{s}_{\theta ,\eta}$ coincides almost everywhere with this projection. Indeed, if $\eta$ is non-atomic then for $\eta ^{\theta\backslash \Gamma}$-almost every $f$ the function $f: H\backslash \Gamma \ra Z$ is injective. Since every $\gamma \in \Gamma _{f}$ satisfies $f(H\gamma ^{-1})= f(H)$, the inclusion $\Gamma _{f}\subseteq H_f$ is immediate for injective $f$, and as $H_f\subseteq \Gamma _{f}$ always holds we conclude that $\Gamma _{f} =H_f$ almost surely. In particular $\mbox{type}(\bm{s}_{\theta ,\eta})=\theta$. We have thus shown the following.

\begin{proposition}[{\cite[Proposition 45]{AGV11}}]\label{proposition:stab1} Let $\Gamma$ be a countable group. For every $\theta \in \mbox{\emph{IRS}}(\Gamma )$ there exists a measure preserving action of type $\theta$. Namely, the $\theta$-random Bernoulli shift $\bm{s}_{\theta ,\eta}$ over a non-atomic base space $(Z,\eta )$ has type $\theta$.
\end{proposition}

It is clear that an isomorphism $(Z_1, \eta _1) \cong (Z_2, \eta _2)$ of measure spaces induces an isomorphism $\bm{s}_{\theta ,\eta _1}\cong \bm{s}_{\theta ,\eta _2}$.  
The next proposition characterizes precisely when $\mbox{type}(\bm{s} _{\theta ,\eta}) =\theta$ for various $\eta$. Below, we write $N(H)$ for the normalizer of a subgroup $H$ of $\Gamma$.

\begin{proposition}\label{proposition:stab2} Let $\Gamma$ be a countable group, let $\theta \in \mbox{IRS}(\Gamma )$, and let $(Z,\eta )$ be a standard probability space.
\begin{enumerate}
\item If $\eta$ is non-atomic then $\Gamma _f = H_f$ almost surely;

\item If $\eta$ is a point mass then $\Gamma _f = N(H_f)$ almost everywhere and the map $f\mapsto H_f$ is an isomorphism $\bm{s}_{\theta ,\eta} \cong \bm{\theta}$ so that $\mbox{\emph{type}}(\bm{s}_{\theta ,\eta}) = \mbox{\emph{type}}(\bm{\theta})$. 

\item Suppose $\eta$ is not a point mass. Then for each infinite index subgroup of $H\leq \Gamma$, $\Gamma _f = H_f$ for $\eta ^{H\backslash \Gamma}$-almost every $f\in Z^{H\backslash \Gamma}$. Thus, if
\[
\theta ( \{ H \csuchthat [\Gamma :H]<\infty \mbox{ and }N(H)\neq H \} ) = 0
\]
then $\Gamma _f = H_f$ almost surely. In particular if $\theta$ concentrates on the infinite index subgroups of $\Gamma$ then $\Gamma _f = H_f$ almost surely.

\item Suppose that $\eta$ contains atoms. If
\[
\theta ( \{ H\csuchthat [\Gamma :H]<\infty \mbox{ and } N(H)\neq H \} ) >0
\]
then $\mbox{\emph{type}}(\bm{s}_{\theta ,\eta}) \neq \theta$.
\end{enumerate}
In particular, $\mbox{\emph{type}}(\bm{s}_{\theta ,\eta})=\theta$ if and only if $H_f=\Gamma _f$ almost surely.
\end{proposition}

\begin{proof}
We have already shown (1) in Proposition \ref{proposition:stab1} and (2) is clear. 
%
For (3) fix an infinite index $H\leq \Gamma$ along with some $\gamma \not\in H$ and inductively define an infinite sequence $\{ \delta _n \} _{n\in\N}$ by taking $\delta _{n+1}\in \Gamma$ to be any element of the complement of $\bigcup _{i\leq n} (H\delta _i \cup H\gamma ^{-1}\delta _i \cup (\gamma H\gamma ^{-1})\delta _i \cup (\gamma H\gamma ^{-1})(\gamma \delta _i) )$ (we are using here the fact that the collection $\{ H\delta \csuchthat H\in \mbox{Sub}(\Gamma ), \, \delta \in \Gamma , \, \mbox{ and }[\Gamma :H] = \infty \}$ of all right cosets of infinite index subgroups of $\Gamma$ generates a proper ideal of $\Gamma$ (see, e.g., the proof of Lemma 4.4 in \cite{Ke05})). By construction all of the cosets $H\delta _0$, $H\gamma ^{-1}\delta _0$, $H\delta _1$, $H\gamma ^{-1}\delta _1 ,\dots$ are distinct so, letting $A\subseteq Z$ be any set with $0<\eta (A)<1$, it follows that
\begin{equation*}
\begin{split}
\eta ^{H\backslash \Gamma} &( \{ f\csuchthat \gamma \in \Gamma _{f} \} ) \leq \eta ^{H\backslash \Gamma} (\{ f \csuchthat \forall \delta \in \Gamma \, (f(H\delta )=f(H\gamma ^{-1}\delta ) )\} ) \\
&\leq \eta ^{H\backslash \Gamma} ( \bigcap _{n\in \N} \{ f\csuchthat f(H\delta _n), f(H\gamma ^{-1}\delta _n)\in A \mbox{ or } f(H\delta _n), f(H\gamma ^{-1}\delta _n) \not\in A \} ) \\
&= \lim _{N\ra \infty} (\eta (A)^2 + (1-\eta (A))^2)^N = 0 .
\end{split}
\end{equation*}
Thus $\gamma \not\in \Gamma _{f}$ for $\eta ^{H\backslash\Gamma}$-almost every $f$, and since this is true for each $\gamma\not\in H$ we obtain $\Gamma _{f}\subseteq H$ for $\eta ^{H\backslash\Gamma}$-almost every $f$.


We now prove (4). Let $\theta _{\bm{s}} =\mbox{type}(\bm{s}_{\theta ,\eta})$. 
Let $z_0 \in Z$ be an atom for the measure $\eta$. The set $A = \{ f\in Z^{\leq \backslash \Gamma}\csuchthat [\Gamma : H_f ] <\infty , \ N(H_f)\neq H_f \mbox{ and }\forall \gamma \in \Gamma \, (f(H_f\gamma ) = z_0) \}$ is $\eta ^{\theta \backslash \Gamma}$-non-null and $\Gamma _f = N(H_f)\neq H_f$ for each $f\in A$. Thus $[\Gamma :\Gamma _f] = [\Gamma :N(H_f)]< [\Gamma :H_f ]$ for each $f\in A$. When $f\not\in A$ we still have $[\Gamma :\Gamma _f]\leq [\Gamma :H_f]$. It follows that
\begin{align*}
\int _H &\frac{1}{[\Gamma :H]}\, d\theta _{\bm{s}}
		= \int _{f\in A} \frac{1}{[\Gamma :\Gamma _f]} \, d\eta ^{\theta \backslash \Gamma} + \int _{f\not\in A}\frac{1}{[\Gamma :\Gamma _f]} \, d\eta ^{\theta \backslash \Gamma} \\
		&> \int _{f\in A}\frac{1}{[\Gamma :H_f]}\, d\eta ^{\theta \backslash \Gamma} + \int _{f\not\in A}\frac{1}{[\Gamma :H_f]} \, d\eta ^{\theta \backslash \Gamma} = \int _f \frac{1}{[\Gamma :H_f]} \, d\eta ^{\theta \backslash \Gamma} = \int _H \frac{1}{[\Gamma :H]}\, d\theta
\end{align*}
and so $\theta _{\bm{s}}\neq \theta$, which finishes (4).
%

It is clear that $\Gamma _f = H_f$ almost everywhere implies $\mbox{type}(\bm{s}_{\theta ,\eta}) =\theta$. Suppose now that $\Gamma _f \neq H_f$ for a non-null set of $f\in Z^{\leq \backslash \Gamma}$. Then (1) implies that $\eta$ contains atoms and (3) implies that the set $J = \{ f\in Z^{\leq \backslash\Gamma} \csuchthat [\Gamma :H_f ] <\infty \mbox{ and } \Gamma _f \neq H_f \}$ is non-null. The inclusions $H_f\subseteq \Gamma _f\subseteq N(H_f)$ holds for all $f\in Z^{\leq \backslash \Gamma}$ and so
\[
\theta ( \{ H \csuchthat [\Gamma : H ] <\infty \mbox{ and } N(H) \neq H \} ) 
\geq \eta ^{\theta \backslash \Gamma} (J) >0.
\]
Part (4) now implies that $\mbox{type}(\bm{s}_{\theta ,\eta })\neq \theta$.
\end{proof}

\begin{theorem}
Let $\Gamma$ be a countable group, let $\theta \in \mbox{\emph{IRS}}(\Gamma )$, and let $\bm{s}_{\theta ,\eta}$ be the $\theta$-random Bernoulli shift over the standard measure space $(Z,\eta )$. Let $p :Z^{\leq \backslash \Gamma }\ra \mbox{\emph{Sub}}(\Gamma )$ denote the projection $p(f)= H_f$ factoring $\bm{s}_{\theta ,\eta}$ onto $\bm{\theta}$. Assume that $\eta$ is not a point mass. Then the following are equivalent
\begin{enumerate}
\item $\theta$ concentrates on the infinite index subgroups of $\Gamma$.
\item The extension $p :\bm{s}_{\theta ,\eta } \ra \bm{\theta}$ is ergodic.
\item The extension $p:\bm{s}_{\theta ,\eta } \ra \bm{\theta}$ is weak mixing.
\end{enumerate}
In particular, if $\bm{\theta}$ is infinite index then $\bm{s}_{\theta ,\eta}$ is ergodic if and only if  $\bm{\theta}$ is ergodic.
\end{theorem}

\begin{proof}
(3)$\Ra$(2) is trivial. (2)$\Ra$(1): Suppose that $\theta (C) >0$ where $C= \{ H\csuchthat [\Gamma :H]<\infty \}$ and let $A\subseteq Z$ be any measurable set with $0<\eta (A)<1$. Then the set $B = \{ f\in Z^{\leq \backslash\Gamma}\csuchthat H_f\in C \mbox{ and } \mbox{ran}(f)\subseteq A\}$ is a nontrivial invariant set that is not $p$-measurable.

(1)$\Ra$(3): We must show that the extension $\tilde{p} : \bm{s}_{\theta ,\eta }\otimes _{\bm{\theta}} \bm{s}_{\theta ,\eta} \ra \bm{\theta}$ is ergodic, where
\[
\bm{s}_{\theta ,\eta }\otimes _{\bm{\theta}} \bm{s}_{\theta ,\eta} = \Gamma \cc ^{s\times s} \big( Z^{\leq \backslash \Gamma}\times Z^{\leq \backslash \Gamma} , \int _H \eta ^{H\backslash \Gamma} \times \eta ^{H\backslash \Gamma}\, d\theta \big)
\]
and $\tilde{p}(f,g) = p(f)$.  Let $(Y,\nu ) = (Z\times Z, \eta \times \eta )$. Then we have the natural isomorphism $\varphi : \bm{s}_{\theta ,\eta }\otimes _{\bm{\theta}} \bm{s}_{\theta ,\eta} \cong \bm{s}_{\theta ,\nu }$ such that $\tilde{p}(f,g) = p\circ \varphi (f,g)$ almost surely, so it suffices to show that the extension $p:\bm{s}_{\theta ,\nu} \ra \bm{\theta}$ is ergodic. If $\theta = \int _{w\in W} \theta (w) \, d\rho (w)$ is the ergodic decomposition of $\theta$ then $\bm{s}_{\theta ,\nu}$ decomposes as $\bm{s}_{\theta ,\nu } = \int _{w\in W} \bm{s}_{\theta _w ,\nu } \, d\rho (w)$ and $p : Y ^{\leq \backslash \Gamma}\ra \mbox{Sub}(\Gamma )$ factors $\bm{s}_{\theta _w ,\nu }$ onto $\bm{\theta} _w$ almost surely. We may therefore assume that $\bm{\theta}$ is ergodic toward the goal of showing that $\bm{s}_{\theta ,\nu }$ is ergodic as well.

Since $\bm{\theta}$ is ergodic, the index $i$ of $N(H)$ in $\Gamma$ is constant on a $\theta$-conull set.  If $i<\infty$ then the orbit of almost every $H$ is finite and ergodicity implies that there exists an $H_0\in \mbox{Sub}(\Gamma )$ such that $\theta$ concentrates on the conjugates of $H_0$. Then $H_0$ is an infinite index normal subgroup of $K_0=N(H_0)$ which implies that the generalized Bernoulli shift action $\bm{s} = K_0 \cc ^s (Y ^{H_0\backslash \Gamma} , \eta ^{H_0\backslash \Gamma})$ is ergodic (see e.g., \cite{KT09}). Example \ref{example:ergodicIRS} below then shows that $\bm{s}_{\theta , \nu }\cong \mbox{Ind}_{K_0}^\Gamma (\bm{s})$, and so $\bm{s}_{\theta ,\nu }$ is ergodic.

If $i=\infty$ then we proceed as follows. Let $(X,\mu )= (Y^{\leq \backslash \Gamma} , \nu ^{\leq \backslash \Gamma})$ and suppose toward contradiction that $B\subseteq X$ is invariant and $0< \mu(B)=r < 1$.  
The map $H\mapsto \nu ^{H\backslash \Gamma} (B)$ is conjugation invariant so ergodicity of $\bm{\theta}$ implies that $\nu ^{H\backslash \Gamma } (B) = \mu (B) = r$ almost surely. Let $\epsilon >0$ be small depending on $r$. Fix some countable Boolean algebra $\bm{A}_0$ generating $\bm{B}(Y)$ and let $\bm{A}$ be the countable Boolean algebra of subsets of $X$ generated by $\{ \pi _\gamma ^{-1} (D) \csuchthat D \in \bm{A}_0\mbox{ and } \gamma \in \Gamma \}$ where $\pi _\gamma (f)= f(H_f\gamma )$ for $f\in X$. Then for every $\epsilon >0$ there exists $A_1,\dots ,A_n\in \bm{A}$ and a partition $C_0,\dots, C_{n-1}$ of $\mbox{Sub}(\Gamma )$ into non-null measurable sets such that $\mu (A\Delta B) <\epsilon ^2$ where $A=\bigsqcup _{i<n} (A_i\cap p^{-1}(C_i))$.
%
%
There exists a finite $F\subseteq \Gamma$ and a collection $\{ D^{i,j}_\delta \csuchthat \delta \in F, \ j< n_i, \ i< n \} \subseteq \bm{A}_0$ such that $A_i= \bigcup _{0\leq j<n_i} \bigcap _{\delta\in F} \pi _\delta ^{-1}(D^{i,j}_\delta )$ for each $i<n$.

\begin{lemma}
Let $C\subseteq \mbox{\emph{Sub}}(\Gamma )$ be any non-null measurable set. Then for $\theta$-almost every $H\in \mbox{\emph{Sub}}(\Gamma )$ there exists $\gamma \in \Gamma$ such that $\{ H\alpha \}_{\alpha \in F} \cap \{ H\gamma ^{-1}\delta \} _{\delta \in F} = \emptyset$ and $\gamma H \gamma ^{-1} \in C$.
\end{lemma}

\begin{proof}
Since $\bm{\theta}$ is ergodic and $[\Gamma : N(H)] =\infty$ almost surely, the intersection $C^H$, of $C$ with the orbit of $H$, is almost surely infinite. 
Fix such an $H$ with both $[\Gamma :N(H)] =\infty$ and $C^H$ infinite. Since the set $FF^{-1}\cdot H = \{ \delta \alpha ^{-1}H\alpha \delta ^{-1} \csuchthat \alpha ,\delta \in F \}$ is finite there exists $\gamma \in \Gamma$ with $\gamma H \gamma ^{-1} \in C^H \setminus (FF^{-1}\cdot H )$. This $\gamma$ works: $\gamma H\gamma ^{-1} \not\in FF^{-1} \cdot H$ is equivalent to $\gamma \not\in \bigcup _{\alpha ,\delta \in F} \delta \alpha ^{-1}N(H)$, so if $\alpha , \delta \in F$ then $\gamma \not\in \delta\alpha ^{-1}N(H)$ and thus $H \alpha \neq H\gamma ^{-1}\delta$.
%
\end{proof}


Using this lemma and measure-theoretic exhaustion we may find a Borel function $\mbox{Sub}(\Gamma ) \ra \Gamma$, $H\mapsto \gamma _H$, with $\{ H \alpha \} _{\alpha \in F} \cap \{ H\gamma _H ^{-1}\delta \} _{\delta\in F} = \emptyset$ and $\gamma _H H \gamma _H ^{-1} \in C_i$ for almost every $H\in C_i$, and such that the function $\psi : \mbox{Sub}(\Gamma ) \ra \mbox{Sub}(\Gamma )$, $H\mapsto \gamma _H H\gamma _H ^{-1}$, is injective on a conull set. In particular, $\psi$ is measure preserving. Let $\varphi : X\ra X$ be given by $\varphi (f) = (\gamma _{H_f})^s\cdot f$ so that $\varphi$ is also injective on a conull set and measure preserving.

For $H\leq \Gamma$ and $D\subseteq X$ let $D_H = D\cap Y^{H\backslash \Gamma}$. Then for each $i<n$ and almost every $H\in C_i$ we have $\gamma _HH\gamma _H ^{-1}\in C_{i}$ and
\begin{align*}
\varphi (A)_{\gamma _HH \gamma _H ^{-1}} 
&= (\gamma _H )^s\cdot ((A_i)_H)
 = \bigcup _{j<n_i} \bigcap _{\alpha \in F}
\{ f \in Y ^{\gamma _H H \gamma _H^{-1} \backslash \Gamma} \csuchthat f(\gamma _H H \gamma _H^{-1} \gamma _H\alpha ) \in D^{i,j}_\alpha \} \\
A_{\gamma _HH \gamma _H ^{-1}} &= (A_{i})_{\gamma _HH \gamma _H ^{-1}} = \bigcup _{j<n_i} \bigcap _{\delta \in F}
\{ f \in Y ^{\gamma _H H \gamma _H^{-1} \backslash \Gamma} \csuchthat f(\gamma _H H \gamma _H^{-1} \delta ) \in D^{i,j}_\delta \}
\end{align*}
By our choice of $\gamma _H$ the sets $\{ \gamma _H H \gamma _H ^{-1}\gamma _H\alpha \} _{\alpha \in F}$ and $\{ \gamma _HH \gamma _H ^{-1}\delta \} _{\delta \in F}$ are almost surely disjoint and it follows that the sets $A$ and $\varphi (A)$ are $\nu ^{\gamma _H H \gamma _H ^{-1}\backslash \Gamma}$-independent almost surely. Since $H\mapsto \gamma _H H\gamma _H ^{-1}$ is a measure preserving injection it follows that $A$ and $\varphi (A)$ are $\nu ^{H\backslash \Gamma}$-independent almost surely.

We have $\epsilon ^2 > \mu (A\Delta B) = \int _H \nu ^{H\backslash \Gamma} (A\Delta B)\, d\theta \geq \int _H |\nu ^{H\backslash \Gamma } (A) - r| \, d\theta$ so that $\theta ( \{ H \csuchthat |\nu ^{H\backslash \Gamma}(A) -r|\leq \epsilon \} ) \geq 1-\epsilon$ and since $\mu (A\Delta B) = \mu (\varphi (A)\Delta B)$ we also have $\theta (\{ H \csuchthat |\nu ^{H\backslash \Gamma} (\varphi (A)) - r| \leq \epsilon \} ) \geq  1- \epsilon$. Then
\begin{align*}
r = \mu (B) &\leq \mu (A\Delta B) + \mu (\varphi (A)\Delta B) + \mu (A\cap\varphi (A))\\
&< 2\epsilon ^2 + \int _H \nu ^{H\backslash \Gamma} (A)\nu ^{H\backslash \Gamma} (\varphi (A)) \, d\theta
\leq 2\epsilon ^2 + 2\epsilon + (r+\epsilon )^2 \ra _{\epsilon \ra 0} r^2.
\end{align*}
This is a contradiction for small enough $\epsilon$ since $0<r<1$.
\end{proof}

\begin{example}\label{example:ergodicIRS}
%
The simplest example of an ergodic $\theta \in \mbox{IRS}(\Gamma )$ is a point mass $\theta =\updelta _N$ on some normal subgroup $N\triangleleft \Gamma$. The corresponding random Bernoulli shift $\bm{s}_{\updelta _N , \eta }$ is isomorphic to the usual generalized shift action of $\Gamma$ on $(Z^{\Gamma /N} , \eta ^{\Gamma /N})$.

Almost as simple is when $\theta \in \mbox{IRS}(\Gamma )$ has the form $\theta = \tfrac{1}{n}\sum _{i=0}^{n-1} \updelta _{\gamma _iH\gamma _i^{-1}}$ where $H\leq \Gamma$ is a subgroup with finitely many conjugates $\gamma _0H\gamma _0^{-1}, \, \gamma _1 H \gamma _1^{-1}, \, \gamma _2H\gamma _2^{-1}, \dots \gamma _{n-1}H\gamma _{n-1}^{-1}$. Clearly $\bm{\theta}$ is ergodic. In this case the random Bernoulli shift $\bm{s}_{\theta ,\eta}$ may be described as follows. The set $T=\{ \gamma _i \} _{i<n}$ is a transversal for the left cosets of the normalizer $K = N(H)$ of $H$ in $\Gamma$, and the natural action of $\Gamma$ on $T$ given by $\gamma\cdot t \in \gamma tK\cap T$ for $\gamma \in \Gamma$ and $t\in T$ preserves normalized counting measure $\nu _T$ on $T$. Since $H$ is normal in $K$, the restriction to $K$ of the action $s$ leaves $Z^{H\backslash \Gamma}$ invariant and preserves the measure $\eta ^{H\backslash \Gamma}$ so that $\bm{s} = K\cc ^s (Z^{H\backslash \Gamma },\eta ^{H\backslash \Gamma})$ becomes the usual generalized Bernoulli shift. We let $\bm{b}$ denote the induced action $\bm{b}=\mbox{Ind}_K^\Gamma (\bm{s})$, which is the measure preserving action $\Gamma \cc ^b (Z^{H\backslash \Gamma} \times T , \, \eta ^{H\backslash \Gamma} \times \nu _T)$ given by
\[
\gamma ^b (f, t) = (\rho (\gamma , t)^{s}f , \gamma \cdot t)
\]
where $\rho :\Gamma \times T \ra K$ is the cocycle given by $\rho (\gamma ,t) = (\gamma \cdot t)^{-1}\gamma t$. The map $\pi : Z^{H\backslash \Gamma}\times T \ra Z^{\leq\backslash\Gamma}$ given by $\pi (f,t) = t^{s}f \in Z^{tHt^{-1}\backslash \Gamma}$ is an isomorphism of $\bm{b}$ with $\bm{s}_{\theta ,\eta}$. Indeed, $\pi$ is equivariant since
\begin{align*}
\pi (\gamma ^b (f,t)) 	&=\pi (\rho (\gamma ,t)^s f, \gamma \cdot t)
						= (\gamma \cdot t)^s\rho (\gamma ,t)^sf 
						= (\gamma t)^sf 
						= \gamma ^st^s f =\gamma ^s \pi (f,t)
\end{align*}
and $\pi$ is measure preserving since
\begin{align*}
\pi _* ( \eta ^{H\backslash\Gamma}\times \nu _T) &= \frac{1}{n}\sum _{t\in T}\pi _*(\eta ^{H\backslash \Gamma} \times \updelta _{t }) = \frac{1}{n}\sum _{t\in T} \eta ^{tHt^{-1}\backslash \Gamma}  = \eta ^{\theta\backslash \Gamma}.
\end{align*}
It is also clear that $\pi$ is injective since $t\mapsto tHt^{-1}$ is a bijection of $T$ with the conjugates of $H$.
\end{example}

\subsection{A sufficient condition for weak containment}

\begin{notation}
For sets $A$ and $B$ we let $A^{\subseteq B} = \bigcup _{C\subseteq B}A^C$. We identify $k\in \N$ with $k=\{ 0,1,\dots ,k-1 \}$. A \emph{partition} of $(X,\mu )$ will always mean a finite partition of $X$ into Borel sets. When $\mc{P}$ is a partition of $(X ,\mu )$ we will often identify elements of $\mc{P}$ with their equivalence class in $\mbox{MALG}_\mu$. We use the script letters $\mc{N}, \mc{O}, \mc{P}$, $\mc{Q}$, $\mc{R}$, $\mc{S}$ and $\mc{T}$ to denote partitions, and the printed letters $N$, $O$, $P$, $Q$, $R$, $S$ and $T$ respectively to denote their corresponding elements. If $\mc{P}$ and $\mc{Q}$ are two partitions of $(X,\mu )$ then we let $\mc{P}\vee \mc{Q} = \{ P\cap Q\csuchthat P\in \mc{P}, \ Q\in \mc{Q}\}$ denote their join. We write $\mc{P}\leq \mc{Q}$ if $\mc{Q}$ is a refinement of $\mc{P}$, i.e., if every $Q\in \mc{Q}$ is contained, modulo null sets, in some $P\in\mc{P}$.

Suppose $\Gamma \cc ^a (X,\mu )$ and $\mc{P} = \{ P_0,\dots ,P_{k-1} \}$ is a partition of $X$. If $J$ is a finite subset of $\Gamma$ and $\tau \in k^J$ then we define
\[
P^a_\tau = \bigcap _{\gamma \in J} \gamma ^a\cdot P_{\tau (\gamma )}.
\]
We will write $P_\tau$ when the action $a$ is understood. Note that $P_\varnothing = X$. We let $\Gamma$ act on the set $\bigcup \{ k^ J \csuchthat J\subseteq \Gamma \mbox{ is finite}\}$ by shift, i.e., $(\gamma \cdot \tau )(\delta ) = \tau (\gamma ^{-1}\delta )$. Then $\dom (\gamma \cdot \tau )=\gamma\dom (\tau )$.
\end{notation}

The following lemma establishes a sufficient condition for a measure preserving action $\bm{a}$ to be weakly contained in $\mc{B}$ which will be used in the proof of Theorem \ref{thm:randshift}.  This lemma is inspired by \cite[Lemma 5]{AW11}.

\begin{lemma}\label{lem:suff}
Suppose $\bm{a}=\Gamma \cc ^a (X,\mu )$ and $\mc{B}$ is a collection of measure preserving actions of $\Gamma$. Suppose $\mc{P} ^{(0)}\leq \mc{P} ^{(1)}\leq \cdots$ is a sequence of partitions of $X$ such that the smallest $\bm{a}$-invariant measure algebra containing $\bigcup _n \mc{P}^{(n)}$ is all of $\mbox{\emph{MALG}}_\mu$. Then $\bm{a}\prec \mc{B}$ if for any $n$, writing $\mc{P}^{(n)} =\mc{P}= \{ P_0,\dots ,P_{k-1} \}$, for all finite subsets $F\subseteq \Gamma$ and all $\delta >0$, there exists some $\Gamma \cc ^b (Y,\nu )=\bm{b}\in \mc{B}$ and a partition $\mc{Q}= \{ Q_0,\dots ,Q_{k-1} \}$ of $Y$ such that for all $\tau \in k^{\subseteq F}$, $| \mu (P_\tau ) - \nu (Q_\tau ) | <\delta$.
\end{lemma}

\begin{proof}
Suppose the condition is satisfied and let $A_1,\dots ,A_m\in\mbox{MALG}_\mu$, $F_0\subseteq \Gamma$ finite with $e\in F_0$, and $\epsilon >0$ be given. Let $e\in G_0\subseteq G_1\subseteq \cdots$ be an increasing exhaustive sequence of finite subsets of $\Gamma$, and let $G_n\cdot \mc{P}^{(n)} = \bigvee _{\gamma \in G_n} \gamma ^a \cdot \mc{P}^{(n)}$. Then $G_n\cdot \mc{P}^{(n)}$, $n=0,1,2,\dots$, is a sequence of finer and finer partitions of $X$ and the algebra generated by $\bigcup _n G_n\cdot \mc{P}^{(n)}$ is dense in $\mbox{MALG}_\mu$. There exists an $N$ and $D_1,\dots ,D_m$ in the algebra generated by $G_N\cdot \mc{P}^{(N)}$ such that $\mu (A_i \Delta D_i )< \tfrac{\epsilon}{4}$ for all $i\leq m$. Let $G=G_N$ and $\mc{P}=\mc{P}^{(N)} = \{ P_0,\dots , P_{k-1} \}$.

We can express each $D_i$ as a finite disjoint union of sets of the form $P_\sigma$, $\sigma \in k^T$, i.e., $D_i = \bigsqcup \{ P_\sigma \csuchthat \sigma \in I_i \}$ for some $I_i\subseteq k^G$. Applying the condition given by the lemma to $F=F_0G$ and $0<\delta < \tfrac{\epsilon}{2k^{|G|}}$ we obtain $\Gamma \cc ^b (Y,\nu )= \bm{b}\in \mc{B}$ and a partition $\mc{Q} = \{ Q_0,\dots ,Q_{k-1} \} \subseteq \mbox{MALG}_\nu$ such that for all $\tau \in k^{\subseteq F_0G}$, $| \mu (P_\tau ) - \nu (Q_{\tau}) | <\delta$. For $i\leq m$ we let $B_i = \bigsqcup \{ Q_\sigma \csuchthat \sigma \in I_i \}$. Note that for $\gamma \in F_0$ and $\sigma ,\sigma '\in k^G$ we have $\mbox{dom}(\gamma \cdot \sigma ) =\gamma G \subseteq F_0G$ and
\[
\gamma ^a P_\sigma \cap P_{\sigma '} = P_{\gamma \cdot \sigma } \cap P_{\sigma '} =
\begin{cases}
P_{\gamma\cdot \sigma \cup \sigma '} &\mbox{ if }\gamma \cdot \sigma\mbox{ and }\sigma ' \mbox{ are compatible} \\
\varnothing &\mbox{ otherwise.}
\end{cases}
\]
Similarly $\gamma ^b\cdot Q_\sigma \cap Q_{\sigma '}$ equals either $Q_{\gamma\cdot \sigma \cup \sigma '}$ or $\varnothing$ depending on whether or not $\gamma\cdot \sigma$ and $\sigma '$ are compatible partial functions. It then follows from our choice of $F$ that $|\mu (\gamma ^a P_\sigma \cap P_{\sigma '} ) - \nu (\gamma ^b Q_\sigma \cap Q_{\sigma '})| <\delta$ for all $\sigma ,\sigma ' \in k^G$.  We now have for $i,j\leq m$ and $\gamma \in F_0$ that
\begin{align*}
|\mu (\gamma ^a A_i \cap A_j) - \mu (\gamma ^bB_i\cap B_j)| &\leq \frac{\epsilon}{2}
			+ |\mu (\bigsqcup _{\substack{\sigma \in I_i , \\ \sigma ' \in I_j}}\gamma ^a P_\sigma \cap P_{\sigma '})
						- \nu (\bigsqcup _{\substack{\sigma \in I_i, \\ \sigma ' \in I_j }}\gamma ^bQ_\sigma \cap Q_{\sigma '}) | \\
&\leq \frac{\epsilon}{2} + |I_i||I_j|\delta < \epsilon .\qedhere
\end{align*}
\end{proof}

\subsection{Independent joinings over an IRS and the proof of Theorem \ref{thm:randshift}}\label{subsection:randshift} Let $\bm{a}=\Gamma \cc ^a (Y,\nu )$ be a non-atomic measure preserving action of $\Gamma$, and let $\theta = \mbox{type}(\bm{a})$. The stabilizer map $y\mapsto \Gamma _y$ factors $\bm{a}$ onto $\bm{\theta}$ and we let $\nu = \int _H \nu _H \, d\theta$ be the corresponding disintegration of $\nu$ over $\theta$. Fix a standard probability space $(Z,\eta )$ and let $\bm{s}_{\theta ,\eta} =\Gamma \cc ^s (Z^{\leq \backslash \Gamma},\eta ^{\theta \backslash\Gamma})$ be the $\theta$-random Bernoulli shift over $(Z,\eta )$. The map $f \mapsto H_f$ factors $\bm{s}_{\theta ,\eta}$ onto $\bm{\theta}$ and the corresponding disintegration is given by $\eta ^{\theta \backslash \Gamma}= \int _H \eta ^{H\backslash \Gamma}\, d\theta$. The \emph{relatively independent joining} of $\bm{s}_{\theta ,\eta}$ and $\bm{a}$ over $\bm{\theta}$ is then the action $\Gamma \cc ^{s\times a} (Z^{\leq \backslash \Gamma}\times Y , \eta ^{\theta\backslash \Gamma} \otimes _{\bm{\theta}} \nu )$ where
\begin{align*}
{\textstyle{\eta ^{\theta\backslash \Gamma}\otimes _{\bm{\theta}} \nu = \int _H (\eta ^{H\backslash \Gamma} \times \nu _H )\, d\theta
= \int _H (\eta ^{H\backslash \Gamma}\times \int _{\{ y:\Gamma _y = H\} }\updelta _y \, d\nu _H(y))\, d\theta = \int _y (\eta ^{\Gamma _y\backslash \Gamma}\times \updelta _y)\, d\nu .}}
\end{align*}
It is clear that $\eta ^{\theta\backslash \Gamma}\otimes _{\bm{\theta}} \nu$ concentrates on the set $Z^{\leq \backslash \Gamma}\otimes _a Y = \{ (f,y ) \csuchthat H_f = \Gamma _y \}$. We write $\bm{b}=\Gamma \cc ^{b} (X ,\mu )$ for $\Gamma \cc ^{s\times a} (Z^{\leq \backslash \Gamma}\otimes _a Y , \eta ^{\theta \backslash \Gamma} \otimes _{\bm{\theta}} \nu )$, so that $b=s\times a$, $X = Z^{\leq\backslash \Gamma}\otimes _a Y$, and
\[
\mu = \int _{y\in Y} \eta ^{\Gamma _y \backslash \Gamma} \times \updelta _y \, d\nu (y) .
\]
Theorem \ref{thm:randshift} then says that $\bm{b}$ is weakly equivalent to $\bm{a}$.

\begin{proof}[Proof of Theorem \ref{thm:randshift}]
It suffices to show that $\bm{b}\prec \bm{a}$. Let $\mc{N}^{(0)}\leq \mc{N}^{(1)}\leq \cdots$ and $\mc{R}^{(0)}\leq \mc{R}^{(1)}\leq \cdots$ be sequences of finite partitions of $Z$ and $Y$ respectively such that $\bigcup _n \mc{N}^{(n)}$ generates $\mbox{MALG}_\eta$ and $\bigcup _n \mc{R}^{(n)}$ generates $\mbox{MALG}_\nu$ (for example if $Z=Y= 2^\N$ then we can let $\mc{N}^{(n)}=\mc{R}^{(n)}$ consist of the rank $n$ basic clopen sets). For each $\gamma \in \Gamma$ let $\pi _\gamma : X \ra Z$ be the projection $\pi _\gamma (f,y) = f(\Gamma _y \gamma )$ and define the finite partitions $\mc{S}^{(0)}\leq \mc{S}^{(1)}\leq \cdots$ of $X$ by
\[
\mc{S}^{(n)} = \{ \pi _e ^{-1}(N)\csuchthat N \in \mc{N}^{(n)} \} .
\]
For $A\subseteq Y$ let $\tilde{A}\subseteq X$ denote the inverse image of $A$ under the projection map $(f,y)\mapsto y \in Y$ and define
\[
\tilde{\mc{R}}^{(n)} = \{ \tilde{R}\csuchthat R\in\mc{R}^{(n)} \} .
\]
Then the smallest $\bm{b}$-invariant measure algebra containing the partitions $\mc{P} ^{(n)} = \mc{S} ^{(n)}\vee \tilde{\mc{R}}^{(n)}$, $n\in \N$ of $X$ is all of $\mbox{MALG}_{\mu}$.  Fix $n$, define $\mc{N}=\mc{N}^{(n)} = \{ N_0,\dots , N_{d-1} \}$ and for $i<d$ define
\begin{align*}
S_i&=\pi _e^{-1}(N_i) \\ 
\alpha _i &=\mu (S_i)=\eta (N_i) 
\end{align*}
along with
\begin{align*}
\mc{S} &= \mc{S}^{(n)}=\{ S_0,\dots ,S_{d-1} \} \\
\mc{R} &= \mc{R}^{(n)} = \{ R_0,\dots ,R_{k-1} \} \\
\mc{P} &= \mc{P}^{(n)} = \{ P_{i,j} = S_i\cap \tilde{R}_j \csuchthat i< d , \, j< k \} .
\end{align*}
For $F\subseteq \Gamma$ finite we naturally identify $(d\times k)^{\subseteq F}$ with $\bigcup _{J\subseteq F} d^J \times k^J$. Under this identification, for $J\subseteq F$ and $(\tau, \sigma ) \in d^{J}\times k^{J}$ we have
\begin{align*}
P^b_{(\tau ,\sigma )} = \bigcap _{\gamma \in J} \gamma ^{s\times a}P_{\tau (\gamma ), \sigma (\gamma )} &= \bigcap _{\gamma \in J}\big( \gamma ^{s\times a}S_{\tau (\gamma )} \cap \gamma ^{s\times a} \tilde{R}_{\sigma (\gamma )} \big) \\
&= \big( \bigcap _{\gamma \in J}\gamma ^{s\times a} S_{\tau (\gamma )}\big) \cap \big( \bigcap _{\gamma \in J} \gamma ^{s\times a} \tilde{R}_{\sigma  (\gamma )}\big) = S^b_{\tau }\cap \tilde{R}^b_{\sigma } .
\end{align*}
By Lemma \ref{lem:suff}, to show that $\bm{b}\prec \bm{a}$ it suffices to show that for every $F\subseteq \Gamma$ finite, and $\epsilon >0$, there exists a partition $\mc{Q}=\{ Q_{i,j}\csuchthat i<d, \ j\leq k \}$ of $Y$ such that for all $J\subseteq F$, $(\tau , \sigma ) \in d^{J}\times k^{J}$
\[
|\mu (S_{\tau } \cap \tilde{R}_{\sigma }) - \nu (Q_{(\tau ,\sigma )}) | <\epsilon .
\]
Fix such an $F\subseteq \Gamma$ finite and $\epsilon >0$. We will proceed by finding a partition $\mc{T} = \{ T_0,\dots , T_{d-1} \}$ of $Y$ and then taking $Q_{i,j} = T_i\cap R_j$, in which case we will have $Q_{(\tau , \sigma )} = (\bigcap _{\gamma \in J} \gamma ^a T_{\tau (\gamma )} )\cap (\bigcap _{\gamma \in J} \gamma ^a R_{\sigma (\gamma )}) = T_{\tau }\cap R_{\sigma }$. We are therefore looking for a partition $\mc{T}$ so that
\begin{equation}\label{condition}
\forall (\tau ,\sigma ) \in (d\times k)^{\subseteq F} \  \, |\mu (S_\tau \cap \tilde{R}_\sigma ) - \nu (T_{\tau }\cap R_{\sigma })| <\epsilon .
\end{equation}
We first calculate the value of $\mu (S_\tau \cap \tilde{A})$ for $\tau \in d ^J$ ($J\subseteq F$) and $A\subseteq Y$. Let $\mc{E}_J$ denote the finite collection of all equivalence relations on the set $J$. For $E\in \mc{E}_J$ let us say that $\tau \in d^J$ \emph{respects} $E$, written $\tau \ll E$, if $\tau$ is constant on each $E$-equivalence class. For a subgroup $H\leq \Gamma$ let $E_J(H) \in \mc{E}_J$ denote the equivalence relation determined by $tE_J(H)s$ if and only if $Ht=Hs$ (if and only if $t^{-1}H = s^{-1}H$). We write $E_J(y)$ for $E_J(\Gamma _y)$. For each $E\in \mc{E}_J$ we fix a transversal $T_E \subseteq J$ for $E$. We then have
\begin{align}\label{eqn:value}
\mu &(S_\tau \cap \tilde{A})
= \int _{y\in A} \eta ^{\Gamma _y\backslash \Gamma}\big( \{ f \in Z^{\Gamma _y\backslash \Gamma} \csuchthat \forall t \in J \ (f(\Gamma _y t ) \in N_{\tau (t)}) \} \big) \, d\nu (y)  \nonumber \\
&= \sum _{\{ E\in \mc{E}_J \csuchthat \tau \ll E \} }
		\int _{\{ y\in A\csuchthat E_J(y)=E \} } \eta ^{\Gamma _y \backslash \Gamma} \big( \{ f \in Z^{\Gamma _y\backslash \Gamma} \csuchthat \forall t \in T_E \ (f(\Gamma _y t) \in N_{\tau (t)} ) \} \big) \, d\nu (y) \nonumber \\
&= \sum _{\{ E\in \mc{E}_J \csuchthat \tau \ll E \} } \nu (A\cap \{ y\csuchthat E_J(y)=E \} )\prod _{t \in T_E} \alpha _{\tau (t)}
\end{align}

We now proceed as in the proof of \cite[Theorem 1]{AW11}. Without loss of generality $Y$ is a compact metric space with compatible metric $d_Y\leq 1$. 
Fix some $\epsilon _0>0$ such that $\epsilon _0 ^{1/2}< \frac{\epsilon}{2(dk)^{|F|/2}2^{|F|+1}}$. For $\delta \geq 0$ define the sets
\begin{align*}
D_\delta &= \{ y\in Y\csuchthat \forall s, t\in F \ (t^{-1}y\neq s^{-1}y \Ra d_Y(t ^{-1} y,s ^{-1} y)>\delta ) \}  \\
E_\delta &= \{ (y,y')\in D_\delta \times D_\delta \csuchthat \forall s, t\in F \ (d_Y(s^{-1}y, t^{-1}y') >\delta ) \} .
\end{align*}
Then $\nu (D_0)= 1$ by definition, and $\nu ^2(E_0) = 1$ since $\nu$ is non-atomic. Thus there exists $\delta >0$ such that $\nu (D_\delta )> 1- \tfrac{\epsilon _0}{4|\mc{E}_F|}$ and $\nu ^2(E_\delta )> 1-\tfrac{\epsilon _0 }{4|\mc{E}_F|^2}$.

Fix a finite Borel partition $\{ O_m \csuchthat 1\leq m \leq M \}$ of $Y$ with $\mbox{diam}(O_m ) <\delta$ for each $m$.  For $y\in Y$ let $\alpha (y) = m$ if and only if $y\in O_m$.  Let $(\Omega , \P ) =  ( d ^M , \rho ^M )$ and let $Y_m (\omega ) = \omega (m)$, so that $\{ Y_m\csuchthat 1\leq m\leq M \}$ are i.i.d. random variables. For $\omega \in \Omega$ and $i=0,\dots ,d-1$ define
\[
T_i (\omega ) =  \{ y\in Y\csuchthat \omega (\alpha (y))= i \} .
\]
Then each $\omega \in \Omega$ defines the partition $\mc{T}(\omega ) = \{ T_0(\omega ),\dots ,T_{d-1}(\omega ) \}$ of $Y$. Let ${T}_i = \{ (\omega ,y)\csuchthat y\in T_i(\omega )\}$ and let ${T}_\tau = \{ (\omega ,y) \in \Omega \times Y \csuchthat y\in T_\tau (\omega ) = \bigcap _{t\in J} t^a\cdot (T_{\tau (t)}(\omega )) \}$, $\tau \in d^{\subseteq F}$.  We view $\mc{T}$ as a ``random partition'' of $Y$. We let $\Gamma$ act on $\Omega$ trivially so that, e.g., $\gamma\cdot (T_\tau (\omega )) = (\gamma \cdot T_\tau )(\omega )$, and for $B\subseteq \Omega \times Y$ and $y\in Y$ we let $B^y$ denote the section $B^y = \{ \omega\csuchthat (\omega ,y)\in B \}$. We show that $\mc{T}$ satisfies (\ref{condition}) with high probability.

Fix now some $A\subseteq Y$ and $\tau \in d^J$, $J\subseteq F$. Note that if $y\in Y$ and $\tau$ does not respect $E_J(y)$ then there exist $t,s\in J$ with $t^{-1}y = s^{-1}y$ and $\tau (t)\neq \tau (s)$, so that $(T_{\tau (t)})^{t^{-1}y} \cap (T_{\tau (s)})^{s^{-1}y} =\emptyset$ and thus $(T_{\tau})^y= \bigcap _{t\in J} (t\cdot T_{\tau (t)})^y = \bigcap _{t\in J}(T_{\tau (t)})^{t^{-1}y}=\emptyset$. It follows that the expected measure of ${T}_\tau (\omega )\cap A$ is
\begin{align}\label{expectation}
\E [\nu (T_\tau (\omega )\cap A)] &= \textstyle{\int _{A} \Big( \int _\Omega 1_{{T}_\tau}(\omega , y) \, d\P (\omega ) \Big) \, d\nu (y)} \nonumber  \\
& = \textstyle{\int _A \P ((T_\tau )^y)\, d\nu (y) \nonumber =\sum _{\{ E\in \mc{E}_J \csuchthat \tau \ll E \} } \int _{\{ y \in A \csuchthat E_J(y)=E \} } \P ((T_\tau )^y )\, d\nu (y)} \nonumber \\
&= \textstyle{\sum _{\{ E\in \mc{E}_J \csuchthat \tau \ll E \} } \Big( \int _{\{ y\in A\cap D_\delta \csuchthat E_J(y)=E \} } \P ((T_\tau )^y) \, d\nu \Big) + \int _{A\setminus D_\delta} \P ((T_\tau )^y) \, d\nu .}
\end{align}
Fix some $E\in \mc{E}_J$ with $\tau \ll E$ and some $y\in D_\delta$ with $E_J(y)=E$. For $t,s\in J$, if $t$ and $s$ are not $E$-related then $t^{-1}y\neq s^{-1}y$ and so $d_Y(t^{-1} y , s^{-1} y) > \delta$. It follows that $O_{\alpha (t^{-1} y)}\neq O_{\alpha (s^{-1} y)}$ since each $O_\alpha$ has diameter smaller than $\delta$.  So as $t$ ranges over $T_E$, the numbers $\alpha (t^{-1}y)$ are all distinct and the variables $Y_{\alpha (t^{-1}y)} : \omega \mapsto \omega (\alpha (t^{-1}y))$, $t\in T_E$, are therefore independent. We have $t^{-1}y \in T_{\tau (t)}(\omega )$ if and only if $\omega (\alpha (t^{-1}y )) =\tau (t)$, so the sets $(t \cdot T_{\tau (t)})^y = (T_{\tau (t)})^{t^{-1}y}$, $t\in T_E$, are all independent. If $tEs$ then as $\tau \ll E$ we have that $(T_{\tau (t)})^{t^{-1}y} = (T_{\tau (s)})^{s^{-1}y}$. It follows that
\begin{align} \label{fixedy}
& \textstyle{\P ((T_\tau )^y) = \P \big( \bigcap _{t\in J} (t\cdot T_{\tau (t)})^y \big) = \prod _{t \in T_E} \P ((T_{\tau (t)})^{t^{-1}y}) = \prod _{t\in T_E}\alpha _{\tau (t)} .}
\end{align}
Continuing the computation, the second integral in (\ref{expectation}) is no greater than $\nu (A\setminus D_\delta ) < \tfrac{\epsilon _0}{4}$ and $\nu (A\cap D_\delta \cap \{ y\csuchthat E_J(y)=E \} )$ is within $\tfrac{\epsilon _0}{4 |\mc{E}_F|}$ of $\nu (A\cap \{ y\csuchthat E_J(y)=E \} )$, so after summing over all $E\in \mc{E}_J$ we see that (\ref{expectation}) is within $\tfrac{\epsilon _0}{2}$ of (\ref{eqn:value}), i.e.,
\begin{equation}\label{expvalue}
\big| \E [\nu (T_\tau (\omega )\cap A) ] - \mu(S_\tau \cap \tilde{A}) \big| < \frac{\epsilon _0}{2} .
\end{equation}

Now we compute the second moment of $\nu (T_\tau (\omega ) \cap A)$.
\begin{align}\label{variance}
\textstyle{\E [ \nu (T_\tau (\omega ) \cap A) ^2 ]} &=\textstyle{\int _\Omega \big( \int _{y\in A} 1_{T_\tau }(\omega ,y) \, d\nu (y) \big)\big( \int _{y'\in A} 1_{B_\tau }(\omega ,y') \, d\nu (y') \big) \, d\P }\nonumber \\
&=\textstyle{\int _{(y,y')\in A\times A} \big( \int _\Omega 1_{T_\tau }(\omega ,y)1_{T_\tau }(\omega ,y')\, d\P \big) \, d\nu ^2}\nonumber \\
&=\textstyle{ \int _{(y,y')\in A\times A} \P ((T_\tau )^y \cap (T_\tau )^{y'}) \, d\nu ^2}
\end{align}
For $(y,y')\in E_\delta$, if $t, s \in J$ then $d_Y(t^{-1} y, s^{-1} y' ) > \delta$, so that $O_{\alpha (t^{-1}y)}$ and $O_{\alpha (s^{-1}x' )}$ are disjoint. It follows that the two events $\{ \omega \csuchthat \forall t\in J \, (Y_{\alpha (t^{-1}y)}(\omega ) = \tau (t) )\} =\bigcap _{t\in J} (T_{\tau (t)})^{t^{-1}y} = (T_\tau )^y$ and $\{ \omega \csuchthat \forall s\in J \, (Y_{\alpha (s^{-1}y')}(\omega )= \tau (s)) \} =\bigcap _{s\in J} (T_{\tau (s)})^{s^{-1}y} =(T_{\tau} )^{y'}$ are independent. We obtain that the part of (\ref{variance}) integrated over $(A\times A)\cap E_\delta$ is equal to
\begin{align}
&\textstyle{\int _{(y,y')\in (A\times A) \cap E_\delta} \P ((T_\tau )^y \cap (T_\tau )^{y'}) \, d\nu ^2  = \int _{(y,y')\in (A\times A) \cap E_\delta} \P ((T_\tau )^y) \P ((T_\tau )^{y'})\, d\nu ^2} \nonumber \\
= \textstyle{\sum}&_{\tau \ll E,E' \in \mc{E}_J} \textstyle{\nu ^2 ( (A\times A)\cap E_\delta \cap \{ (y,y'): E_J(y)=E,\, E_J(y')=E' \} )\prod _{t\in T_E}\alpha _{\tau (t)}\prod _{s\in T_{E'}}\alpha _{\tau (s)} }\nonumber
\end{align}
where we used the fact that $E_\delta \subseteq D_\delta \times D_\delta$ along with the known values from (\ref{expectation}) and (\ref{fixedy}). The part of (\ref{variance}) integrated over $(A\times A)\setminus E_\delta$ is no greater than $\tfrac{\epsilon _0}{4}$, and for each pair $E,E' \in \mc{E}_J$ with $\tau \ll E,E'$, the value of $\nu ^2( (A\times A)\cap E_\delta \cap \{ (y,y')\csuchthat E_J(y)=E,\, E_J(y')=E' \} )$ is within $\tfrac{\epsilon _0}{4|\mc{E}_F|^2}$ of $\nu (A\cap \{ y\csuchthat E_J(y)=E \} )\nu (A\cap \{ y'\csuchthat E_J(y')=E' \} )$. Summing over all such $E,E'\in \mc{E}_J$ we obtain that (\ref{variance}) is within $\tfrac{\epsilon _0}{2}$ of the square of (\ref{eqn:value}), i.e.,
\begin{equation}\label{secvalue}
\big| \E [ \nu (T_\tau (\omega ) \cap A) ^2 ] - \mu(S_\tau \cap \tilde{A})^2 \big| < \frac{\epsilon _0}{2} .
\end{equation}
From (\ref{expvalue}) and (\ref{secvalue}) it follows that the variance of $\nu (T_\tau (\omega ) \cap A)$ is no greater than $\epsilon _0$. By Chebyshev's inequality we then have
\begin{align}
\P &\textstyle{\big( |\nu (T_\tau (\omega ) \cap A) - \mu(S_\tau \cap \tilde{A}) | \geq \epsilon \big) \leq \P \big( |\nu (T_\tau (\omega )\cap A) - \E [\nu (T_\tau (\omega )\cap A)] | \geq \tfrac{\epsilon}{2}\big) } \nonumber \\
&\textstyle{\leq \P \big( |\nu (T_\tau (\omega )\cap A) - \E [\nu (T_\tau (\omega )\cap A)] | \geq (kd)^{|F|/2}2^{|F|+1}\epsilon _0^{1/2} \big) \leq \frac{1}{(kd)^{|F|}2^{2|F|+2}}} \nonumber
\end{align}
and since this is true for each $\tau \in d^{\subseteq F}$ and $|d^{\subseteq F}|\leq 2^{|F|}d^{|F|}$, we find that
\[
\P \big( \exists \tau \in d^{\subseteq F} \ ( |\nu (T_\tau (\omega ) \cap A) - \mu(S_\tau \cap \tilde{A})| \geq \epsilon ) \big) \leq \frac{1}{2^{|F|+2}k^{|F|}} .
\]
Since $A\subseteq Y$ was arbitrary, this is in particular true for each $A=R_\sigma$, $\sigma \in k^{\subseteq F}$, so that
\[
\P \big( \exists \tau \in d^{\subseteq F}, \, \sigma \in k^{\subseteq F} \ ( |\nu (T_\tau (\omega ) \cap R_\sigma) - \mu ( S_\tau \cap \tilde{R}_\sigma )| > \epsilon ) \big) \leq \frac{1}{4} .
\]
So taking any $\omega _0$ in the complement of the above set, we obtain a partition $\mc{T} = \mc{T}(\omega _0)$ satisfying (\ref{condition}).
\end{proof}

Theorem \ref{thm:randshift} shows that among all non-atomic weak equivalence classes of type $\theta$ there is a least, in the sense of weak containment. Namely $\bm{s}_{\theta ,\lambda}$ where $\lambda$ is Lebesgue measure on $[0,1]$. We note that there is also a greatest.

\begin{theorem}\label{thm:greatest}
Let $\theta \in \mbox{\emph{IRS}}(\Gamma )$. Then there exists a measure preserving action $\bm{a}_\theta$ of $\Gamma$ with $\mbox{type}(\bm{a}_\theta ) =\theta$ such that for all measure preserving actions $\bm{b}$ of $\Gamma$, if $\mbox{type}(\bm{b})=\theta$ then $\bm{b}\prec \bm{a}_\theta$.
\end{theorem}

\begin{proof}
Let $(Y,\nu )$ be a non-atomic standard probability space. If $\bm{b}$ is any measure preserving action of $\Gamma$ of type $\theta$ then $\bm{\iota}\times\bm{b}$ is also of type $\theta$, weakly contains $\bm{b}$, and is isomorphic to an element of $A(\Gamma ,Y,\nu )$. It thus suffices to show there is an action $\bm{a}_\theta$ of type $\theta$ that weakly contains every element in the set $\mc{A}_\theta = \{ \bm{a}\in A(\Gamma ,Y,\nu )\csuchthat \mbox{type}(\bm{a}) = \theta \}$. 

Let $\{ \bm{a}_n \} _{n\in \N}$ be a countable dense subset of $\mc{A}_\theta$. For each $n$ the stabilizer map $y\mapsto \mbox{stab}_{a_n}(y) = \{ \gamma \in \Gamma \csuchthat \gamma ^{a_n}y= y \}$ factors $\bm{a}_n$ onto $\bm{\theta}$. Let $\bm{a}_\theta$ denote the relatively independent joining of the actions $\bm{a}_0,\bm{a}_1,\bm{a}_2,\dots$ over the common factor $\bm{\theta}$, i.e., $\bm{a}_\theta = \Gamma \cc ^{\prod _n a_n}(Y ^\N , \nu _\theta )$ where the measure $\nu _\theta$ has each marginal equal to $\nu$ and concentrates on the set $\{ (y_0,y_1,y_2,\dots )\in Y^\N \csuchthat\forall n \, (\mbox{stab}_{a_n}(y_n)=\mbox{stab}_{a_0}(y_0) )\}$. Then for $\nu _\theta$-almost every $(y_0,y_1,\dots ) \in Y^\N$ we have $\mbox{stab}_{\prod _n a_n}((y_0,y_1,\dots )) = \mbox{stab}_{a_0}(y_0)$, from which it follows that $\mbox{type}(\bm{a}_\theta )= \theta$. Since $\bm{a}_n\sqsubseteq \bm{a}_\theta$ for all $n$ the set $\{ \bm{a} \in \mc{A}_\theta \csuchthat \bm{a}\prec \bm{a}_\theta\}$ is dense in $\mc{A}_\theta$ so by Lemma \ref{lem:10.1} $\bm{a}_\theta$ weakly contains every element of $\mc{A}_\theta$.
\end{proof}

\section{Non-classifiability} \label{section5}

\subsection{Non-classifiability by countable structures of $\cong$, $\cong ^w$, and $\cong ^\U$ on free weak equivalence classes}\label{subsection:freeunclass}

\begin{definition}Let $E$ and $F$ be equivalence relations on the standard Borel spaces $X$ and $Y$ respectively.
\begin{enumerate}
\item A \emph{homomorphism} from $E$ to $F$ is a map $\psi :X\ra Y$ such that $xEy \Ra \psi (x)F\psi (y)$.
\item A \emph{reduction} from $E$ to $F$ is a map $\psi :X\ra Y$ such that $xEy \IFF \psi (x)F\psi (y)$.
\item $E$ is said to \emph{admit classification by countable structures} if there exists a countable language $\mc{L}$ and a Borel reduction from $E$ to isomorphism $\cong _{\mc{L}}$ on $X_\mc{L}$, where $X_\mc{L}$ is the space of all $\mc{L}$-structures with universe $\N$.
\item Suppose that the space $X$ is Polish. 
We say that $E$ is \emph{generically $F$-ergodic} if for every Baire measurable homomorphism $\psi$ from $E$ to $F$, there exists some $y\in Y$ such that $\psi ^{-1}([y]_F)$ is comeager.
\end{enumerate}
\end{definition}

The proof of the following lemma is clear.

\begin{lemma}\label{lem:generg}
Let $F_1$ and $F_2$ be equivalence relations on the standard Borel spaces $Y_1$ and $Y_2$ respectively, and let $E$ be an equivalence relation on the Polish space $P$. Suppose that $E$ is generically $F_2$-ergodic and that there exists a Borel reduction from $F_1$ to $F_2$. Then $E$ is generically $F_1$-ergodic.
\end{lemma}

Since the orbit equivalence relation associated to a generically turbulent Polish group action is generically $\cong _{\mc{L}}$-ergodic for all countable languages $\mc{L}$ (\cite{Hj00}), Lemma \ref{lem:generg} immediately implies the following.

\begin{lemma} \label{cor:generg}
Let $\EuScript{G}$ be a Polish group and let $P$ be a generically turbulent Polish $\EuScript{G}$-space with corresponding orbit equivalence relation $E^P_{\EuScript{G}}$. Let $F$ be an equivalence relation on a standard Borel space $Y$ and suppose that $E^P_{\EuScript{G}}$ is not generically $F$-ergodic. Then $F$ does not admit classification by countable structures.
\end{lemma}


Let $\H$ be an infinite dimensional separable Hilbert space and let $\U (\H )$ denote the unitary group of $\H$ which is a Polish group under the strong operator topology. The group $\U (\H)$ acts on $\U (\H )^\Gamma$ by conjugation on each coordinate and we may view the space $\mbox{Rep} (\Gamma ,\H )$ of all unitary representations of $\Gamma$ on $\H$ as an invariant closed subspace of $\U (H) ^\Gamma$, so that it is a Polish $\U (\H )$-space. We call the corresponding orbit equivalence relation on $\mbox{Rep}(\Gamma ,\H )$ \emph{unitary conjugacy} and if $\pi _1$ and $\pi _2$ are in the same unitary conjugacy class then we say that $\pi _1$ and $\pi _2$ are \emph{unitarily conjugate} and write $\pi _1\cong \pi _2$. Let $\lambda _\Gamma : \Gamma \ra \U (\ell _2 (\Gamma ))$ denote the left regular representation of $\Gamma$ and let $\mbox{Rep}_\lambda (\Gamma ,\H )$ be the set of unitary representations of $\Gamma$ on $\U (\H )$ that are weakly contained in $\lambda _\Gamma$. Then $\mbox{Rep}_\lambda (\Gamma ,\H )$ is also a Polish $\U (\H )$ space, being an invariant closed subspace of $\mbox{Rep} (\Gamma ,\H )$.

The following lemma is proved in the same way as \cite[Lemma 2.4]{KLP10}, using that the reduced dual $\hat{\Gamma}_\lambda$, which may be identified with the spectrum of the reduced $C^*$-algebra $C^*_\lambda (\Gamma )$, contains no isolated points (\cite[3.2]{KLP10}).

\begin{lemma}\label{KLPlem2.4}
Let $\kappa$ be a unitary representation of $\Gamma$ on $\H$. Then the set $\{ \pi \in \mbox{\emph{Rep}}_\lambda (\Gamma ,\H ) \csuchthat \pi \perp \kappa \}$ is dense $G_\delta$ in $\mbox{\emph{Rep}}_\lambda (\Gamma ,\H )$.
\end{lemma}
We are now ready to prove Theorem \ref{thm:classif}.


\begin{proof}[Proof of Theorem \ref{thm:classif}]
Given a free action $\bm{a}_0 \in A(\Gamma , X,\mu )$, we let $[\bm{a}_0] = \{ \bm{b}\in A(\Gamma ,X,\mu )\csuchthat \bm{b}\sim \bm{a}_0 \}$ denote its weak equivalence class in $A(\Gamma ,X,\mu )$. 
Let $\H = \ell _2(\Gamma )$ and let $\bm{g} : \mbox{Rep} (\Gamma ,\H ) \ra A(\Gamma ,X,\mu )$ be the continuous map assigning to each $\pi \in \mbox{Rep}_\lambda (\Gamma ,\H )$ the corresponding Gaussian action $\bm{g}(\pi )\in A(\Gamma ,X,\mu )$ (see \cite[Appendix E]{Ke10}).  We have that $\bm{g}(\pi ) \prec \bm{g}(\infty \cdot \lambda _\Gamma ) \cong \bm{s}_\Gamma$ and so by Corollary \ref{thm:shift}, $\bm{a}_0\times \bm{g}(\pi ) \sim \bm{a}_0$. Fix some isomorphism $\varphi : X^2 \ra X$ of the measure spaces $(X^2 ,\mu ^2 )$ and $(X ,\mu )$ and denote by $\bm{b}\mapsto \varphi \cdot \bm{b}$ the corresponding homeomorphism of $A(\Gamma ,X^2,\mu ^2)$ with $A(\Gamma ,X,\mu )$. Let $\psi : \mbox{Rep}_\lambda (\Gamma ,\H ) \ra [\bm{a}_0]$ be the map $\pi \mapsto \varphi \cdot (\bm{a}_0 \times \bm{g}(\pi ))$. This is a continuous homomorphism from unitary conjugacy on $\mbox{Rep}_\lambda (\Gamma ,\H )$ to isomorphism on $[\bm{a}_0]$, and is therefore also a homomorphism to $\cong ^w$ and to $\cong ^\U$ on $[\bm{a}_0]$.

\begin{claim} The inverse image under $\psi$ of each unitary equivalence class in $[\bm{a}_0]$ is meager.  In particular the same is true for each isomorphism class and each weak isomorphism class. 
\end{claim}

\begin{proof}[Proof of Claim] Let $\bm{c}\in [\bm{a}_0]$. By Lemma \ref{KLPlem2.4} the set $\{ \pi \in \mbox{Rep}_\lambda (\Gamma ,\H ) \csuchthat \pi \perp \kappa ^{\bm{c}}_0 \}$ is comeager in $\mbox{Rep}_\lambda (\Gamma ,\H )$. If $\psi (\pi )\cong ^\U \bm{c}$ then $\pi \leq \kappa ^{\bm{g}(\pi )}_0 \leq \kappa ^{\bm{a}_0\times \bm{g}(\pi )}_0 \cong \kappa ^{\bm{c}}_0$, so that $\pi\not\perp \kappa ^{\bm{c}}_0 $. \qedhere[Claim]
\end{proof}

By \cite[3.3]{KLP10}, the conjugacy action of $\U (\H )$ on $\mbox{Rep}_\lambda (\Gamma , \H )$ is generically turbulent. The homomorphism $\psi$ witnesses that unitary conjugacy on $\mbox{Rep}_\lambda (\Gamma ,\H )$ is not generically $F|[\bm{a}_0]$-ergodic when $F$ is any of $\cong$, $\cong ^w$, or $\cong ^\U$. The theorem now follows from Lemma \ref{cor:generg}.
\end{proof}

\begin{remark} If the weak equivalence class $[\bm{a}_0]$ contains an ergodic (resp. weak mixing) action $\bm{b}_0$, then the action $\bm{b}_0 \times \bm{g}(\pi )$ is ergodic (resp. weak mixing) provided that the representation $\pi \in \mbox{Rep}_\lambda (\Gamma ,\H )$ is weak mixing. Since the weak mixing $\pi$ are dense $G_\delta$ in $\mbox{Rep}_\lambda (\Gamma ,\H )$ (\cite[3.6]{KLP10}) we conclude that isomorphism (and $\cong ^w$ and $\cong ^\U$) restricted to the ergodic (resp. weak mixing) elements of $[\bm{a}_0]$ does not admit classification by countable structures.

It also follows from the above arguments and \cite[2.2]{HK95} that the equivalence relation $E_0$ of eventual agreement on $2^\N$ is Borel reducible to $F|[\bm{a}_0]$ when $F$ is any of $\cong$, $\cong ^w$, or $\cong ^\U$ (and the same holds for $F| \{ \bm{b}\in [\bm{a}_0]\csuchthat \bm{b}\mbox{ is ergodic (resp. weak mixing)} \}$ when $[\bm{a}_0]$ contains ergodic (resp. weak mixing) elements). 
\end{remark}

\subsection{Extending Theorem \ref{thm:classif}}\label{subsec:nonfree} It would be interesting to see an extension of Theorem \ref{thm:classif} to weak equivalence classes of measure preserving actions that are not necessarily free. We outline here one possible generalization of the argument given in the proof of Theorem \ref{thm:classif} to measure preserving actions that almost surely have infinite orbits. 
Let $\bm{a}=\Gamma \cc ^a (X,\mu )$ be such an action, and let $\theta = \mbox{type}(\bm{a})$, so that $\theta$ concentrates on the infinite index subgroups of $\Gamma$. 
In place of unitary conjugacy on $\mbox{Rep}_\lambda (\Gamma ,\H )$ we work with the cohomology equivalence relation on a certain orbit closure in the Polish space $Z^1 (\bm{\theta}, \U (\H ) ) $ of unitary cocycles of $\bm{\theta}$, where $\H = \ell ^2(\N )$.  The cohomology equivalence relation on $Z^1(\bm{\theta},\U (\H ))$ is the orbit equivalence relation generated by the action of the Polish group $\widetilde{\U (\H )}=L(\mbox{Sub}(\Gamma ) ,\theta  ,\U (\H ))$ given by
\[
(f\cdot \alpha )(\gamma ,H ) = f(\gamma H\gamma ^{-1} )\alpha (\gamma ,H)f(H)^{-1} \in \U (\H )
\]
where $f\in \widetilde{\U (\H )}$, $\alpha \in Z^1(\bm{\theta} ,\U (\H ))$, $\gamma \in \Gamma$, and $H\leq \Gamma$ (see \cite[Chapter III]{Ke10}). In place of the left regular representation $\lambda$ of $\Gamma$ we use a cocycle $\lambda _\theta$ associated to $\theta$ defined as follows. Identify right cosets of the infinite index subgroups $H\leq \Gamma$ with natural numbers by fixing a Borel map $n:\mbox{Sub}(\Gamma ) \times \Gamma \ra \N$ such that for each infinite index $H\leq \Gamma$ the map $\gamma \mapsto n(H,\gamma )$ is a surjection onto $\N$ and satisfies $n(H, \gamma ) = n(H,\delta )$ if and only if $H\gamma = H\delta$. Let $\{ e_n\} _{n\in \N}$ be the standard orthonormal basis for $\ell ^2(\N )= \H$ and define $\lambda _\theta \in Z^1(\bm{\theta}, \U (\H ))$ by
\[
\lambda _\theta (\gamma ,H)(e _{n(H,\delta )}) = e _{n(\gamma H\gamma ^{-1} , \gamma \delta )}
\]
for all $\gamma \in \Gamma$ and $H\leq \Gamma$ of infinite index (recall that $\theta$-almost every $H$ is infinite index in $\Gamma$). Fix an isomorphism $T: \infty\cdot \H\ra \H$ and let $\sigma \in Z^1 (\bm{\theta} , \U (\H ))$ be the image of $\infty\cdot \lambda _{\theta}$ under $T$, i.e., $\sigma (\gamma ,H )= T \circ (\infty\cdot \lambda _\theta )(\gamma ,H)\circ T ^{-1}$.  Let $Z^1_\lambda (\bm{\theta}, \U (\H ))$ denote the orbit closure of $\sigma$ in $Z^1(\bm{\theta}, \U (\H ) )$.  Using the Gaussian map $U(\H ) \ra \mbox{Aut}(X,\mu )$ (see \cite[Appendix E]{Ke10} or \cite{BT-D11}), each $\alpha \in Z^1_\lambda (\bm{\theta}, \U (\H ))$ gives rise to a cocycle $g(\alpha ): \Gamma \times \mbox{Sub}(\Gamma ) \ra \mbox{Aut}(X,\mu )$ of $\bm{\theta}$ with values in the automorphism group $\mbox{Aut}(X,\mu )$ of a non-atomic probability space $(X,\mu )$. We obtain a skew product action $\bm{g}(\alpha ) = (X,\mu )\ltimes _{g(\alpha )}\bm{\theta}$ on the measure space $(Y,\nu ) = (X\times \mbox{Sub}(\Gamma ),\mu\times \theta )$, which is an extension of $\bm{\theta}$. The action $\bm{g}(\lambda _\theta )$ is isomorphic to $\bm{s}_{\theta , \eta }$ (where $\eta$ is non-atomic) and so the action $\bm{g}(\sigma )$ is isomorphic to $\bm{s}_{\theta , \eta ^\N}\cong \bm{s}_{\theta ,\eta}$ as well. Since $\alpha \in Z^1_\lambda (\bm{\theta}, \U (\H ))$ we have $\bm{g}(\alpha )\prec \bm{s}_{\theta ,\eta}$ and thus the relatively independent joining $\bm{g}(\alpha ) \otimes _{\bm{\theta}} \bm{a}$ is weakly equivalent to $\bm{a}$ by Theorem \ref{thm:randshift}. The map $\psi _\theta (\alpha ):= \varphi \cdot (\bm{g}(\alpha ) \otimes _{\bm{\theta}} \bm{a})$ is then a homomorphism from the cohomology equivalence relation on $Z^1_\lambda (\bm{\theta}, \U (\H ))$ to isomorphism on $[\bm{a}]$, where $\varphi :Y\times X \ra X$ is once again an isomorphism of measure spaces. The remaining ingredient that is needed is an analogue of the results from \cite{KLP10}.

\begin{question}
Let $\theta$ be an ergodic IRS of $\Gamma$ with infinite index. Is the action of $\widetilde{\U (\H )}$ on the space $Z^1_\lambda (\bm{\theta} , \U (\H ))$ generically turbulent? Is the preimage under $\psi _\theta$ of each $\cong ^{\U}$-class meager?
\end{question}

Two ergodic theoretic analogues of the space $\mbox{Rep}_{\lambda}(\Gamma , \H )$ are the spaces $A_0 (\Gamma ,X,\mu ) = \{ \bm{a}\in A(\Gamma ,X ,\mu )\csuchthat \bm{a}\prec \bm{s}_\Gamma \}$ and $A_1 (\Gamma ,X,\mu ) = \{ \bm{a}\in A(\Gamma ,X ,\mu )\csuchthat \bm{a}\prec _s \bm{s}_\Gamma \}$, where $(X,\mu )$ is non-atomic. 
When $\Gamma$ is amenable it follows from \cite{FW04} that these spaces both coincide with $A(\Gamma ,X,\mu )$ and the conjugacy action of $\mbox{Aut}(X,\mu )$ on $A(\Gamma ,X ,\mu  )$ is generically turbulent. For non-amenable $\Gamma$, the spaces $A_0(\Gamma ,X,\mu )$, $A_1(\Gamma ,X,\mu )$ and $A(\Gamma ,X,\mu )$ do not all coincide.

\begin{question}
Let $\Gamma$ be a non-amenable group. Is conjugacy on either of $A_0(\Gamma ,X ,\mu  )$ or $A_1(\Gamma ,X,\mu )$ generically turbulent?
\end{question}

For all non-amenable $\Gamma$ the set $A_0 (\Gamma ,X,\mu )$ is nowhere dense in $A_1(\Gamma ,X,\mu )$ (by Theorem \ref{thm:nonerg}), so these two spaces may behave quite differently, generically (indeed, every action in $A_0(\Gamma ,X,\mu )$ is ergodic, while the generic action in $A_1(\Gamma ,X,\mu )$ has continuous ergodic decomposition). The question of generic turbulence of conjugacy on $\mbox{ERG}(\Gamma ,X,\mu ) = \{ \bm{a}\in A(\Gamma ,X,\mu )\csuchthat \bm{a}\mbox{ is ergodic}\}$ is discussed in \cite[\S 5 and \S 12]{Ke10}.  
%
%
%
\section{Types and amenability}\label{sec:amenable}

As noted in Remark \ref{rem:amen1}, any two free measure preserving actions of an infinite amenable group $\Gamma$ are weakly equivalent. In this section we prove theorem \ref{thm:typeamen}, which extends this to actions that are not necessarily free.

\subsection{The space $\mbox{COS}(\Gamma )$}\label{subsection:COS}
%
%
%
Let $\mbox{COS}(\Gamma )$ be the space of all left cosets of all subgroups of $\Gamma$. Since $F\in \mbox{COS}(\Gamma ) \Leftrightarrow \forall \delta \in \Gamma \, (\delta \in F \, \Ra \delta ^{-1}F \in \mbox{Sub}(\Gamma ))$ it follows that $\mbox{COS}(\Gamma )$ is a closed subset of $2^\Gamma$. 
As every left coset of a subgroup $H\leq \Gamma$ is equal to a right coset of a conjugate of $H$ and vice versa, $\mbox{COS}(\Gamma )$ is also the space of all right cosets of subgroups of $\Gamma$ and we have the equality $\mbox{COS}(\Gamma )= \{ \gamma H \delta ^{-1} \csuchthat H\leq \Gamma , \ \gamma ,\delta \in \Gamma \} \subseteq 2^\Gamma$. We let $\ell$ denote the continuous action of $\Gamma$ on $\mbox{COS}(\Gamma )$ by left translation, $\gamma ^\ell\cdot (H\delta ) = \gamma H\delta$. 

\begin{lemma}\label{prop:Jfunc}
Let $\Gamma$ be a countable amenable group and let $\bm{a}=\Gamma \cc ^a (X,\mu )$ be a measure preserving action of $\Gamma$. Then for any finite $F\subseteq \Gamma$ and $\delta >0$ there exists a measurable map $J: X\ra \mbox{\emph{COS}}(\Gamma )$ such that
\[
\mu ( \{ x\in X \csuchthat \forall \gamma \in F \ J(\gamma ^a x) = \gamma ^{\ell}\cdot J(x) \} ) \geq 1-\delta
\]
and $J(x) \in \Gamma _x \backslash \Gamma$ for all $x$.
\end{lemma}

\begin{proof}
We note that this is a generalized version of \cite[Theorem 3.1]{BT-D11} which applies to the case in which $\bm{a}$ is free and which is an immediate consequence of the Rokhlin lemma for free actions of amenable groups. For the general case we use the Ornstein-Weiss Theorem \cite[Theorem 6]{OW80} which implies that the orbit equivalence relation $E_{a}$ generated by $\bm{a}$ is hyperfinite when restricted to an invariant co-null Borel set $X'\subseteq X$. We may assume without loss of generality that $X'=X$ and $E_a$ is hyperfinite. Then there exists an increasing sequence $E_0\subseteq E_1\subseteq \cdots$ of finite Borel sub-equivalence relations of $E_{a}$ such that $E_a = \bigcup _{n=0}^\infty E_n$. Let $F$ and $\delta >0$ be given and find $N \in \N$ large enough so that $\mu ( X_N ) >1-\delta$ where $X_N = \{ x\csuchthat \gamma ^ax \in [x]_{E_N} \mbox{ for all }\gamma \in F \}$. Fix a Borel selector $s:X\ra X$ for $E_N$, i.e., for all $x$, $xE_Ns(x)$ and $xE_Ny\Ra s(x)=s(y)$, and let $x\mapsto \gamma _x \in \Gamma$ be any Borel map such that $\gamma _x ^a\cdot s(x)=x$ for all $x\in X$. Define $J: X\ra \mbox{COS}(\Gamma )$ by $J(x)= \gamma _x\Gamma _{s(x)}$. Then $J(x)\in \Gamma _x\backslash \Gamma$ since $\Gamma _x = \Gamma _{\gamma _x^a\cdot s(x)} = \gamma _x\Gamma _{s(x)}\gamma _x$.  
For each $x\in X_N$ and $\gamma \in F$ we have $\gamma ^a x \in [x]_{E_N}$ so that $s(\gamma ^ax ) = s(x)$ and thus $(\gamma _{\gamma ^ax})^a \cdot s(x) =\gamma ^ax = (\gamma \gamma _x )^a\cdot s(x)$. It follows that
\[
J(\gamma ^ax) = \gamma _{\gamma ^ax}\Gamma _{s(x)} = \gamma \gamma _x \Gamma _{s(x)} = \gamma ^\ell \cdot J(x). \qedhere
\]
\end{proof}
%
%

\subsection{Proof of Theorem \ref{thm:typeamen}}\label{subsection:typeamen}

\begin{proof}[Proof of Theorem \ref{thm:typeamen}.(1)]
Since $\mbox{type}(\bm{a})$ is an invariant of stable weak equivalence (see Remark \ref{rem:stableconv}), it remains to show the following:
\begin{center}
$(*)$  If $\theta \in \mbox{IRS}(\Gamma )$ and $\bm{a}$ and $\bm{d}$ are measure preserving actions of $\Gamma$ both of type $\theta$, then $\bm{a}\sim _s\bm{d}$.
\end{center}

We first show that $(*)$ holds under the assumption that $\bm{a}$ and $\bm{d}$ are both ergodic. For this, by Theorem \ref{thm:randshift} it suffices to show that for any ergodic measure preserving action $\bm{a} = \Gamma \cc ^a (X,\mu )$ of $\Gamma$, if $\mbox{type}(\bm{a}) = \theta$ then $\bm{a} \prec \bm{s}_{\theta ,\eta}$ for some standard probability space $(Z , \eta )$.

We will define a measure preserving action $\bm{b}$ containing $\bm{\theta}$ as a factor, and show that the relatively independent joining $\bm{b} \otimes _{\bm{\theta}} \bm{s}_{\theta ,\eta }$ weakly contains $\bm{a}$ when $\eta$ is a standard non-atomic probability measure. Then we will be done once we show $\bm{b}\otimes _{\bm{\theta}} \bm{s}_{\theta ,\eta} \cong \bm{s}_{\theta ,\eta}$. 

Let $\mu = \int _H \mu _H \, d\theta$ be the disintegration of $\mu$ via $x\mapsto \mbox{stab}_a(x)$, and define the measure $\nu$ on the space $Y = \bigsqcup _{H\in \mbox{\tiny{Sub}}(\Gamma )} \{ f\in X^{H\backslash \Gamma} \csuchthat \mbox{stab}_a (f(H\delta )) = H \mbox{ for all }\delta \in \Gamma\} \subseteq X^{\leq\backslash \Gamma}$ by the equation $\nu = \int _H \mu _H ^{H\backslash \Gamma}\, d\theta$. Let $a ^{\leq \backslash \Gamma}$ be the action on $X^{\leq\backslash \Gamma}$ that is equal to $a^{H\backslash \Gamma}$ on $X^{H\backslash \Gamma}$. Then $a^{\leq \backslash \Gamma}$ commutes with the shift action $s$ on $X^{\leq\backslash \Gamma}$ and since $(\gamma ^s )_*(\gamma ^{a^{H\backslash \Gamma}})_*(\mu _H)^{H\backslash \Gamma} = \mu _{\gamma H\gamma ^{-1}}^{(\gamma H\gamma ^{-1})\backslash \Gamma}$ it follows from invariance of $\theta$ that the action $\gamma ^b = \gamma ^s \gamma ^{a^{\leq\backslash \Gamma}}$ preserves the measure $\nu$. We let $\bm{b} = \Gamma \cc ^b (Y,\nu )$. Then $\bm{\theta}$ is a factor of $\bm{b}$ via the map $f\mapsto H_f$. Let $(Z,\eta )$ be a standard non-atomic probability space, and let $\bm{b} \otimes _\theta \bm{s}_{\theta ,\eta }$ denote the relatively independent joining of $\bm{b}$ and $\bm{s}_{\theta ,\eta }$ over $\bm{\theta}$.

We now apply Lemma \ref{prop:Jfunc} to $\bm{s}_{\theta ,\eta}$. Given $F\subseteq \Gamma$ finite and $\epsilon >0$ there exists a measurable $J: Z^{\leq \backslash \Gamma}\ra \mbox{COS}(\Gamma )$ such that $\eta ^{{\theta \backslash \Gamma}} (Z_0) \geq 1-\epsilon$ where $Z_0 = \{ g\in Z^{\leq \backslash \Gamma} \csuchthat J(\gamma ^s \cdot g) = \gamma ^{\ell}\cdot J(g) \mbox{ for all }\gamma \in F \}$, and with $J(g) \in \Gamma _g\backslash \Gamma = H_g\backslash \Gamma$ for all $g\in Z^{\leq \backslash \Gamma}$. We let $\varphi : Y\times Z^{\leq \backslash \Gamma} \ra X$ be the map defined ($\nu \otimes _{\bm{\theta}} \eta ^{\leq \backslash \Gamma}$)-almost everywhere by $\varphi (f,g) = f(J(g))$. Then for all $g\in Z_0$ and $\gamma \in F$ we have $\varphi (\gamma ^{b \times s}(f,g)) = \gamma ^a ((\gamma ^s f)(J(\gamma ^s g))) = \gamma ^a (f(J(g))) = \gamma ^a \varphi ((f,g))$ and
\begin{align*}
\varphi _*(\nu \otimes _{\bm{\theta}} \eta ^{\leq \backslash \Gamma}) &= \int _H \int _g\int _f \updelta _{f(J(g))} \, d\mu _H ^{H\backslash \Gamma } \ d\eta ^{H\backslash \Gamma} \, d\theta \\
&= \int _H \sum _{t\in H\backslash \Gamma} \int _{\{ g\csuchthat J(g) =t \} } \mu _H \, d\eta ^{H\backslash \Gamma} \, d\theta = \int _H \mu _H \, d\theta = \mu .
\end{align*}
It then follows that $\bm{a}\prec \bm{b}\otimes _{\bm{\theta}} \bm{s}_\theta$ since for any measurable partition $A_0,\dots ,A_{k-1}\subseteq X$ of $X$, the sets $B_0 = \varphi ^{-1}(A_0),\dots ,B_{k-1}=\varphi ^{-1}(A_{k-1})$ form a measurable partition of $Y\times X^{\leq\backslash \Gamma}$ satisfying $|\mu (\gamma ^aA_i\cap A_j)- (\nu \otimes _{\bm{\theta}}\eta ^{\leq\backslash \Gamma})(\gamma ^{b\times s}B_i\cap B_j) | <\epsilon$ for all $\gamma \in F$.

By the Rokhlin skew-product theorem there exists a standard probability space $(Z_1,\eta _1)$ and an isomorphism $\Psi$ of $\bm{a}$ with a skew product action $\bm{d}=(Z_1,\eta _1)\ltimes \bm{\theta}$ on the space $(Z_1\times \mbox{Sub}(H), \eta _1 \times \theta )$. 
The isomorphism $\Psi$ is of the form $\Psi (x) = (\Psi _0(x), \Gamma _x )$ and so the restriction $\Psi _H$ of $\Psi _0$ to $X_H =\{ x\csuchthat \Gamma _x = H \}$ is an isomorphism of $(X_H, \mu _H)$ with $(Z_1, \eta _1)$ almost surely. We now define an isomorphism $\Phi : Y \ra Z_1^{\leq\backslash \Gamma}$ of $\bm{b}$ with $\bm{s}_{\theta ,\eta _1}$ by taking $H_{\Phi (f)} = H_f$ and $\Phi (f)(H\gamma ) = \Psi _{\gamma ^{-1}H\gamma}((\gamma ^{-1})^a(f(H\gamma )))$, where $H=H_f$. This is almost everywhere well-defined since $f(H\gamma )\in X_{H}$ almost surely, which ensures that $(\gamma ^{-1})^a(f(H\gamma ))$ is independent of our choice of representative for the coset $H\gamma$, and $(\gamma ^{-1})^a(f(H\gamma ))\in X_{\gamma ^{-1}H\gamma }$ so that we may apply $\Psi _{\gamma ^{-1}H\gamma}$. The map $\Phi$ is equivariant since if $H_f =H$ then $H_{\delta ^bf}= \delta H \delta ^{-1}$ and $\Phi (\delta ^b f)(\delta H\delta ^{-1} \gamma ) = \Psi _{\gamma ^{-1}\delta H\delta ^{-1}\gamma }((\gamma ^{-1})^a (\delta ^bf (\delta H\delta ^{-1}\gamma ))
= \Psi _{\gamma ^{-1}\delta H(\gamma ^{-1}\delta )^{-1}}((\gamma ^{-1}\delta )^a (f(H\delta ^{-1}\gamma ))) = \Phi (f)(H\delta ^{-1}\gamma )  = (\delta ^s\Phi (f))(\delta H \delta ^{-1}\gamma )$. Finally, $\Phi _* \nu = \eta _1^{\theta \backslash \Gamma}$ since
\begin{align*}
\Phi _*\nu =\int _H \Phi _* \mu _H ^{H\backslash \Gamma} \, d\theta &=\int _H \prod _{H\gamma \in H\backslash \Gamma} (\Psi _{\gamma ^{-1}H\gamma})_*(\gamma ^{-1})^a_*\mu _H \, d\theta \\
&= \int _H \prod _{H\gamma \in H\backslash \Gamma}(\Psi _{\gamma ^{-1}H\gamma })_*\mu _{\gamma ^{-1}H\gamma } \, d\theta = \int _H \eta _1^{H\backslash \Gamma} \, d\theta = \eta _1^{\theta \backslash \Gamma}
\end{align*}
and so $\bm{b}\cong \bm{s}_{\theta ,\eta _1}$. Since $H_f = H_{\Phi (f)}$, this extends to an isomorphism of $\bm{b}\otimes _{\bm{\theta}}\bm{s}_{\theta ,\mu }$ with $\bm{s}_{\theta ,\eta _1}\otimes _{\bm{\theta}}\bm{s}_{\theta ,\eta } \cong \bm{s}_{\theta ,\eta _1\times \eta} \cong \bm{s}_{\theta ,\eta}$, as was to be shown.

We next show that $(*)$ holds under the assumption that $\bm{\theta}$ is ergodic. Let $i\in \N \cup \{ \infty \}$ be the index of $\theta$. If $i$ is finite then the orbit of almost every $H\in \mbox{Sub}(\Gamma )$ is finite so by ergodicity of $\bm{\theta}$ there exists $H_0\leq \Gamma$ of index $i$ such that $\theta$ concentrates on the conjugates of $H_0$. Then for some spaces $(Z_1,\eta _1 )$ and $(Z_2,\eta _2)$ we have $\bm{a} \cong \bm{\iota}_{\eta _1} \times \bm{a}_{\Gamma /H_0}$ and $\bm{d}\cong \bm{\iota }_{\eta _2}\times \bm{a}_{\Gamma /H_0}$ where $\bm{a}_{\Gamma /H_0}$ denotes the action of $\Gamma$ on the left cosets of $H_0$ with normalized counting measure. Thus $\bm{a}\sim _s \bm{d}$. If $i= \infty$ then we let $\bm{a}= \int _{Z} \bm{a}_z \, d\eta$ and $\bm{d}= \int _W \bm{d}_w \, d\rho$ be the ergodic decompositions of $\bm{a}$ and $\bm{d}$ respectively. By Proposition \ref{prop:factor}, $\mbox{type}(\bm{a}_z)= \theta$ and $\mbox{type}(\bm{d}_w)=\theta$ almost surely, and $\bm{a}_z$ and $\bm{d}_w$ are non-atomic almost surely since $\theta$ is infinite index. Letting $\bm{b}$ be any non-atomic ergodic action of type $\theta$ the above case implies that $\bm{a}\sim _s\bm{b}\sim _s\bm{d}$.

Finally, we show that $(*)$ holds in general. Let $\theta = \int _{w\in W} \theta _w \, d\rho$ be the ergodic decomposition of $\theta$.  We then obtain corresponding decompositions $\bm{a} = \int _w \bm{a}_w \, d\rho$ and $\bm{d} = \int _w \bm{d}_w \, d\rho$ of $\bm{a}$ and $\bm{d}$ with $\mbox{type}(\bm{a}_w)=\theta _w = \mbox{type}(\bm{d}_w)$ almost surely. The above cases imply that $\bm{a}_w\sim _s \bm{d}_w$ almost surely. Theorem \ref{thm:more} then implies $\bm{a}\sim _s\bm{d}$.
\end{proof}

\begin{proof}[Proof of Theorem \ref{thm:typeamen}.(2)]
Let $\theta = \mbox{type}(\bm{a}) = \mbox{type}(\bm{b})$. If $\bm{\theta}$ is ergodic then by Proposition \ref{prop:factor} almost every ergodic component of $\bm{a}$ and $\bm{b}$ have type $\theta$ and so Theorem \ref{thm:typeamen} and Corollary \ref{cor:sweakerg} imply that $\bm{a}\sim \bm{\iota }_{\eta _1}\times \bm{d}$ and $\bm{b}\sim \bm{\iota}_{\eta _2}\times \bm{d}$ for some ergodic $\bm{d}$ of type $\theta$ and some spaces $(Z_1,\eta _1)$, $(Z_2 ,\eta _2)$. Since $\Gamma$ is amenable, $\bm{d}$ is not strongly ergodic, and since $\theta$ is infinite index, $\bm{d}$ is non-atomic, so by \cite[Theorem 3]{AW11} $\bm{d}\sim \bm{\iota}\times \bm{d}$ and thus $\bm{a}\sim \bm{b}$. The general case now follows by considering the ergodic decomposition of $\theta$.
\end{proof}
\appendix

\section{Ultraproducts of measure preserving actions}\label{app:ultra}

In this appendix we establish some properties of ultraproducts of measure spaces and actions.

\medskip

{\bf Notation.} We refer to \cite{CKT-D11} for background on ultraproducts of measure preserving actions and also \cite{ES08} for background on ultraproducts of measure spaces. Our notation has some changes from that of \cite{CKT-D11} and is as follows. Given a sequence $\bm{a}_n =\Gamma \cc ^{a_n} (X_n, \mu _n)$, $n\in \N$, of measure preserving actions of $\Gamma$ and a non-principal ultrafilter $\mc{U}$ on $\N$ we denote by $(\prod _n \bm{a}_n )/\mc{U} = \Gamma \cc ^{(\prod _n a_n )_{\mc{U}}} ((\prod _n X_n )/\mc{U} , (\prod _n \mu _n )/\mc{U})$, or simply $\bm{a}_{\mc{U}} = \Gamma \cc ^{a_\mc{U}} (X_{\mc{U}} , \mu _{\mc{U}})$ when there is no danger of confusion, the corresponding ultraproduct of the sequence $(\bm{a}_n )$. We let $[x_n]$ denote the equivalence class of the sequence $(x_n) \in \prod _n X_n$ in $X_{\mc{U}}$ and we let $[B_n]$ denote the subset of $X_{\mc{U}}$ determined by the sequence $(B_n) \in \prod _n \bm{B}(X_n)$ of Borel sets. When $x_n=x$ for all $n$ then we write $[x]$ for $[x_n]$ and when $B_n=B$ for all $n$ we write $[B]$ for $[B_n]$. Then $\bm{A}_{\mc{U}}=\bm{A}_\mc{U} (X_{\mc{U}}) = \{ [B_n]\csuchthat (B_n) \in \prod _n \bm{B}(X_n) \}$ is an algebra of subsets of $X_{\mc{U}}$ and $\mu _{\mc{U}}$ is the unique measure on the $\sigma$-algebra $\bm{B}_{\mc{U}} (X_{\mc{U}}) = \sigma (\bm{A}_{\mc{U}})$ whose value on $[A_n] \in \bm{A}_{\mc{U}}$ is $\mu _{\mc{U}} ([A_n]) = \lim _{n\ra \mc{U}} \mu _n (A_n)$. We note that every element of $\bm{B}_{\mc{U}}$ is within a $\mu _{\mc{U}}$-null set of an element of $\bm{A}_{\mc{U}}$.

\medskip

The following proposition deals with lifting measure disintegrations to ultraproducts.

\begin{proposition}\label{prop:ultrafactor}
Suppose that for each $n\in \N$ the Borel map $\pi _n : (Y _n ,\nu _n ) \ra (Z_n,\eta _n)$ factors $\bm{b}_n = \Gamma \cc ^b (Y_n ,\nu _n)$ onto $\bm{d}_n \cc ^d (Z _n ,\eta _n )$ and let $\nu _n = \int _{z\in Z_n} \nu ^n _z \, d\eta _n (z)$ be the disintegration of $\nu _n$ over $\eta _n$ with respect to $\pi _n$. Let $\bm{b}_{\mc{U}} = \Gamma \cc ^{b_{\mc{U}}}(Y_{\mc{U}}, \nu _{\mc{U}} )$ and $\bm{d}_{\mc{U}} = \Gamma \cc ^{d_{\mc{U}}}(Z_{\mc{U}} , \eta _{\mc{U}} )$ be the ultraproducts of the sequences $(\bm{b} _n )$ and $(\bm{d} _n )$ respectively. Then the map $\pi _{\mc{U}} : Y_{\mc{U}} \ra Z_{\mc{U}}$ given by $\pi _{\mc{U}}([y_n]) = [\pi _n(y_n)]$ factors $\bm{b}_{\mc{U}}$ onto $\bm{d}_{\mc{U}}$. If for $[z_n]\in Z_{\mc{U}}$ we let $\nu _{[z_n]} = (\prod _n \nu ^n_{z_n})/\mc{U}$ then
\begin{itemize}
\item[(I)] Each of the measures $\nu _{[z_n]}$ is a probability measure on $(Y_{\mc{U}} , \bm{B}_{\mc{U}} (Y_{\mc{U}}))$ and almost surely $\nu _{[z_n]}$ concentrates on $\pi _{\mc{U}} ^{-1}( [z_n] )$.
\item[(II)] For each $D \in \bm{B}_{\mc{U}}(Y_{\mc{U}})$ the map $(Z_\mc{U} , \bm{B}_{\mc{U}} (Z_{\mc{U}}))\ra ([0,1],\bm{B}([0,1]) )$ sending $[z_n] \mapsto \nu _{[z_n]}(D)$ is measurable and $\nu _{\mc{U}} (D) = \int _{[z_n] \in Z_{\mc{U}}} \, \nu _{[z_n]} (D) \, d\eta _{\mc{U}} ([z_n])$.
\item[(III)] If $[z_n]\mapsto \mu _{[z_n]}$ is another assignment satisfying \emph{(I)} and \emph{(II)} then for all $D\in \bm{B}_{\mc{U}} (Y_{\mc{U}})$ almost surely $\mu _{[z_n]} (D) = \nu _{[z_n]} (D)$.
\end{itemize}
Additionally, for almost all $[z_n]\in Z_{\mc{U}}$ and every $\gamma \in \Gamma$ we have $(\gamma ^{b_{\mc{U}}})_*\nu _{[z_n]} = \nu _{\gamma ^{d_{\mc{U}}}[z_n]}$.
\end{proposition}

\begin{proof}
It is clear that $\pi _{\mc{U}}$ factors $\bm{b}_{\mc{U}}$ onto $\bm{d}_{\mc{U}}$. Property (I) follows from the fact that for each $n$ and $z\in Z_n$, each $\nu ^n _z$ is a Borel probability measure on $Y_n$ and almost surely $\nu ^n_z$ concentrates on $\pi _n ^{-1}(\{ z \} )$. Now let $\bm{D}$ be the collection of all subsets of $Y_{\mc{U}}$ satisfying (II). 
Given $[A_n]\in \bm{A}_{\mc{U}}$ and $V\subseteq [0,1]$ open we have $\nu _{[z_n]}(A_n) \in V$ if and only if $[z_n] \in [ \{ z\csuchthat \nu ^n_z (A_n)\in V \} ]$, so that $[z_n]\mapsto \nu _{[z_n]}([A_n])$ is measurable. As in \cite[Lemma 2.2]{ES08} we have
\begin{align*}
\int _{[z_n]} \nu _{[z_n]} (A_n) \, d\eta _{\mc{U}} &= \int _{[z_n]}\lim _{n\ra \mc{U}} \nu ^n _{z_n} (A_n) \, d\eta _{\mc{U}} \\
& = \lim _{n\ra \mc{U}} \int _{z\in Z_n} \nu ^n _z (A_n ) \, d\eta _n = \lim _{n\ra \mc{U}}\nu _n (A_n) = \nu _{\mc{U}}([A_n])
\end{align*}
which shows that $[A_n] \in \bm{D}$. Thus $\bm{A}_{\mc{U}}\subseteq \bm{D}$, and it is clear that $\bm{D}$ is a monotone class so $\bm{B}_{\mc{U}}\subseteq \bm{D}$, which shows (II). Suppose now that $[z_n]\mapsto \mu _{[z_n]}$ satisfies (I) and (II). Then for each $[B_n]\in \bm{A}_{\mc{U}}(Z_{\mc{U}})$ and $D\in \bm{B}_{\mc{U}}(Y_{\mc{U}})$ we have $\int _{[B_n]} \mu _{[z_n]} (D) \, d\eta _{\mc{U}} = \nu _{\mc{U}}(D\cap \pi _{\mc{U}}^{-1}( [B_n]) ) = \int _{[B_n]} \nu _{[z_n]} (D)\, d\eta _{\mc{U}}$ so that $\mu _{[z_n]} (D) = \nu _{[z_n]}(D)$ almost surely, so that (III) holds.

For the last statement let $B_n\subseteq Z_n$ be an invariant $\eta _n$-conull set on which $(\gamma ^{b_n})_*\nu ^n_z = \nu ^n_{\gamma ^{d_n}z}$ for all $\gamma \in \Gamma$. Then for all $[z_n]$ in the $\eta _{\mc{U}}$-conull set $[B_n]\subseteq Z_{\mc{U}}$ we have for all $\gamma \in \Gamma$ and $[A_n]\in \bm{A}_{\mc{U}}(Y_{\mc{U}})$ that $(\gamma ^{d_{\mc{U}}})_* \nu _{[z_n]}(A_n) = \lim _{n\ra \mc{U}} (\gamma ^{d_n})_*\nu _{z_n}^n (A_n) = \lim _{n\ra \mc{U}} \nu ^n _{\gamma ^{d_n}z}(A_n) = \nu _{\gamma ^{d_{\mc{U}}}[z_n]}([A_n])$ so that $(\gamma ^{d_\mc{U}})_*\nu _{[z_n]}= \nu _{\gamma ^{d_{\mc{U}}}[z_n]}$.
\end{proof}

The next proposition describes the ultrapower of a standard probability space with atoms.

\begin{proposition}\label{prop:measalgcomp} Let $(Z,\eta )$ be a standard probability space and let $A\subseteq Z$ be the set of atoms of $(Z ,\eta )$.
\begin{enumerate}
\item If $(Z, \eta )$ is discrete then $(\mbox{\emph{MALG}}_\eta , d_\eta )$ is a compact metric space homeomorphic to $2^A$ with the product topology, and the map $I_{\mc{U}} : \mbox{\emph{MALG}}_{\eta _{\mc{U}}}\ra \mbox{\emph{MALG}}_{\eta}$ given by $I_{\mc{U}} ( [ B_n ] ) = \lim _{n\ra \mc{U}} B_n = \{ z\in A \csuchthat \{ n \csuchthat \ z\in B_n\} \in \mc{U} \}$ is a measure algebra isomorphism.
\item In general $[A] = \{ [z]\csuchthat z\in A \}\subseteq Z_{\mc{U}}$ is the set of all atoms of $\eta _{\mc{U}}$ and the restriction $\eta |A$ of $\eta$ to $A$ is isomorphic as a measure space to the restriction $\eta _{\mc{U}}|[A]$ of $\eta _{\mc{U}}$ to $[A]$ via the map $z\mapsto [z]$. Under this isomorphism, letting $C=Z\setminus A$, we may identify $(Z_{\mc{U}} , \eta _{\mc{U}})$ with $([C]\sqcup A, (\eta |C)_{\mc{U}} + \eta |A)$.
\end{enumerate}
\end{proposition}

\begin{proof}
First suppose that $(Z,\eta )$ is discrete. Without loss of generality we may assume $Z=A$. As sets we may identify $\mbox{MALG} _{\eta}$ with $2^{A}$. Let $B_0,B_1,\dots$ be a sequence in $2^A$ converging in the product topology to some set $B\in 2^A$. Given $\epsilon >0$ let $F\subseteq A$ be a finite set such that $\eta (A\setminus F) <\epsilon$. For all large enough $n$, $B_n$ and $B$ agree on $F$, so that $\eta (B_n\Delta B)< \eta (A\setminus F)<\epsilon$ and thus $d_\eta (B_n, B) \ra 0$. This shows that the map $2^{A}\ra \mbox{MALG}_\eta$ is a continuous bijection from the compact Hausdorff space $2^{A}$ (with the product topology) to $(\mbox{MALG}_\eta , d_\eta )$, so it is a homeomorphism. It is clear that the map $\varphi$ taking $B\subseteq A$ to $[B] \subseteq [A]$ is an isometric embedding of $\mbox{MALG}_{\eta}$ to $\mbox{MALG}_{\eta _{\mc{U}}}$ that preserves all Boolean operations. If now $[B_n]\subseteq [A]$ and $\lim _{n\ra \mc{U}}B_n = B$ then $d_{\eta _{\mc{U}}} ( [B_n] , [B] ) = \lim _{n\ra \mc{U}} d_\eta ( B_n ,B) = 0$ so that $[B_n] =[B]$ and thus $\varphi ^{-1} = I_{\mc{U}}$ which completes the proof of (1). Part (2) follows since $(Z_{\mc{U}}, \eta _{\mc{U}})$ decomposes as $([C] \sqcup [A], (\eta |C)_{\mc{U}} + (\eta |A)_{\mc{U}})$ and part (1) shows that $([A] , (\eta |A)_{\mc{U}}) \cong (A, \eta )$.
\end{proof}

\begin{theorem}\label{thm:genmeasalg}
Let $\bm{a}_0,\bm{a}_1,\dots$ be a sequence of measure preserving actions of $\Gamma$ on the standard probability space $(X,\mu )$ and let $\bm{a}_{\mc{U}} = \Gamma \cc ^{a_{\mc{U}}} (X_{\mc{U}} ,\mu_{\mc{U}})$ be their ultraproduct. Let $\bm{M}_0\subseteq \mbox{\emph{MALG}}_{\mu _{\mc{U}}}$ be any subset such that $( \bm{M}_0 , d_{\mu _{\mc{U}}}|\bm{M}_0 )$ is separable. Then there exists an invariant measure sub-algebra $\bm{M}$ of $\mbox{\emph{MALG}}_{\mu _{\mc{U}}}$ containing $\bm{M}_0$ that is isomorphic as a measure algebra to $\mbox{\emph{MALG}}_\mu$.
\end{theorem}

\begin{proof}
Let $A\subseteq X$ be the collection of atoms of $X$ and let $C=X\setminus A$. 
By Proposition \ref{prop:measalgcomp}.(2), $[A]\subseteq X$ is the discrete part of $\mu _{\mc{U}}$ and $x\mapsto [x]$ is an isomorphism $\mu |A \cong \mu _{\mc{U}} |[A]$. Define a function $S_{\mc{U}} : \mbox{MALG}_{\mu _{\mc{U}}}\ra \mbox{MALG}_{\mu _{\mc{U}}}$ first on subsets $D\subseteq [C]$ by taking $S_{\mc{U}}(D)$ to be any subset of $D$ satisfying $\mu _{\mc{U}}(S_{\mc{U}}(D)) = \tfrac{1}{2}\mu _{\mc{U}} (D)$, and then extending this to all of $\mbox{MALG}_{\mu _{\mc{U}}}$ by taking $S_{\mc{U}}(D) = S_{\mc{U}}(D\cap [C]) \sqcup (D\cap [A])$. Fix a countable dense subset $\bm{M}_1$ of $\bm{M}_0$ and let $\bm{B}_0\subseteq \mbox{MALG}_{\mu _{\mc{U}}}$ be a countable Boolean algebra containing $\bm{M}_1 \cup \{ \{ [x] \} \csuchthat x\in A \}$ and closed under the functions $S_{\mc{U}}$ and $\gamma ^{a_{\mc{U}}}$ for all $\gamma \in \Gamma$. Then the $\sigma$-algebra $\bm{M} = \sigma (\bm{B}_0 )$ equipped with $\mu _{\mc{U}}$ is an invariant countably generated measure sub-algebra of $\mbox{MALG}_{\mu _{\mc{U}}}$ containing $\bm{M}_0$. Since $\bm{B}_0$ is closed under $S_{\mc{U}}$, the atoms of $\bm{B}_0$, and hence also those of $\bm{M}$, must be contained in $[A]$, and as $\bm{M}$ contains $\{ [B]\csuchthat B\subseteq A \}$, the descrete part of $\bm{M}$ is isomorphic to the discrete part of $\mbox{MALG}_\mu$. It follows that $\bm{M}\cong \mbox{MALG}_\mu$. 
\end{proof}

\begin{proposition}\label{prop:weakcontspace}
Let $\bm{a}= \Gamma \cc ^a (X,\mu )$ and $\bm{b} = \Gamma \cc ^b (Y,\nu )$ be measure preserving actions of $\Gamma$. If $\bm{a}$ is weakly contained in $\bm{b}$ then then the measure space $(X,\mu )$ is a quotient of the measure space $(Y,\nu )$. If $\bm{a}$ and $\bm{b}$ are weakly equivalent then $(X,\mu )$ is isomorphic to $(Y,\nu )$. In particular, the identity actions $\bm{\iota}_{\eta _1}$ and $\bm{\iota} _{\eta _2}$ are weakly equivalent if and only if $(Z_1, \eta _1)$ and $(Z_2, \eta _2)$ are isomorphic measure spaces.
\end{proposition}

\begin{proof}
Suppose first that $\bm{a}\prec \bm{b}$. Let $\phi : X \ra K=2^\N$ be any Borel isomorphism and let $\lambda = (\Phi ^{\phi ,a})_*\mu$. Then $\bm{a} \cong \Gamma \cc ^s (K^\Gamma ,\lambda )$ and as $\bm{a}\prec \bm{b}$ there exists $\lambda _n = (\Phi ^{\phi _n , b})_*\nu \in E(\bm{b}, K)$ with $\lambda _n \ra \lambda$. By Proposition \ref{prop:ultrafunct} $\Gamma \cc ^s (K^\Gamma ,\lambda )$ is a factor of the ultrapower $\bm{b}_{\mc{U}}$ of $\bm{b}$ via $\Phi ^{\phi , b_{\mc{U}}}$ where $\phi$ is the ultralimit of the $\phi _n$. Thus $\bm{a}$ is also a factor of $\bm{b}_{\mc{U}}$ so by Theorem \ref{thm:genmeasalg} this implies $(X,\mu )$ is a factor of $(Y,\nu )$.

Now suppose that $\bm{a}$ and $\bm{b}$ are weakly equivalent. Then the measure spaces $(X,\mu )$ and $(Y,\nu )$ are factors of each other, say $\pi : (Y,\nu ) \ra (X,\mu )$ and $\varphi : (X,\mu )\ra (Y,\nu )$. Let $A\subseteq X$ be the set of atoms of $X$ and let $B\subseteq Y$ be the set of atoms of $Y$. If $\mu (A)=0$ then we are done since this implies both $(X,\mu )$ and $(Y,\nu )$ are non-atomic. So suppose that $\mu (A)>0$. It is clear that $A\subseteq \varphi ^{-1}(B)$ and $B\subseteq \pi ^{-1}(A)$, hence $\mu (A) = \nu (B)$. Additionally, $\mu (\varphi ^{-1}(B)\setminus A) = 0$, otherwise $\nu (B) =\mu (\varphi ^{-1}(B)) > \mu (A)$.  Similarly $\nu (\pi ^{-1}(A)\setminus B) =0$. Thus $\varphi ^{-1}: (\mbox{MALG}_{\nu _B}, d_{\nu _B}) \ra (\mbox{MALG}_{\mu _A}, d_{\mu _A})$ and $\pi ^{-1} : (\mbox{MALG}_{\mu _A},d_{\mu _A})\ra (\mbox{MALG}_{\nu _B},d_{\nu _B})$ are isometric embeddings of compact metric spaces (Proposition \ref{prop:measalgcomp}), so it follows that both $\pi ^{-1}$ and $\varphi ^{-1}$ are in fact isometric isomorphisms. Since these maps are also Boolean algebra homomorphisms it follows that both are measure algebra isomorphisms. This shows that the discrete parts of $(X,\mu )$ and $(Y,\nu )$ are isomorphic, from which it follows that $(X,\mu )$ and $(Y,\nu )$ are isomorphic.
\end{proof}

\section{Stable weak containment}\label{app:stable}

In this appendix we establish some basic properties of stable weak containment of measure preserving actions. Our development mirrors our development of weak containment of measure preserving actions.

\begin{definition}\label{def:swc}
Let $\mc{A}$ and $\mc{B}$ be two sets of measure preserving actions of $\Gamma$. We say that $\mc{A}$ is \emph{stably weakly contained in} $\mc{B}$, written $\mc{A}\prec _s \mc{B}$ if for every $\Gamma \cc ^a (X,\mu ) = \bm{a} \in \mc{A}$, for any Borel partition $A_0,\dots ,A_{k-1}$ of $X$, $F\subseteq \Gamma$ finite, and $\epsilon >0$, there exist nonnegative reals $\alpha _0,\dots ,\alpha _{m-1}$ with $\sum _{i<m}\alpha _i =1$ along with actions $\Gamma \cc ^{b_i}(Y_i,\nu _i) = \bm{b}_i \in \mc{B}$, $i<m$, and a Borel partition $B_0,\dots ,B_{k-1}$ of $\sum _{i<m}Y_i$ such that
\[
|\mu (\gamma ^aA_i\cap A_j) - (\textstyle{\sum _{i<m}}\alpha _i\tilde{\nu} _i)(\gamma ^{\sum _{i<m}b_i}B_i\cap B_j) |<\epsilon
\]
for all $i,j< k$ and $\gamma \in F$. (See \S\ref{subsection:convexity} for notation.)
\end{definition}

The relation $\prec _s$ is a reflexive and transitive relation on sets of measure preserving actions. We call $\mc{A}$ and $\mc{B}$ \emph{stably weakly equivalent}, written $\mc{A}\sim _s \mc{B}$, if both $\mc{A}\prec _s \mc{B}$ and $\mc{B}\prec _s \mc{A}$. We write $\bm{a}\prec _s \mc{B}$, $\mc{A}\prec _s\bm{b}$, and $\bm{a}\prec _s \bm{b}$ for $\{ \bm{a}\} \prec _s \mc{B}$, $\mc{A}\prec _s \{ \bm{b} \}$ and $\{ \bm{a}\} \prec _s \{\bm{b} \}$ respectively, and similarly with $\sim _s$ in place of $\prec _s$.

It is clear that $\bm{a}\prec _s\bm{b}$ if and only if $\bm{a}\prec \{ \bm{\iota}_{\eta _{\bm{\alpha}}}\times \bm{b}\csuchthat \bm{\alpha} = ( \alpha _0,\dots ,\alpha _{m-1}) \in [0,1]^m, \ \sum _{i<m}\alpha _i = 1 , \ m\in \N \}$, so by Lemma \ref{lem:setcon} we have $\bm{a}\prec _s \bm{b}$ if and only if $\bm{a}\prec \bm{\iota}\times \bm{b}$ if and only if $\bm{\iota}\times \bm{a} \prec \bm{\iota}\times \bm{b}$. From this point of view Theorem \ref{thm:main} says that if $\bm{a}$ is ergodic then $\bm{a}\prec _s \bm{b}$ if and only if $\bm{a}\prec \bm{b}$. Theorem \ref{thm:conv} implies that $\bm{a}\prec _s \bm{b}$ if and only if $E(\bm{a},K)\subseteq \ol{\mbox{co}}{E(\bm{b},K)}$ for every compact Polish space $K$, and $\bm{a}\sim _s \bm{b}$ if and only if $\ol{\mbox{co}}E(\bm{a},K)=\ol{\mbox{co}}E(\bm{b},K)$ for every compact Polish space $K$. More generally, we have the following analogue of Proposition \ref{prop:AWgen} which can be proved directly by using the same methods.

\begin{proposition}\label{prop:stab}
Let $\mc{A}$ and $\mc{B}$ be sets of measure preserving actions of $\Gamma$. Then the following are equivalent
\begin{enumerate}
\item $\mc{A} \prec _s \mc{B}$;
\item $\bigcup _{\bm{d}\in \mc{A}} E(\bm{d} , K) \subseteq \ol{\mbox{\emph{co}}} (\bigcup _{\bm{b}\in \mc{B}}E(\bm{b},K))$ for every finite $K$;
\item $\bigcup _{\bm{d}\in \mc{A}} E(\bm{d} , K) \subseteq \ol{\mbox{\emph{co}}}( \bigcup _{\bm{b}\in \mc{B}}E(\bm{b},K))$ for every compact Polish $K$;
\item $\bigcup _{\bm{d}\in \mc{A}} E(\bm{d} , 2^\N ) \subseteq \ol{\mbox{\emph{co}}} ( \bigcup _{\bm{b}\in \mc{B}}E(\bm{b},2 ^\N ))$.
\end{enumerate}
\end{proposition}

$\ $\\
\noindent Department of Mathematics \\
California Institute of Technology \\
Pasadena, CA 91125 \\
\texttt{rtuckerd@caltech.edu}

\end{document}